\documentclass[11pt,pdflatex]{amsart}
\usepackage[usenames]{color}
\usepackage{color}
\usepackage{amssymb}
\usepackage{graphicx, epsfig}
\usepackage{latexsym, amsfonts, amscd, amsmath}
\usepackage{mathrsfs}
\usepackage{lscape}
\makeindex 
\input xy
\xyoption{all}

\definecolor{orange}{rgb}{1,0.5,0}
\definecolor{Indigo}{rgb}{0.2,0.1,0.7}
\definecolor{Violet}{rgb}{0.5,0.1,0.7}

\newtheorem{thm}{Theorem}[subsection]
\newtheorem{prop}[thm]{Proposition}
\newtheorem{lem}[thm]{Lemma}
\newtheorem{cor}[thm]{Corollary}

\theoremstyle{definition}

\newtheorem{exa}[thm]{Example}

\theoremstyle{remark}
\newtheorem{rmk}[thm]{Remark}

\newenvironment{mat}{\left(\begin{smallmatrix}
\end{smallmatrix}\right)}{enddef}

\numberwithin{equation}{subsection}
\numberwithin{figure}{subsection} \numberwithin{table}{subsection}

\newcommand{\Aut}{{\operatorname{Aut}}}

\newcommand{\diag}{{\operatorname{diag}}}

\newcommand{\End}{{\operatorname{End}}}
\newcommand{\Fr}{{\operatorname{Fr }}}
\newcommand{\Hom}{{\operatorname{Hom}}}

\newcommand{\Ind}{{\operatorname{Ind }}}
\newcommand{\Jac}{{\operatorname{Jac }}}
\newcommand{\Ker}{{\operatorname{Ker}}}

\newcommand{\Norm}{{\operatorname{Norm }}}

\newcommand{\Span}{{\operatorname{Span }}}
\newcommand{\Spec}{{\operatorname{Spec }}}

\newcommand{\Tr}{{\operatorname{Tr }}}
\newcommand{\val}{{\operatorname{val}}}

\newcommand{\SL}{{\operatorname{SL }}}
\newcommand{\GL}{{\operatorname{GL}}}
\newcommand{\Gal}{{\operatorname{Gal}}}
\newcommand{\Symp}{{\operatorname{Sp }}}

\newcommand{\PGL}{{\operatorname{PGL}}}

\newcommand{\gera}{{\frak{a}}}
\newcommand{\gerb}{{\frak{b}}}

\newcommand{\gerh}{{\frak{h}}}
\newcommand{\geri}{{\frak{i}}}
\newcommand{\gerj}{{\frak{j}}}

\newcommand{\germ}{{\frak{m}}}

\newcommand{\gerp}{{\frak{p}}}
\newcommand{\gerq}{{\frak{q}}}

\newcommand{\gerH}{{\frak{H}}}
\newcommand{\gerI}{{\frak{I}}}

\newcommand{\gerP}{{\frak{P}}}

\newcommand{\gerR}{{\frak{R}}}

\newcommand{\uomega}{{\underline{\omega}}}

\newcommand{\calA}{{\mathcal{A}}}

\newcommand{\calO}{{\mathcal{O}}}

\def\AA{\mathbb{A}}
\def\BB{\mathbb{B}}
\def\CC{\mathbb{C}}
\def\DD{\mathbb{D}}
\def\EE{\mathbb{E}}
\def\FF{\mathbb{F}}
\def\GG{\mathbb{G}}
\def\HH{\mathbb{H}}

\def\PP{\mathbb{P}}
\def\QQ{\mathbb{Q}}
\def\RR{\mathbb{R}}

\def\WW{\mathbb{W}}

\def\ZZ{\mathbb{Z}}

\newcommand{\scrA}{{\mathscr{A}}}
\newcommand{\scrB}{{\mathscr{B}}}

\newcommand{\scrH}{{\mathscr{H}}}

\newcommand{\scrM}{{\mathscr{M}}}

\newcommand{\id}{{\noindent}}
\newcommand{\normal}{{\vartriangleleft}}

\newcommand{\ilim}{{\underset{\longleftarrow}{\lim} \; \; }}
\newcommand{\arr}{{\; \rightarrow \;}}
\newcommand{\Arr}{{\; \longrightarrow \;}}
\newcommand{\surjects}{{\; \twoheadrightarrow \;}}
\newcommand{\injects}{{\; \hookrightarrow \;}}

\newcommand{\ok}{{\mathcal{O}_K}}

\newcommand{\fpbar}{\overline{\mathbb{F}}_p}
\newcommand{\Spf}{\operatorname{Spf }}

\newcommand{\Qpbar}{\overline{\QQ}_p}

\newcommand{\App}{{\rm App}}

\begin{document}
\marginparwidth 50pt

\title{Genus 2 Curves with Complex Multiplication}

\author{Eyal Z. Goren \& Kristin E. Lauter}
\address{Department of Mathematics and Statistics,
McGill University, 805 Sherbrooke St. W., Montreal H3A 2K6, QC,
Canada.}\address{Microsoft Research, One Microsoft Way, Redmond,
WA 98052, USA.} \email{goren@math.mcgill.ca;
klauter@microsoft.com} 

%

\maketitle


\section{Introduction}

\id While the main goal of this paper is to give a bound on the denominators of
Igusa class polynomials of genus 2 curves, our motivation is two-fold: on the one hand we are interested in applications to cryptography via the use of genus $2$ curves with a prescribed number of points, and on the other hand, we are interested in construction of class invariants with a view towards explicit class field theory and Stark's conjectures. In the following we give an overview of these motivating problems and explain the contents of the paper.

\

\id Some basic protocols in public key cryptography such as key exchange and digital signatures rely on
the assumption that the discrete logarithm problem in an underlying group is hard.  Current available
alternatives favor the use of the group of points on an elliptic curve or the Jacobian of a hyperelliptic genus 2 curve over a finite field as the underlying group.
The security of the system depends on the
largest prime factor of the group order, so it is crucial to
be able to construct curves such that the resulting group order is
prime. Also, for applications in pairing-based cryptography, it may be necessary to impose
additional divisibility conditions on the group order.  Parameterized families of curves satisfying these
type of conditions are called pairing friendly curves.  Thus algorithms to construct curves with
prescribed group orders are required.  Currently, typical minimum security requirements require a group size of at least $2^{256}$ when the best-known attacks are square-root algorithms, giving roughly $128$ bits of security.  Compared to elliptic curves, Jacobians of genus 2 curves are an attractive alternative because they offer comparable security levels over a field of half the bit size, since the group size of the Jacobian of a genus 2 curve over a finite field $\FF_p$ is roughly $p^2$, whereas elliptic curves have group size roughly~$p$.

In the case of elliptic curves, the polynomial-time point-counting algorithm proposed by Schoof and improved by Elkies and Atkin (or the newer Arithmetic-Geometric mean algorithm) allows the following approach: one can pick elliptic curves over a finite field of cryptographic size and count points until a prime group order is found.  This solution will not work for generating pairing-friendly curves however.
Also, over prime fields of cryptographic size, it will not work for hyperelliptic curves of genus greater than one, either. Starting with the work of Atkin and Morain on generating elliptic curves with a prescribed group order for primality proving, the standard approach to constructing such curves has been to use the theory of Complex Multiplication in the so-called CM method.

Given a prime number $p$, and a non-negative group order $N$ lying in
the Hasse-Weil interval $[p+1-2\sqrt{p},p+1+2\sqrt{p}]$, the goal is
to produce an elliptic curve $E$ over $\FF_p$ with $N$ $\FF_p$-points:
$\#E(\FF_p) = N =p+1-t$,
where $t$ is the trace of the Frobenius endomorphism of $E$ over $\FF_p$.
Set $D=t^2-4p$.  The Frobenius endomorphism of $E$ has characteristic polynomial
$x^2-tx+p$, so it follows from the quadratic formula that the roots of this polynomial
lie in $\QQ(\sqrt{D})$.  It is standard to associate the Frobenius
endomorphism with a root of this polynomial.  If $E$ is not supersingular,
then $R$, the endomorphism ring of $E$, is an order in the ring
of integers of $K=\QQ(\sqrt{D})$.  Now the problem is transformed into one of
generating elliptic curves  with endomorphism ring equal to an order in $K$.
The correspondance between isomorphism classes of elliptic curves over $\overline{\QQ}$
with endomorphism ring equal to~$\ok$ and primitive, reduced, positive
definite binary quadratic forms of discriminant~$D$ gives an easy way to
run through all such elliptic curves.

Define the Hilbert class polynomial~$H_D(X)$ associated to the field $K$ as follows:
\[H_D(X) = \prod \left(X - j \left(\frac{-b+\sqrt{D}}{2a} \right) \right),\]
where the product ranges over the set of $(a,b) \in \ZZ^2$ such that
$ax^2 + bxy + cy^2$ is a primitive, reduced, positive definite binary
quadratic form of discriminant~$D$ for some $c \in \ZZ$,
and $j$ denotes the modular $j$-function.
The degree of $H_D(X)$ is equal to $h_K$, the class number of $K$, and it is known that $H_D(X)$ has integer coefficients.
To find an elliptic curve modulo~$p$ with $N$ points over $\FF_p$, it suffices to find a root $j$
of ~$H_D(X)$ modulo~$p$. One can then reconstruct the elliptic curve from its $j$-invariant $j$.
Assuming $j \neq 0, 1728$ and $p \ne 2,3$, the required elliptic curve is given by the Weierstrass equation
$y^2=x^3+3kx+2k$, where $k=\frac{j}{1728-j}$. The number of points on the elliptic curve is either
$p+1-t$ or~$p+1+t$, and one can easily check which one it is by randomly picking points and checking whether they are killed by the group order.

There are at least three approaches to computing the Hilbert class polynomial.
The complex analytic approach computes $H_D(X)$ as an integral polynomial
by listing all the relevant binary quadratic forms, evaluating the
$j$-function as a floating point integer with sufficient precision,
and then taking the product and rounding the coefficients to nearest
integers.   The Chinese remainder theorem (CRT) approach
computes $H_D(X) \bmod \ell$ for sufficiently many small primes $\ell$ and then uses
the Chinese remainder theorem (CRT) to compute $H_D(X)$ as a polynomial with
integer coefficients.  The $p$-adic approach uses p-adic lifting to approximate the roots and recognize the polynomial.
These algorithms are all satisfactory in practice for small $D$, and the current world record for the largest
$D$ for which $H_D(X)$ has been computed is held by the Explicit CRT method for some $|D|>10^{13}$
\cite{Sutherland}.

The situation for generating genus $2$ curves is more difficult.
The moduli space of genus 2 curves is $3$-dimensional and so at least $3$ invariants are needed to specify a curve up to isomorphism, and, in fact, Igusa's results show that most genus curves are determined by $3$ invariants.
The CM algorithm for genus $2$ is analogous to the Atkin-Morain CM
algorithm for elliptic curves just described. But whereas the Atkin-Morain algorithm computes
the Hilbert class polynomial of an imaginary quadratic field $K$ by evaluating the modular $j$-invariants of all elliptic curves with CM by $K$, the genus $2$ algorithm computes {\em Igusa class polynomials} of a quartic CM field $K$ by evaluating
the modular invariants of all the abelian varieties of dimension $2$ with CM by $K$.

For a primitive quartic CM field $K$ we can define Igusa class polynomials
\[h_i(X)=\prod_\tau(X-j_i(\tau)), \qquad i=1,2,3,\] in analogy with the
Hilbert class polynomial for a quadratic imaginary field; they depend on the quartic CM field $K$, but we suppress it in the notation.  The roots are CM values of
Siegel modular functions, and it is known that these roots generate abelian extensions of
the reflex field of $K$.  Again in analogy with the elliptic curve case,
CM values of modular functions on the Siegel upper half space can
be directly related to the invariants of a binary sextic defining the genus 2 curve
associated to the CM point.
Note that the $j$-invariant of an elliptic curve can be calculated
in two ways, either as the value of a modular function on a lattice
defining the elliptic curve as a complex torus over $\CC$, or directly
from the coefficients of the equation defining the elliptic curve.  Similarly for genus $2$ curves, the
triple of Igusa invariants of a genus $2$ curve
can also be calculated in two different ways.  Using classical
invariant theory over a field of characteristic zero, Clebsch defined
the triple of invariants of a binary sextic $f$ defining a genus $2$
curve $y^2=f(x)$.  Bolza showed how those invariants could also be
expressed in terms of theta functions on the period matrix associated
to the Jacobian variety and its canonical polarization over $\CC$.
Igusa showed how these invariants could be extended to work in
arbitrary characteristic, and so the invariants are often referred
to as Igusa or Clebsch-Bolza-Igusa invariants.  These invariants will be discussed in more detail in \S~\ref{section: moduli genus 2} below.  To recover the equation of a genus $2$ curve given its invariants,
Mestre gave an algorithm which works in most cases, and involves
possibly passing to an extension of the field of definition of the
invariants (\cite{Mestre}).

The CM algorithm for genus $2$ curves takes as input a quartic CM field $K$ and outputs the Igusa class
polynomials with coefficients in $\QQ$ and if desired, a suitable prime
$p$ and a genus~$2$ curve over~$\FF_p$ whose Jacobian has CM by $K$.
The CM algorithm was first implemented by Spallek~\cite{Spallek}, van
Wamelen~\cite{vanW}, and Weng~\cite{Weng}.  Alternative algorithms for 
computing Igusa class polynomials have also been proposed and studied, such as the genus 2 Explicit CRT algorithm~\cite{EL} and a p-adic approach~\cite{GHKRW}.

To compute the Igusa polynomials, Spallek~\cite{Spallek} determined a collection of period matrices which form a set of representatives for isomorphism classes of polarized abelian surfaces with CM by a given field.  Determining this set was complicated, and a complete set of representatives in general was not determined until the recent work of Streng~\cite{Streng}.  In~\cite{Weng}, Weng gave an algorithm for computing the minimal polynomials of Igusa invariants by evaluating Siegel modular forms to very high precision in order to recognize the coefficients of the minimal polynomials as rational numbers.  Unfortunately, the polynomials $h_i(X)$ have rational coefficients, not integral coefficients, which makes them harder to recognize from floating point approximations.  The large number of floating point multiplications performed in the computation causes loss of precision and makes the algorithm hard to analyze. The running time of the CM algorithm had until recently not yet been analyzed due to the
fact that no bound on the denominators of the coefficients of the
Igusa class polynomials was known.

Since the polynomials $h_i(X)$ have rational coefficients,
we can ask about the prime factorization of the coefficients.
In particular, the primes appearing in the denominators are of special interest.
In~\cite{Lauter}, it was conjectured that primes in the denominator
are bounded by the discriminant of the CM field and satisfy some additional arithmetic conditions.
In fact, the primes in the denominator are primes of bad reduction for the associated curve.
It was shown in~\cite{GL} that bad reduction of a CM curve at a prime is equivalent to existence of a solution to a certain embedding problem: embedding the ring of integers of the primitive quartic CM field
into the endomorphism ring of a product of supersingular elliptic curves in a way which is
compatible with the Rosati involution induced by the product polarization.
In~\cite{GL}, we provided bounds on the primes which can appear
in the prime factorization of the denominators.  In the present paper, we extend that work to provide bounds on the powers to which those primes appear, thereby proving an absolute upper bound on the size of the denominators.  Our bounds have already been used in ~\cite{Streng} to provide a complete running time analysis of the complex analytic CM method for genus $2$. The additional arithmetic conditions turn out to be equivalent to superspecial reduction of the abelian surfaces in question and so are essentially covered by \cite{GorenReduction}, and in more generality in \S \ref{section: reduction of abelian surfaces with complex multiplication} of this paper.
In related work~\cite{BY}, the factorization of the denominators, when averaged over the corresponding CM cycle, was studied and a precise conjecture 
was formulated.  In subsequent work of Yang, the conjecture was proved for certain classes of quartic CM fields, thereby giving tight bounds on the size of the denominators in those cases.  But a general bound
needed for the complexity analysis has not been known until the work of the present paper.

\

\id The investigations carried out in this paper also have a completely different motivation, which comes from class field theory and Stark's conjecture. Consider the modular form that we call $\Theta$ in this paper (\S \ref{section: Igusa class polynomials}); it is the unique Siegel cusp form of weight $10$ and full level, up to a scalar, and is equal, up to a scalar, to the product of the squares of the $10$ even Riemann theta functions of integer characteristics. In many ways $\Theta$ is the analogue the elliptic cusp form~$\Delta$ of weight $12$. Because of this analogy, Goren and Deshalit have studied in \cite{DeShalitGoren} certain algebraic numbers constructed from values of $\Theta$ at CM points associated to a primitive quartic CM field~$K$, whose definition parallels the definition of the Siegel units. Certain expressions in such values gave quantities $u(\gera, \gerb)$ associated to certain ideals in $K$, that depend also on the choice of CM type. These quantities lie in the Hilbert class field of the reflex field of $K$ and have many appealing properties, such as a nice transformation law under Galois automorphism, and their dependence only on the ideal classes of $\gera$ and $\gerb$. Thus, one is justified in calling them class invariants.

A natural question that arose is whether these invariants are actually units, or close to being units, in the sense that one knows their exact prime factorization, and these primes are small relative to, say, the discriminant of the field $K$. While we do not have complete solutions, several results concerning these have been obtained by the authors in recent years \cite{GL, GL2}. See also \cite{Vallieres} for numerical data. One of the main reasons to study such quantities is Stark's conjecture.

Recall that for a number field $K$ and $\germ$, a modulus in $K$ divisible by all the infinite primes, Stark's conjecture asserts that if
$\zeta(K; \mathcal{A}, 0) = 0$ then there exists a unit $u(\calA)$ of $K[\germ]$ (the associated ring class field) such that
\[ \zeta^\prime(K, \mathcal{A}, 0) = \log \vert u(A)\vert, \]
where $\zeta(K; \mathcal{A}, s)$ is the partial zeta function associated to an ideal class $\calA$ modulo $\germ$. In spite of much work in this area, it is fair to say that Stark's conjecture is essentially completely open. It is believed that the main obstacle is finding a ``good" construction of units, and that was precisely the motivation of \cite{DeShalitGoren}, although the problem of relating the class invariants $u(\gera, \gerb)$ to $L$-functions is still outstanding.

Now, as it turns out, the denominators occurring in the coefficients of the Igusa class polynomials $h_i$ have to do with the modular form $\Theta$ as well, and essentially both questions - the nature of the denominators and the factorization of the invariants $u(\gera, \gerb)$ - have the same underlying geometric question, which is whether an abelian surface with CM by $K$, over some artinian local ring, can be isomorphic to a product of elliptic curves (with additional conditions on polarizations). The main theorem of the paper is the following.

\

\id {\bf Theorem~\ref{theorem: main theorem}.} \emph{Let $f = g/\Delta^k$ be a modular function of level one on $\gerH_2$ where:
\begin{enumerate}
\item $\Delta$ is Igusa's $\chi_{10}$, the product of the squares of the ten Riemann theta functions with even integral chracteristics, normalized to have Fourier coefficients that are integers and of g.c.d. $1$.
\item $g$ is a level one modular form of weight $10 k$ with integral Fourier coefficients whose g.c.d. is $1$.
\end{enumerate}
Then $f(\tau) \in L= NH_{K^\ast}$ and
\begin{equation}
\val_{\gerp_L}( f(\tau)) \geq \begin{cases}
-4ke\left(\log_p\left(\frac{d \Tr(r)^2}{2}\right)+1\right) & e \leq p-1 \\
-4ke\left(8\log_p\left(\frac{d \Tr(r)^2}{2}\right)+2\right) & \text{any other case.}
\end{cases}
\end{equation}
Furthermore, unless we are in the situation of superspecial reduction, namely, we have a check mark in the last column of the tables in \S~\ref{section: reduction of abelian surfaces with complex multiplication}, $\val_{\gerp_L}( f(\tau)) \geq  0$. The valuation is normalized so that a uniformizer at $\gerp_L$ has valuation $1$.}

\

\id Corollaries  \ref{cor: denominators}, \ref{cor: class invariants}, of this theorem give the applications to denominators of Igusa class polynomials and class invariants described above.

In order to prove this theorem, we develop several tools that are of independent interest. The first one is an explicit determination of the reduction of abelian surfaces with complex multiplication. The main invariants of an abelian surface $A$ over a field of characteristic $p>0$ are its $f$-number, that determines the size of the \'etale quotient of $A[p]$, and the $a$-number that determines the size of the local-local part of that group scheme.  These numbers determine, for example, in which Ekedahl-Oort strata the moduli point corresponding to $A$ lies. It turns out, and that was essentially known by \cite{GorenReduction} and \cite{Yu}, that these numbers can be read from the prime factorization of $p$ in the normal closure $N$ of $K$ over $\QQ$ and the CM type. However, to our knowledge, a complete analysis had not appeared in the literature, and we make this analysis explicit here, dealing also with ramified primes, in a self-contained manner.

For a prime $p$ to appear in the denominators of the $h_i$, or for $\gerp\vert p$ to appear in the factorization of a $u(\gera, \gerb)$, some abelian surface with CM by $K$ must be isomorphic over $\fpbar$ to the product of two supersingular elliptic curves $E \times E'$. This gives $f = 0, a = 2$, and so sieves out the ``evil primes" according to their factorization in $N$. A further, and most important condition, is imposed by the fact that the Rosati involution of $E \times E'$ must induce complex conjugation on $K$. We are able to translate the fact that a prime appears to a certain power in the denominators of the $h_i$ (similarly for the $u(\gera, \gerb)$) to the fact that a similar situation must hold over a certain artinian ring $(R, \germ)$ and the index of nilpotency of $\germ$ is proportional to the power of the prime. This requires some results in intersection theory (\S \ref{section: intersection theory}) and the introduction of an auxiliary moduli space (\S \ref{section: The moduli space of pairs of elliptic curves}).

A certain maneuver, already used in \cite{GL}, allows us to reduce the problem to a question about endomorphisms of elliptic curves over $R$ whose reduction modulo $\germ$ is supersingular. Some special instance of this problem was studied by Gross in \cite{Gross}, but his results do not suffice for our purposes. We approach this problem using crystalline deformation theory in \S \ref{section: deformation theory}; in the course of developing the results we need, we provide some more general results that are natural in that context and will, so we believe, be useful for others. Since crystalline deformation theory is only valid under certain restrictions on ramification, we provide an alternative approach that works without any restriction (\S \ref{subsec: high ramification}) and gives results that are not too much worse than crystalline deformation theory gives.



\

\

\section{Moduli of curves of genus $2$}\label{section: moduli genus 2}
\subsection{Curves of genus two - Igusa's results}
Let $y_1, y_2, y_3$ be independent variables and let $y_4 = \frac{1}{4} (y_1y_3 - y_2^2)$. The group of fifth roots of unity $\mu_5$ acts on the ring $\ZZ[\zeta_5][y_1, y_2, y_3, y_4]$ by $[\zeta](y_i) : = \zeta^i y_i$ (and trivially on the coefficients). The ring of invariants is defined over $\ZZ$ and we denote it, \emph{by abuse of notation}, \[\ZZ[y_1, y_2, y_3, y_4]^{\mu_5}.\] One of Igusa's main results \cite[p. 613]{IgusaArithmeticModuli} is that $\scrM_2$,\label{M2} the coarse moduli space of curves of genus~$2$, satisfies
\begin{equation} \scrM_2 \cong \Spec(\ZZ[y_1, y_2, y_3, y_4]^{\mu_5}).\end{equation}
This ring of invariant elements is generated over $\ZZ$ by 10 elements. We remark that outside the prime $2$, namely if we work over $\ZZ[1/2]$, we can dispense with $y_4$ and conclude that
\[ \scrM_2 \otimes \ZZ[1/2] \cong \Spec(\ZZ[1/2][y_1, y_2, y_3]^{\mu_5}).\]
(Same abuse of notation.)
Note that to find generators over $\ZZ[1/2]$ for $\ZZ[1/2][y_1, y_2, y_3]^{\mu_5}$ amounts to finding vectors $(a, b, c) \in \ZZ_{\geq 0}^3$ such that $a + 2b + 3c \equiv 0 \pmod{5}$ that generate the semigroup $\{(a, b, c) \in \ZZ_{\geq 0}^3: a + 2b + 3c \equiv 0 \pmod{5}\}$ -- one associate to the vector $(a, b, c)$ the monomial $y_1^ay_2^by_3^c$. Such generators are given by the following $8$ triples:
\begin{equation}\label{equation: triplets generators}\{ (0, 0, 5), (0, 5, 0), (5, 0, 0), (0, 1, 1), (1, 2, 0), (3, 1, 0), (1, 0, 3), (2, 0, 1)\}.
\end{equation}
On the other hand, given a field $k$ of odd characteristic, to find generators for the fraction field ${\rm Frac}(k[y_1, y_2, y_3]^{\mu_5})$, one needs generators for the group $\{(a, b, c) \in \ZZ^3: a + 2b + 3c \equiv 0 \pmod{5}\}$, which one can choose to be the vectors $(2, -1, 0), (3, 0, -1), (5, 0, 0)$ (corresponding to the monomials $y_1^2/y_2, y_1^3/y_3, y_1^5$), for example. 

Igusa's construction is based on much earlier work by Clebsch and others on invariants of sextics. A genus $2$ curve is hyperelliptic, where a hyperelliptic curve is defined to be a curve which is a double cover of the projective plane. In characteristic different from $2$ the situation is very much like over the complex numbers, and one can conclude that such a curve can be written as $y^2 = f(x)$, where $f(x)$ is a separable monic polynomial of degree $6$, uniquely determined up to projective substitutions, thus reducing the problems of classifying genus $2$ curves to studying when two sextics are equivalent under a projective transformation, or, equivalently, studying the space parameterizing unordered 6-tuples of points in $\PP^1$. 

\subsection{Igusa's coordinates}To describe the invariants of sextics we use Igusa's notation. Let
\[ y^2 = f(x) = u_0x^6 + u_1x^5 + \cdots + u_6,\]
be a hyperelliptic curve and let $x_1, \dots, x_6$ be the roots of the polynomial $f(x)$. The noation $(ij)$ is a shorthand for the expression $(x_i - x_j)$. Consider then
\begin{eqnarray}\label{ABCD}
A(u) & = & u_0^2 \sum_{\rm{fifteen}} (12)^2(34)^2(56)^2 \\
B(u) & = &u_0^4 \sum_{\rm{ten}} (12)^2(23)^2(31)^2(45)^2(56)^2(64)^2 \\
C(u) & = &u_0^6 \sum_{\rm{sixty}} (12)^2(23)^2(31)^2(45)^2(56)^2(64)^2(14)^2(25)^2(36)^2 \\
D(u) & = & u_0^{10} \prod_{i<j} (ij)^2
\end{eqnarray}
The subscript ``fifteen" in $A$ refers to the fact that there are $15$ ways to partition~$6$ objects into~$3$ groups of $2$ elements, the subscript ``ten" in $B$ refers to the fact that there are $10$ ways to partition~$6$ objects into $2$ groups of $3$ elements. The subscript ``sixty" refers to partitioning~$6$ objects into two groups and then finding a matching between these two groups: there are $10$ ways to partition into $2$ groups and six matching between the two chosen groups. The invariants $A, B, C, D$ are denoted $A', B', C', D'$ in \cite[p. 319]{Mestre}, but we follow Igusa's notation; these invariants are often called now the \emph{Igusa-Clebsch invariants}. Another common notation one finds in the literature is $I_2 = A, I_4 = B, I_6 = C, I_{10} = D$, for example in the Magma help pages on February 2010,  but we shall avoid using it, especially since it conflicts with Igusa notation as in \cite[p. 848]{IgusaProj}.

The invariants $A, B, C, D$ are homogenous polynomials of weights $2, 4, 6$ and $10$, respectively, in $u_0, \dots, u_6$, thought of as variables. In addition they are invariants of index $6, 12, 18$ and $30$, respectively. This means the following: Let $f(x, z)$ be the homogenized form of $f$, that is, 
\[ f(x, z) = u_0x^6 + u_1x^5z + \cdots + u_6z^6.\]
Let $M=\left(\begin{smallmatrix} \alpha & \beta \\ \gamma & \delta \end{smallmatrix} \right)\in \GL_2$ and let
\[ 
x = \alpha x^\prime + \beta z^\prime, \qquad z = \gamma x^\prime + \delta z^\prime.\] 
Write, by substituting these expressions for $x, z,$ and expanding,
\[ f(x, z) =  u_0^\prime {x^\prime}^6 + u_1^\prime {x^\prime}^5z^\prime + \cdots + u_6^\prime {z^\prime}^6.\] 
Then, a polynomial $J = J(u_0, \dots, u_6)$ in the variables $u_0, \dots, u_6$ is called an \emph{invariant of index~$k$} if
\[ J(u_0^\prime, \dots, u_6^\prime)= \det(M)^{k} J(u_0, \dots, u_6).\]
The terminology here is classical and follows, e.g., \cite{Mestre}. (An \emph{invariant}, in the terminology of loc. cit, is a covariant of order $0$, which means it is an expression in the coefficients of $f$ alone, as is the case here.) An invariant of degree $r$ of a sextic has index $3r$; cf. loc. cit. p. 314. 

Note that if we let  $f^\prime$ be the polynomial $f^\prime(t) = u_0^\prime t^6 + \cdots + u_6^\prime$ then the two hyperelliptic curves \[C: y^2 = f(x), \qquad C^\prime: y^{\prime,2} = f^\prime(x),\] are isomorphic. Indeed, the map \[(x^\prime, y^\prime) \mapsto (x, y):= \left(\frac{ax^\prime + b}{cx^\prime+d}, \frac{y^\prime}{(cx^\prime + d)^3}\right)\] gives an isomorphism $C^\prime \arr C$ as $(\frac{y^\prime}{(cx^\prime + d)^3})^2 = f(\frac{ax^\prime + b}{cx^\prime+d})$.

\

\id In characteristic $0$, every sextic gives a vector $(A, B, C, D)$ with $D \neq 0$ and, vice-versa, every such vector comes from a sextic. Two curves over an algebraically closed field are isomorphic if and only if one curve has invariants  $(A, B, C, D)$ and the invariants of the other curve are $(r^2A: r^4B: r^6 C: r^{10}D)$ for some $r\neq 0$ in the field \cite[Corollary, p. 632]{IgusaArithmeticModuli} (it would have been more natural to write the powers of $r$ in multiple of $6$, but we follow convention here). Thus, it is natural to associate to a sextic a vector $(A:B:C:D)$ in the weighted projective space $\PP^3_{2, 4, 6, 10}$. Similar to the case of the usual projective space $\PP^3_{1, 1, 1, 1}$, the complement of the hypersurface defined by $D = 0$ is affine. But, where for a usual projective space with coordinates $(x_0, x_1, x_2, x_3)$ the affine variety is $\Spec(\QQ[x_0/x_3, x_1/x_3, x_2/x_3])$, for a weighted projective space we need more functions; at the case at hand 
one needs $10$ functions, and these will be given below in terms of the $J_{2i}$;
every regular function on the affine variety $\PP^3_{2, 4, 6, 10} \setminus \{ D = 0\}$ is a polynomial in these functions.

Define

\begin{center}
{\renewcommand{\arraystretch}{2}
\begin{tabular}{p{4.5 cm}p{4.5 cm}p{5.5 cm}}\label{Js} $J_2 = 2^{-3}A$ & $J_4 = 2^{-5}3^{-1}(4J_2^2 - B)$ &$J_6 = 2^{-6}3^{-2}(8J_2^3 - 160J_2J_4 - C)$ \\ $J_8 = 2^{-2}(J_2 J_6 - J_4^2)$ &$J_{10}= 2^{-12}D$ &\end{tabular}
}
\end{center}

\

\id A calculation shows that these invariants still make sense in characteristic $2$. 

Let $\gerR$ be the ring of homogenous elements of degree zero in the graded ring generated over~$\ZZ$ by $J_2, J_4, \dots, J_{10}$ and localized at $J_{10}$. In fact, any absolute invariant, namely any invariant which is the quotient of two invariants of the same index, belongs to $\gerR$ (\cite[Proposition 3, p. 633]{IgusaArithmeticModuli}). One can show that there is an isomorphism
\[ \gerR \overset{\sim}{\Arr} \ZZ [y_1, y_2, y_3, y_4]^{\mu_5}, \]
uniquely determined by
\[ J_2^{e_1}J_4^{e_2}J_6^{e_3}J_8^{e_4}J_{10}^{-e_5} \mapsto y_1^{e_1}y_2^{e_2}y_3^{e_4}y_4^{e_4},\] where the $e_i$ are non-negative integers satisfying the relation $e_1+2e_2 + 3e_3 + 4e_4 =5e_5$ and as before $y_4 = \frac{1}{4}(y_1y_3-y_2^2)$. Igusa proceeds to show that the ring $\gerR$ is generated by $10$ elements over $\ZZ$, and by $8$ elements over $\ZZ[1/2]$, and that is best possible. 

Over $\ZZ$ the generators of $\gerR$ can be taken to be the following.

\begin{center}
{\renewcommand{\arraystretch}{2}
\begin{tabular}{p{3 cm}p{3 cm}p{3 cm}p{3 cm}p{3 cm}}\label{gammas}  $\gamma_1=J_2^5/J_{10}$ &  $\gamma_2=J_2^3J_4/J_{10}$ &  $\gamma_3= J_2^2J_6/J_{10}$ &$\gamma_4 = J_2J_8/J_{10}$& $\gamma_5=J_4J_6/J_{10}$   \\  $\gamma_6 = J_4J_8^2/J_{10}^2$  &  $\gamma_7= J_6^2J_8/J_{10}^2$ &  $ \gamma_8=J_6^5/J_{10}^3$ &  $\gamma_9= J_6J_8^3/J_{10}^3$ & $\gamma_{10} = J_8^5/J_{10}^4$
\end{tabular}
}
\end{center}

\

\id (Over $\ZZ[1/2]$ a set of generators is
\begin{center}
{\renewcommand{\arraystretch}{2}
\begin{tabular}{p{3 cm}p{3 cm}p{3 cm}p{3 cm}}$g_1=J_2^5/J_{10}$ &  $g_2=J_2^3J_4/J_{10}$ &  $g_3=J_2J_4^2/J_{10}$ &  $g_4= J_2^2J_6/J_{10}$ \\  $g_5=J_4J_6/J_{10}$ &  $g_6= J_2J_6^3/J_{10}^2$ &  $ g_7=J_4^5/J_{10}^2$ &  $g_8= J_6^5/J_{10}^3$
\end{tabular}
}
\end{center}
\id (and the reader will recognize the exponents from (\ref{equation: triplets generators}).) We call them \emph{the Igusa coordinates of $\scrM_2$}. Here are some consequences of these results.
\begin{enumerate}
\item
 Let $C_1, C_2$, be curves over an algebraically closed field $k$, and write $C_i: y^2 = f_i(x)$, where $f_i(x) \in k[x]$ is a sextic. Then,
\[ C_1 \cong C_2 \quad \Longleftrightarrow \quad (\gamma_1(f_1), \dots, \gamma_{10}(f_1)) = (\gamma_1(f_2), \dots, \gamma_{10}(f_2)).\]
\item
Let $C$ now be defined over a number field $L_0$, $C: y^2 = f(x), f(x) \in L_0[x]$, then $C$ has potentially good reduction at a prime $\gerp$ of $L_0$, namely, there exists a finite extension field $L/L_0$ and an ideal $\gerP\vert \gerp$ of $L$ such that $C$ has good reduction modulo $\gerP$, if and only if \[\val_\gerp(\gamma_i(f)) \geq 0, \qquad  i = 1, \dots, 10.\]
\item
Let $C_1, C_2$, be curves over a number field $L$, $C_i: y^2 = f_i(x)$ as above, having good reduction at $\gerp$. Then, 
\begin{multline*}\qquad \qquad  C_1 \pmod{\gerp}\cong_{/\fpbar} C_2 \pmod{\gerp} \quad \Longleftrightarrow\\ (\gamma_1(f_1), \dots, \gamma_{10}(f_1))  \equiv (\gamma_1(f_2), \dots, \gamma_{10}(f_2)) \pmod{\gerp}.\end{multline*}
\end{enumerate} 

\subsection{Efficacy of the absolute Igusa invariants} \label{section: absolute invariants} The so-called \emph{absolute Igusa invariants} are the functions
\[\label{is} i_1= A^5/D, \qquad i_2 = A^3B/D, \qquad i_3 = A^2C/D.\]
The choice of terminology is somewhat unfortunate, as it leads one to think that these invariants determine the isomorphism class of the curve; we'll discuss it further below.

Since $D= 2^{12}J_{10}$, the functions $i_1, i_2, i_3,$ belong to $\gerR \otimes \ZZ[1/2]$. It is a consequence of the results mentioned so far that the functions $\gamma_j$ are rational functions of the functions $i_j$ and vice-versa. This calculation is presented in the following two tables.

\begin{center}
\begin{table}[h]
\caption{The absolute Igusa invariants $i_1, i_2, i_3$ in terms of the generators $\gamma_j$}
{\renewcommand{\arraystretch}{2}
\begin{tabular}{||l | l ||} \hline\hline
$i_1$ & $ 8 \cdot \gamma_1$\\ \hline
$i_2$ & $\frac{1}{2}\cdot(\gamma_1 -24\cdot\gamma_2)$ \\ \hline
$i_3$ & $\frac{1}{8}\cdot(\gamma_1 -20 \cdot \gamma_2 - 72\cdot \gamma_3)$ \\ \hline\hline
\end{tabular}
}
\end{table}
\end{center}

\

\

\begin{table}
\caption{The generators $\gamma_j$ in terms of the absolute Igusa invariants $i_1, i_2, i_3$ 
\newline
(the last column gives the denominator)}
{\renewcommand{\arraystretch}{2}
{\small
\begin{tabular}{|| p{0.6 cm} | p{9.5 cm}||}\hline \hline
$\gamma_1$ &  $2^{-3}\cdot i_1$ \\ \hline
$\gamma_2$ &  $2^{-6}3^{-1}\cdot (i_1 - 16\cdot i_2) $  \\ \hline
$\gamma_3$ &  $\frac{1}{3456}\cdot (i_1 + 80 \cdot i_2 - 384 \cdot i_3)$  \\ \hline
$\gamma_4$ & $2^{-11}3^{-3}\cdot\,{\frac {{{  i_1}}^{2}+416 \cdot\,{  i_1}\,{ i_2}-
1536 \cdot\,{  i_1}\,{  i_3}-768 \cdot\,{{  i_2}}^{2}}{{  i_1}}}
$\\ \hline

$\gamma_5$ & $2^{-10} \cdot 3^{-4}\cdot\,{\frac { \left( {  i_1}-16 \cdot\,{  i_2} \right) 
 \left( {  i_1}+80 \cdot\,{  i_2}-384 \cdot\,{  i_3} \right) }{{  i_1}}}
$  \\ \hline

$\gamma_6$ & $2^{-25}\cdot 3^{-7} \cdot\,{\frac { \left( {  i_1}-16 \cdot\,{ i_2}
 \right)  \left( {{  i_1}}^{2}+416 \cdot\,{  i_1}\,{ i_2}-1536 \cdot\,{  i_1}
\,{  i_3}-768 \cdot\,{{  i_2}}^{2} \right) ^{2}}{{{  i_1}}^{3}}}
$ 

\\ \hline

$ \gamma_7$ &  

$2^{-22} \cdot 3^{-9} \cdot\,{\frac { \left( {  i_1}+80 \cdot\,{ i_2}-384 \cdot\,{
  i_3} \right) ^{2} \left( {{  i_1}}^{2}+416 \cdot\,{  i_1}\,{ i_2}-
1536 \cdot\,{  i_1}\,{  i_3}-768 \cdot\,{{  i_2}}^{2} \right) }{{{ i_1}}^{2}}
}
$

\\ \hline

$\gamma_8$ & 

$2^{-29} \cdot 3^{-15}\cdot\,{\frac { \left( { i_1}+80\cdot\,{  i_2}-
384\cdot\,{  i_3} \right) ^{5}}{{{  i_1}}^{2}}}
$

\\ \hline
$ \gamma_9$ &  

$2^{-37} \cdot 3^{-12} \cdot\,{\frac { \left( { i_1}+80 \cdot\,{  i_2}-
384 \cdot\,{  i_3} \right)  \left( {{  i_1}}^{2}+416 \cdot\,{  i_1}\,{ i_2}-
1536 \cdot\,{  i_1}\,{  i_3}-768 \cdot\,{{  i_2}}^{2} \right) ^{3}}{{{ i_1}}^
{4}}}
$

\\ \hline

$\gamma_{10}$ &

$2^{-52} \cdot 3^{-15} \cdot\,{\frac { \left( {{ i_1}}^{2}+
416 \cdot\,{  i_1}\,{  i_2}-1536 \cdot\,{  i_1}\,{  i_3}-768 \cdot\,{{ i_2}}^{2}
 \right) ^{5}}{{{  i_1}}^{6}}}
$
 \\ \hline \hline
\end{tabular}
}
}
\end{table}

\id An interesting consequence of this calculation is that the natural map 
\[ \scrM_2 \otimes \ZZ[1/6] = \Spec(\gerR \otimes [1/6]) \Arr \Spec(\ZZ[1/6][i_1, i_2, i_3]) = \AA^3_{\ZZ[1/6]},\]
 can be inverted whenever $i_1 \neq 0$. However, given a triple $(i_1, i_2, i_3)$, which is in the image of the map, and such that $i_1 = 0$, we find that $A= 0$ and hence also that $i_2 = i_3 = 0$. Thus, there is a unique point of $\AA^3$, which is in the image, for which we cannot invert the map and it corresponds to all the genus $2$ curves for which $A = 0$. Thus, the absolute Igusa invariants fail to completely determine the isomorphism class of the curve, but only if $i_1 =0$. 
 
The vanishing locus of $A$ is a surface in $\scrM_2$. 
 There is a natural immersion,
 \[\rho: \scrM_2 \Arr \scrA_{2,1}, \]
 of the moduli space of curves $\scrM_2$ to the moduli space of principally polarized abelian surfaces with no level structure $\scrA_{2, 1}$, sending a curve to its canonically polarized Jacobian. The image is the complement of the Humbert surface $H_1$, which is the divisor of the modular form $\Theta $ defined below. Via this map, each of the Igusa invariants is, in a suitable sense, a pull-back via $\rho$ of a meromorphic Siegel modular form whose polls are supported on $H_1$. These modular forms were calculated by Igusa \cite[p. 177-8]{IgusaRingOfModularFormsOverZ} and the reader is referred to this reference for details. The invariant $D$ is the pullback of a scalar multiple of $\Theta$, defined in page \pageref{Delta2}. There is a modular form of weight $12$, which Igusa denotes $\chi_{12}$, such that, in a suitable sense, $A$ is a scalar multiple of the weight $2$ meromorphic form $\chi_{12}/\Theta$. We have thus, as sets,
 \[ \{ A = 0\} = \rho^{-1} \{ \chi_{12} = 0\}.\]
 The modular form $\chi_{12}$ is a cusp form (see \cite[p. 195]{IgusaSiegel}). However, there does not seem to be any simple interpretation to its vanishing loci. Indeed, the results of \cite{vdG} (see, in particular, \S 8 there) imply that the divisor of $\chi_{12}$ is \emph{not} supported on a union of Humbert surfaces. 
 
\begin{prop}\label{prop: failure of igusa invariants} Let $V \subseteq \scrA_{2, 1}(\CC)$ be the support of the divisor of $\chi_{12}$. There are finitely many primitive CM points on $V$, that is, CM points associated to primitive CM fields of degree four. 
\end{prop} 

\begin{proof} Let $S$ be the collection of all primitive CM points on $V$. We note that the description of~$\scrM_2$ implies that $V$ is irreducible. Let $C$ be the Zariski closure of $S$. If $S$ is infinite then $C$ is either a curve, or $V$ itself. In either case, it follows from the Andr\'e-Oort conjecture, known to be true under GRH by the work of Klinger-Yafaev \cite{Yafaev}, that $C$ is either a Shimura curve, or a Shimura surface. It remains to review the possibilities: (i) if $C$ is a Shimura curve then every CM point on $C$ is coming from a bi-quadratic (equivalently, non-primitive) CM field of degree $4$; (ii) if $C = V$ then $V$ is a priori in the Hecke orbit of some Humbert surfaces, but that Hecke orbit is a union of Humbert surfaces (this follows easily from the moduli interpretation). Since the Humbert surfaces in $\scrA_{2, 1}$ are irreducible, $V$ is a Humbert surface itself, which is not the case. 
\end{proof}

\subsection{Igusa class polynomials}\label{section: Igusa class polynomials} 
 In~\cite[\S 5.2]{GL} it was explained how the absolute Igusa invariants can also be expressed in terms of Siegel modular functions.  We summarize this here for the reader's convenience.
 
 \
 \id 
The Igusa functions $i_1$, $i_2$, $i_3$ can be defined as rational functions in
Siegel Eisenstein series,~$\psi_w$, of weights $w=4$, $6$, $10$, $12$. To begin with,
the cusp forms $\Theta$ and $\chi_{12}$, of weights $10$ and $12$, introduced above can be expressed in terms of these Eisenstein series as follows (\cite[p.195]{IgusaSiegel} and \cite[p. 848]{IgusaProj}):
\[
-2^{-2} \Theta = \chi_{10} = \frac{-43867}{2^{12} \cdot 3^{5} \cdot 5^2 \cdot 7 \cdot 53} (\psi_4\psi_6 -\psi_{10})
\]
and 
\[
\chi_{12} = \frac{131\cdot 593}{ 2^{13} \cdot 3^{7} \cdot 5^{3} \cdot 7^{2} \cdot 337} (3^2\cdot 7^2 \psi_4^3 +2 \cdot 5^3 \psi_6^2 -691 \psi_{12}).
\]
Then the {\it Igusa functions\/} $i_1, i_2, i_3$ can be expressed as
\[ i_1 = 2\cdot 3^5 \frac{\chi_{12}^5}{\chi_{10}^6}, \quad 
i_2 = 2^{-3}\cdot 3^3 \frac{\psi_4 \chi_{12}^3}{ \chi_{10}^4}, \quad 
i_3 = 2^{-5} \cdot 3 \frac{\psi_6 \chi_{12}^2 \chi_{10} + 2^{2}\cdot 3 \psi_4 \chi_{12}^3}{ \chi_{10}^4}.
\]
(These are often called $j_1$, $j_2$, $j_3$ in the literature, but we stick with Igusa's notation.)

Let $K$ be a primitive, i.e. not biquadratic, CM field of degree $4$ over $\QQ$.
We define the {\it Igusa class polynomials} to be:
\begin{equation}\label{class polynomials 2} 
h_1(x) = \prod_\tau (x - i_1(\tau)), \quad h_2(x) = \prod_\tau (x - i_2(\tau)), \quad h_3(x) = \prod_\tau (x - i_3(\tau)),
\end{equation}
where the product is taken over all $\tau\in {\rm Sp}(4,\ZZ)\backslash \gerH_2$, such that the associated principally polarized abelian variety has CM by $\calO_K$. 
One can define other absolute invariants, called $\gerj_1$, $\gerj_2$, $\gerj_3$, as in~\cite[p. 473]{GL}, where it is also remarked that
$i_1 = 2^{-12}\gerj_1$ and $i_2 = 2^{-12}\gerj_2$, and then we define the corresponding class polynomials as follows:
\begin{equation} \label{equation for gerh polynomials}\gerh_i(x) = \prod_\tau (x - \gerj_i(\tau)), \qquad i=1,2,3.\end{equation}
The advantage of using these is that it is clear from their definition that they satisfy the hypotheses of our Main Theorem.

\subsection{Rosenhain normal form}
While Igusa's approach to the moduli of genus $2$ curves is through the study of invariants of sextics, we remark that after a finite extension of the base we may always arrange the six ramification points to contain $\{0, 1, \infty\}$ and arrive at an equation of the form
\begin{equation}\label{equation: of a curve}C: y^2 = x(x-1)(x - \lambda_1)(x - \lambda_2)(x - \lambda_3), 
\end{equation}
called the Rosenhain normal form of the curve $C$. The $\lambda_i$'s are in a finite field extension of the field of definition of the curve and are determined up to the action of the stabilizer of the triple $\{0, 1, \infty\}$ in $\PGL_2$ a group isomorphic to the symmetric group $S_3$ on $3$ letters. If $\tau \in \gerH_2$ is the period matrix of the polarized abelian surface $\Jac(C)$, then, up to projective equivalence we may take the $\lambda_i$ to be 

\

\[ \lambda_1 = \frac{
\Theta\left[ \begin{smallmatrix} 1 \\ 1 \\ 0 \\ 0
\end{smallmatrix}\right] (\tau)^2\;
\Theta\left[ \begin{smallmatrix} 1 \\ 0 \\ 0 \\ 0
\end{smallmatrix}\right] (\tau)^2
}{ \Theta\left[ \begin{smallmatrix} 0 \\ 1 \\ 0 \\ 0
\end{smallmatrix}\right] (\tau)^2\;
\Theta\left[ \begin{smallmatrix} 0 \\ 0 \\ 0 \\ 0
\end{smallmatrix}\right] (\tau)^2
}, \quad \lambda_2 = \frac{
\Theta\left[ \begin{smallmatrix} 1 \\
0 \\ 0 \\ 1
\end{smallmatrix}\right] (\tau)^2\;
\Theta\left[ \begin{smallmatrix} 1 \\ 1 \\ 0 \\ 0
\end{smallmatrix}\right] (\tau)^2
}{ \Theta\left[ \begin{smallmatrix} 0 \\ 0 \\ 0 \\ 1
\end{smallmatrix}\right] (\tau)^2\;
\Theta\left[ \begin{smallmatrix} 0 \\ 1 \\ 0 \\ 0
\end{smallmatrix}\right] (\tau)^2
}, \quad
\lambda_3 = \frac{ \Theta\left[ \begin{smallmatrix} 1 \\
0 \\ 0 \\ 1
\end{smallmatrix}\right] (\tau)^2\;
\Theta\left[ \begin{smallmatrix} 1 \\ 0 \\ 0 \\ 0
\end{smallmatrix}\right] (\tau)^2
}{ \Theta\left[ \begin{smallmatrix} 0 \\ 0 \\ 0 \\ 1
\end{smallmatrix}\right] (\tau)^2\;
\Theta\left[ \begin{smallmatrix} 0 \\ 0 \\ 0 \\ 0
\end{smallmatrix}\right] (\tau)^2
}.
\]

\

\id (See \cite[p. 179]{IgusaSiegel}.) Note that if $C$ is defined over some number field and the equation (\ref{equation: of a curve}) is over that field then the $\lambda_i$ appearing there are algebraic numbers. When obtaining a triple of $\lambda_i$ from the period matrix the $\lambda_i$ are in fact algebraic. This is a consequence of the fact that the ratios of the theta functions appearing above are modular functions defined over $\QQ(i)$ (say) and full level $2$.

The theta functions used here are Riemann's theta functions:
Let~$\epsilon, \epsilon' \in \QQ^g$,~$\tau\in
\gerH_g$, and define the \emph{Riemann theta function} with
\emph{characteristic}~$\left[ \begin{smallmatrix} \epsilon \\
\epsilon'
\end{smallmatrix}\right]$ to be the power series
\[ \Theta\left[ \begin{smallmatrix} \epsilon \\ \epsilon'
\end{smallmatrix}\right] (\tau) =
\sum_{N\in
\ZZ^g}e\left(\frac{1}{2}{^t\left(N+\frac{\epsilon}{2}\right)} \tau
\left(N+\frac{\epsilon}{2}\right) +
{^t\left(N+\frac{\epsilon}{2}\right)}\frac{\epsilon'}{2}
\right),\] where~$e(x) = e^{2\pi i x}$. It can be shown that this
series defines a holomorphic function~$\gerH_g \rightarrow \CC$.
If~$\epsilon, \epsilon' \in \ZZ^g$ they are called an
\emph{integral characteristic}. If~$^t\epsilon \epsilon' \equiv 0
\pmod{2}$ they are called \emph{even}, and else \emph{odd}. It
turns out that for an odd characteristic the theta function
vanishes identically, and for even characteristic~$ \Theta\left[
\begin{smallmatrix} \epsilon \\ \epsilon'
\end{smallmatrix}\right] (\tau) ^2$ depends only on~$(\epsilon,
\epsilon')$ modulo~$\ZZ^{2g}$. For $g=1$ this gives us~$3$ (squares of) theta functions and for $g=2$ this gives us ten (squares of) even
theta functions. 

One can show that each~$\Theta\left[ \begin{smallmatrix} \epsilon
\\ \epsilon'
\end{smallmatrix}\right] (\tau)$ to a large enough even power~$2r$ is a
Siegel modular form of weight~$r$ of some level. For~$g=1$, it
goes probably to Jacobi that
\[ \Delta  =  c \prod_{\left[ \begin{smallmatrix} \epsilon
\\ \epsilon'
\end{smallmatrix}\right] \; \text{\rm even}} \Theta\left[ \begin{smallmatrix} \epsilon
\\ \epsilon'
\end{smallmatrix}\right] (\tau)^4, \]
where~$c$ is a constant and~$\Delta = E_4^3 - E_6^2$ is the
classical modular form of weight~$12$. Recall that the divisor of
$\Delta$ is the cusp of~$\SL_2(\ZZ) \backslash \gerH$. Igusa
proved for~$g=2$ that
\[ \Theta: = 2^{-12}\prod_{\left[ \begin{smallmatrix} \epsilon
\\ \epsilon'
\end{smallmatrix}\right] \; \text{\rm even}} \Theta\left[ \begin{smallmatrix} \epsilon
\\ \epsilon'
\end{smallmatrix}\right] (\tau)^2\]\label{Delta2}
is a Siegel modular form of level~$\Symp(4, \ZZ)$ and weight~$10$. The factor of $2$ is introduced to ensure integral Fourier coefficients with gcd $1$. 
The divisor of $\Theta $ is precisely the Humbert divisor~$H_1$ (with multiplicity~$2$). See \S \ref{section: The moduli space of pairs of elliptic curves}.

\

\id Finally, we make some remarks as the utility of the Rosenhain normal form for generating curves of genus $2$ with CM. Given a primitive CM field $K$, it is an easy matter to enumerate representatives for the ideal classes of $K$ and so to get, by varying over all CM types, the period matrices whose classes modulo $\Gamma(2)$ give all the CM points of level $2$ associated to this field. The modular forms used above are of level $2$ and so, by evaluating them on these period matrices, we get a collection of equations $C: y^2 = x(x-1)(x-\lambda_1)(x-\lambda_2)(x - \lambda_3)$ defining, in particular, the isomorphism classes of all the curves of genus $2$ whose Jacobians have CM by $\calO_K$. For a \emph{generic} period matrix $\tau\in \gerH_2$ the $\lambda_i(\tau)$ live in a very large extension $L$ of the field of definition, say $L_0$, of the curve. Let $L' = L(\sqrt{\lambda_1},\sqrt{\lambda_2}, \sqrt{\lambda_3})$; note that $[L': L]\leq 8$. Implicit in the fact that $\lambda_i \in L$ is that all the $2$-torsion of $\Jac(C)$ are defined over $L'$, because under the embedding $C \arr \Jac(C)$ (taking the point at infinity as the base point) the images of the Weierstrass points generate the $2$ torsion subgroup of $\Jac(C)$. For a generic curve, the extension $L'/L_0$ has degree $\sharp\; \Symp(4, \ZZ)/\Gamma(2) = 720$. However, in the case of complex multiplication, the field of definition of $C$ can be taken $H_{K^\ast}$, the Hilbert class field of the reflex field, and so the $\lambda_i$ generates the ray class field of conductor $2$ of $K^\ast$. In fact, since $\Symp(4, \ZZ)/\Gamma(2) \cong \Symp(4, \FF_2) \cong S_6$, the symmetric group on $6$ letters, and since the maximal abelian subgroups of $S_6$ have degrees $5,6, 8, 9$, it follows that $[L':L_0] \vert a$ for some $a\in \{5, 6, 8, 9\}$. This can be utilized to construct curves over a number fields whose Jacobians have CM.


\subsection{Ramification locus of $\scrA_{2, N}(\CC) \arr \scrA_{2, 1}(\CC)$} Let $N$ be a positive integer. We denote by $\scrA_{2, N}$\label{A2N} the moduli scheme of principally polarized abelian surfaces with symplectic level $N$ structure over $\Spec(\ZZ[\zeta_N, 1/N])$. $\scrA_{2, 1}\otimes \ZZ[\zeta_N, 1/N]$ is the quotient of $\scrA_{2, N}$ by the finite group $\Symp(4, \ZZ/N\ZZ)/\{ \pm I_4\}$.  We denote the by $H_{\Delta, N}$\label{HDeltaN} the Humbert surface of invariant $\Delta$ (a discriminant of real quadratic order) in $\scrA_{2, N}(\CC)$. It is irreducible for $N=1$, but reducible for $N>1$.

Let $N\geq 3$, so the representation $\Aut(A, \lambda) \arr \Aut(A[3])$ is faithful. The ramification locus of $\pi_N:\scrA_{2, N} \arr \scrA_{2, 1}$ is clearly the locus of points $x$ on $\scrA_{2, 1}$ with non-trivial stabilizers, which, by the moduli interpretation, correspond to principally polarized abelian surfaces $(A, \lambda)$ such that ${\rm r}\Aut(A, \lambda) \neq\{1\}$, where ${\rm r}\Aut$ is the reduced automorphism group (namely the automorphisms $\varphi:A \arr A$ such that $\varphi^\ast \lambda = \lambda$, taken modulo 
$\{\pm 1\}$). Furthermore, in that case, any point in the fibre over $x$ has the same ramification index, equal to the cardinality of ${\rm r}\Aut(A, \lambda)$.

We say that a component of the Humbert divisor $H_{\Delta, N}$ in $\scrA_{2, N}(\CC)$ is ramified if it is contained in the ramification locus and otherwise we say it is unramified. If every component of $H_{\Delta, N}$ is unramified then
\[ \pi_N^\ast (H_{\Delta, 1}) = H_{\Delta, N}.\] 
\begin{lem} \label{Lemma: ramification along Humbert surfaces}If $\Delta \neq 1, 4$ every component of $H_{\Delta, N}$ is unramified. If $\Delta \in \{1, 4\}$ the ramification index along each component of $H_{\Delta, N}$ is $2$.
\end{lem}
\begin{proof} Suppose first that $\Delta$ is not a square. In this case, every abelian variety $(A, \lambda)$ parameterized by $H_{\Delta,N}$ has real multiplication by a real quadratic order of discriminant $\Delta$ and, generically, only by that order. Thus, generically, $\Aut(A, \lambda) = \{ \pm 1\}$ (as the Rosati involution is the identity). That resolves this case.

Suppose now that $\Delta$ is a square, but $\Delta \neq 1$. Then. except of codimension one subset, the points of $H_{\Delta, N}$ correspond to curves $C$ of genus $2$ affording a map of degree $\Delta$, $C\arr E$, to an elliptic curve $E$, that does not factor non-trivially thorough another elliptic curve \cite{FreyKani}. 

From Igusa's classification of $\Aut(C)$ we deduce that there is only one 2-dimensional family of curves of genus $2$ with a non-trivial reduced automorphism group. This family, as one observes, is exactly the curves $C$ of genus $2$ allowing a map of degree $2$, $C \arr E$, to an elliptic curve, ramified at exactly two points of $E$. This family is the Humbert divisor $H_{4, 1}$, and in particular, we have proven the lemma for all cases but $\Delta = 1$.

It is easy to see that for a generic pair of elliptic curves $E_1, E_2$ we have $\Aut(E_1 \times E_2, \lambda_1 \times \lambda_2) = \{ (\pm 1, \pm 1)\}$. Thus, our proof is complete.
\end{proof}

\subsection{Existence of good models for abelian varieties with complex multiplication} Our purpose is to prove the following lemma.
\begin{lem}\label{Lemma: existence of models with good reduction}
Let $(A, \lambda)$ be a $g$-dimensional principally polarized complex abelian variety with complex multiplication by the ring of integers $\calO_K$ of a CM field $K$ of degree $2g$, $\iota:\calO_K \arr \End(A)$. Let $\Phi$ be the associated CM type and $K^\ast$ the reflex field associated to $K$ and $\Phi$. Assume that $\Phi$ is a simple CM type. Let $\gerp$ be a prime of $K^\ast$ and $R$ the completion of the ring of integers of $K^\ast$ by $\gerp$. Then $(A, \lambda, \iota)$ has a model with good reduction over an unramified extension $\calO$ of $R$. 
\end{lem}
\begin{proof}
As is well known (\cite[Ch. 3, Thm. 1.1]{LangCM}), $(A, \iota, \lambda)$ has a model over $H_{K^\ast}$, the Hilbert class field of $K^\ast$, corresponding to an $H_{K^\ast}$-rational point $a \in \scrA_{g, 1}$. In fact, $a$ is defined over the field of moduli $M$ of  $(A, \iota, \lambda)$ which is contained in $H_{K^\ast}$. Let $N\geq 3$ be an integer prime to~$p$. Let $\tilde a$ be a point of $\scrA_{g, N}$ lying above $a$. The field of definition of the point $\tilde a$ is contained in $H_{K^\ast}(A[N])$ and is equal to the field of moduli $M[N]$ of $(A, \iota, \lambda, A[N])$, which, since the moduli scheme is a fine moduli scheme, is also the field of definition of $(A, \iota, \lambda, A[N])$.

Let $\scrA_{g, N}^\dagger$ be a smooth toroidal compactification of $\scrA_{g, N}$ over $\Spec(\ZZ[\zeta_N, 1/N])$. It carries a semi-abelian variety $X$ over it. Choose a prime $\gerP$ of $M[N]$ over $\gerp$. Since the morphism $\scrA_{g, N}^\dagger \arr \Spec(\ZZ[\zeta_N, 1/N])$ is proper, the morphism $\Spec(M[N]_\gerP) \arr  \scrA_{g, N} \injects \scrA_{g, N}^\dagger$ induces by the valuative criterion a morphism $\beta:\Spec(\calO) \arr \scrA_{g, N}^\dagger $, where $\calO$ is the valuation ring of $M[N]_\gerP$. Then $\beta^\ast X$ is a principally polarized semi-abelian variety over $\calO$ whose generic fiber is $(A, \lambda)\otimes M[N]_\gerP$ ($\iota$ extends automatically). As is well known, since $[K:\QQ] = 2g > g$, the toric part of the mod $\gerP$ reduction of $\beta^\ast X$ must be trivial and so $A\otimes M[N]$ has good reduction modulo~$\gerP$.

We claim that the extension $M[N]/K^\ast$ is unramified at $\gerp$ and so $\calO$ is an unramified extension of $R$. This follows from the main theorems of complex multiplication. In fact $M[N]$ is an abelian extension of $K^\ast$ corresponding to a precisely described group of ideals and has conductor dividing $N$. See \cite{LangCM}, Chapter 5, Theorem 4.3 (use also Theorem 3.3). Thus, the extension $M[N]/K^\ast$ is unramified at $\gerp$.
\end{proof}
\begin{rmk} In fact, using more subtle results in complex multiplication due to Shimura, one can conclude the existence of a model over $H_{K^\ast}$ with good reduction at $\gerp$. See \cite[Proposition 2.1]{GorenReduction}.
\end{rmk}
\begin{rmk} The abelian variety $(A, \iota, \lambda)$ has a model over $H_{K^\ast}$, but this model is not unique. In fact, the forms of  $(A, \iota, \lambda)$ over $H_{K^\ast}$ are classified, up to $H_{K^\ast}$ isomorphism, by the Galois cohomology group $H^1(G_{H_{K^\ast}}, \Aut(A, \iota, \lambda))$, where $G_{H_{K^\ast}}$ is the absolute Galois group of $H_{K^\ast}$ and $\Aut(A, \iota, \lambda)$ are the automorphisms commuting with the action of $K$ and preserving the polarization. It is easy to see that $\Aut(A, \iota, \lambda)$ is equal to $\mu_K$, the group of roots of unity lying in $K$. Typically this group is $\{\pm 1\}$ and the forms correspond to quadratic twists, but it may be larger. It can be $\mu_t$ for any $t$ such that $\varphi(t) \vert 2g$. On the other hand, with accordance with the fine moduli space property $(A, \iota, \lambda, A[N])$ has no forms as $\Aut((A, \iota, \lambda, A[N])) = \{ 1\}$ for $N\geq 3$.   
\end{rmk}

\

\


\section{Reduction type of abelian surfaces with complex multiplication} 
\label{section: reduction of abelian surfaces with complex multiplication}

\id Our goal in this section is to study the reduction type of an abelian surface with complex multiplication by a field $K$ modulo a prime ideal of the field of definition, lying above $p$, as a function of the decomposition of the prime $p$ in $K$

\subsection{Combinatorics of embeddings and primes}\label{subsec: embeddings and primes} Let $K$ be a number field and $N$ its normal closure over $\QQ$. Let $G$ be the Galois group $\Gal(N/ \QQ)$, acting on $K$ by $k \mapsto g(k), g\in G$, and let $H = \Gal(N/K) < G$. Fix inclusions 
\[ \varphi_\CC\colon N \arr \CC, \qquad \varphi_p\colon N \arr \Qpbar.\] This allows us to make the following identifications:
\[ \Hom(K, \CC) = \varphi_\CC \circ G/H, \qquad \Hom(K, \Qpbar) = \varphi_p \circ G/H, \]
where a left coset $gH$ gives the embeddings $\varphi_\CC \circ g$ and $\varphi_p \circ g$. We then have an identification
\[ \Hom(K, \CC) = \Hom(K, \Qpbar).\] 
Let $L\supseteq N$ be a finite extension and choose an extension of $\varphi_\CC, \varphi_p$ to $L$. We have the following diagram:
\[\xymatrix@!C=3pc{& L\ar@{-}[d] \ar@{^{(}->}[r]^{\varphi_\CC (\varphi_p)} & \CC ( \Qpbar)\\ & N\ar@{-}[dl]\ar@{-}[dr] & \\K\ar@{-}[dr] & & K^\ast\ar@{-}[dl]\\ & \QQ & }\]
Let $\gerP$ be the maximal ideal of $\Qpbar$. 
The choice of $\varphi_p$ provides us with a prime ideal $\gerp_{L, 1}:= \varphi_p^{-1}(\gerP)$ of $L$, and so with prime ideals $\gerp_{N, 1} = \gerp_{L, 1} \cap N$ of $N$ and $\gerp_{K, 1} = \gerp_{L, 1} \cap K$ of $K$. Let $D$ be the decomposition group of $\gerp_{N, 1}$ in $N$ and $I$ its inertia group. Let $e = \sharp\; I $. The primes ideals above $p$ in $N$ are in bijection with the cosets $G/D$:
\[ p\calO_N = \prod_{\alpha \in G/D} \gerp_{N, \alpha}^e, \qquad \gerp_{N, \alpha} = \alpha(\gerp_{N, 1}).\]
The decomposition (respectively, inertia) group of $\gerp_{N, \alpha}$ is $D^\alpha: = \alpha D \alpha^{-1}$ (respectively, $I^\alpha:= \alpha I \alpha^{-1}$). The primes dividing $p$ in $K$ correspond to the double cosets $H\backslash G /D$. More precisely, 
\[ p \calO_K = \prod_{H\alpha D \in H\backslash G/D}\gerp_{K, \alpha}^{e(\alpha)}, \qquad \gerp_{K, \alpha} = \alpha(\gerp_{N, 1}) \cap K, \]
where, by Lemma~\ref{lemma: inertia in towers} below, $e(\alpha) = [I^\alpha: I^\alpha \cap H]$.

Let $\alpha \in G$; it induces a homomorphism $\varphi_p \circ \alpha\colon K \arr \Qpbar$ that depends only on $\alpha H$. It therefore defines a prime $(\varphi_p \circ \alpha)^{-1} (\gerP)$ of $K$, or more precisely $(\varphi_p\vert_K \circ \alpha)^{-1} (\gerP)$. We have 
\begin{equation}\label{equation: primes and embeddings} (\varphi_p\vert_K \circ \alpha)^{-1} (\gerP) = (\alpha^{-1} \varphi_p\vert_N^{-1} (\gerP)) \cap K = \alpha^{-1}(\gerp_{N, 1})\cap K = \gerp_{K, \alpha^{-1}}.
\end{equation}
 That is, \emph{the coset $\alpha H$ corresponding to an embedding $K \arr \Qpbar$ induces the prime  corresponding to the double coset $H \alpha^{-1} D$.} (This ``inversion" is a result of our definition of $\gerp_{N, \alpha}$ as $\alpha(\gerp_{N, 1})$, as opposed to $\alpha^{-1}(\gerp_{N, 1})$, made in order to conform with \cite{GorenReduction}.)

\

\id Suppose that we are given a finitely generated torsion free $\calO_L$-module $M$ on which $\calO_K$ acts as endomorphisms. Then $M_\CC = M \otimes_{\calO_L, \varphi_\CC}\CC$ is a finite dimensional vector space over $\CC$, which is an $\calO_K \otimes_\ZZ \CC = K \otimes_\QQ \CC$ module. We have then a decomposition
\begin{equation}\label{equation: decomposition of MC} M_\CC = \oplus_{\varphi\in \Hom(K, \CC)} M_\CC(\varphi) = \Hom_{\alpha \in G/H} M_\CC(\alpha),
\end{equation}
where $M_\CC(\varphi)$ is the eigenspace for the character $\varphi \colon K \arr \CC$, where, using the identifications $\Hom(K, \CC) = \Hom(K, N) = G/H$, we have let  $M_\CC(\alpha): = M_\CC(\varphi_\CC\circ \alpha)$. We assume that each eigenspace is either zero or one dimensional and so we get a subset \[\Phi \subset \Hom(K, N),\] corresponding to the non-trivial eigenspaces. We call $\Phi$ a ``CM type", although none of the fields appearing in our discussion so far needs to be CM.

On the other hand, we also have the finite dimensional $\Qpbar$-vector space $M_p: = M \otimes_{\calO_L, \varphi_p} \Qpbar$, grace of the homomorphism $\varphi_p\colon L \arr \Qpbar$, which is an $\calO_K \otimes_\ZZ \Qpbar = K \otimes_\QQ \Qpbar$-module. Since all the homomorphisms $K \arr \Qpbar$ factor as $K \arr N \overset{\varphi_p}{\Arr} \Qpbar,$ we have a decomposition, similar to the one in (\ref{equation: decomposition of MC}),
\begin{equation}\label{equation: decomposition of Mp} M_p = \oplus_{\varphi\in \Hom(K, \Qpbar)} M_p(\varphi)  = \oplus_{\alpha \in G/H} M_p(\alpha).\end{equation}
Moreover, for each $\alpha \in G/H$ there is a one dimensional  $L$-subspace $M_L(\alpha)$ of $M_L:= M \otimes_{\calO_L} L$ such that 
\[M_\CC(\alpha) = M_L(\alpha)\otimes_{L,\varphi_\CC} \CC, \qquad M_p(\alpha) = M_L(\alpha)\otimes_{L,\varphi_p} \Qpbar.\]
And so, in the obvious sense, $\Phi$ becomes a ``$p$-adic CM type" as well.

Now, the decomposition in (\ref{equation: decomposition of Mp}) can be packaged as follows: We have $\calO_K \otimes_\ZZ \ZZ_p = \oplus_{\gerp \vert p} \calO_{K_\gerp}$ and thus $K \otimes_\QQ \Qpbar = \oplus_{\gerp \vert p} (K_\gerp \otimes_{\QQ_p} \Qpbar)$, or
\[ K \otimes_\QQ \Qpbar = \oplus_{\alpha \in H\backslash G/D} (K_{\gerp_{K, \alpha}}\otimes_{\QQ_p} \Qpbar).\]
This decomposition induces a decomposition
\begin{equation}\label{equation: decomposition of Mp using ideals} M_p = \oplus_{\alpha \in H\backslash G/D} M_{p, \alpha}.\end{equation}
Note that, due to (\ref{equation: primes and embeddings}), the relation between (\ref{equation: decomposition of Mp}) and (\ref{equation: decomposition of Mp using ideals}) is (sic!)
\begin{equation}\label{equation: type and primes}
M_p(\alpha) \subseteq M_{p, \alpha^{-1}}.
\end{equation}

\subsection{The case of quartic fields and Dieudonn\'e modules} Let $K$ be a CM field of degree four over the rational numbers and let $A$ be a principally polarized abelian surface with complex multiplication by $\calO_K$, CM type $\Phi$, defined over a field $L$ and having everywhere good reduction. Let $K^\ast$ be the reflex field. We assume that $L$ contains a normal closure $N$ of $K$ and let $G = \Gal(N/\QQ)$. Thus, our notation conforms with the one in the previous section. 

Let $K^+$ be the totally real subfield of $K$. Let $p$ be a prime number. Our purpose is to determine the reduction $\bar A$ of $A$ modulo a prime ideal $\gerp_L$ of $L$. It follows from results of C.-F. Yu \cite{Yu} that the Dieudonn\'e module of $\bar A$ is determined uniquely by $\Phi$ and the prime decomposition of $p$ in $K$ (and not just in the case of surfaces). A fortiori, the Ekedahl-Oort strata in which it falls is determined. In the case of surfaces, the complete information is contained in two numbers
\[ a(\bar A) = \dim \Hom_{\fpbar}(\alpha_p, \bar A\otimes \fpbar), \qquad f(\bar A) = \log_p \sharp A[p](\fpbar),\]
the $a$-number and $f$-number. 

We the situation more explicit that in loc. cit., and provide a self-contained proof in our case. We will have several fields to consider $N$, $K$,  $K^\ast$ (the reflex field determined by $K$ and $\Phi$), and their totally real subfields $K^+$ and $K^{\ast+}$, respecively. The basic information is the factorization of $p$ in $N$. As above, we fix a prime ideal $\gerp  = \gerp_{N, 1} = \gerp_L\cap N$ of $N$. The decomposition of $p$ in each field is determined by the pair of subgroups $(I, D)$, where $I$ is the inertia group of $\gerp$ in $N$ and $D$ is its decomposition group. 
The pair of subgroups $(I, D)$ of $\Gal(N/\QQ)$ must satisfy the two restrictions:
\begin{itemize}\item $I \normal D$;
\item $D/I$ is a cyclic group. 
\end{itemize}
As explained above, having chosen $\gerp$, we may index the primes dividing $p$ in $N$ by coset representatives for $D$ in $G$. If these coset representatives are $a, b, c, \dots$ (so $G = aD \sqcup bD \sqcup cD \sqcup \dots$) then we write $p\calO_N = \gerp_{N, a}^e \gerp_{N, b}^e \gerp_{N, c}^e \cdots $, where $e = \sharp I$ and $\gerp_{N, a} := a(\gerp_N)$ (and in particular, $\gerp_{N, 1} = \gerp$). If the primes appearing in the decomposition of $p$ in $N$ are determined by $G/D$ then the primes appearing in the decomposition of $p$ in a subfield $N^H$ of $N$, corresponding to a subgroup $H$ of $G$, are determined by $H\backslash G /D$. As above, we shall denote such primes by $\gerp_{N^H, x}$ where $x$ is a representative for a double coset $HxD$. (This is consistent with the previous notation for $H = \{1\}$.) The following lemma is used to determine ramification in subfields. 
\begin{lem} \label{lemma: inertia in towers} Let $Q \subset B \subset N$ be three number fields where $N/Q$ is Galois with Galois group $G$. Let $B$ correspond to a subgroup $H$ of $G$. Let $\gerp_N$ be a prime ideal of $N$, $\gerp_B = \gerp_N \cap B$ and $\gerp_Q = \gerp_N \cap Q$. Let $I(\gerp_N)$ be the inertia group in $G$. Then,
\[e(\gerp_B/\gerp_Q) = [I(\gerp_N): I(\gerp_N) \cap H]\]
and
\[e(\gerp_N/\gerp_B) = \sharp\; I(\gerp_N) \cap H.\]
\end{lem}
\begin{proof} This is Lemma 3.3.29 in \cite{Cohen}. \end{proof}

\medskip

\id The main tool for studying the reduction $\bar{A} = A \pmod{\gerp_L}$ of the abelian surface $A$ is the following. Let $\DD$ be the Dieudonn\'e module of $\bar A[p]$ over $\fpbar$.\label{DD} The formalism of the previous section will be applied to the algebraic first de Rham cohomology of $A/L$, serving as $M$ in the previous section, which by base change gives us the complex de Rham cohomology of $A$ as well as the crystalline cohomology of $\bar{A}$, of which $\DD$ is the reduction modulo $p$. 

The $a$-number and $f$-number of $\bar{A}$ can of course be read from $\DD$. The Dieudonn\'e module has a decomposition relative to the $\calO_{K^+}$ action and a refined decomposition relative to the $\calO_K$ action. Using $\gerp_{K^+}$ to denote a prime ideal of $\calO_{K^+}$ above $p$ and similarly for $\gerp_K$, we have, by virtue of the decompositions $\calO_{K^+} \otimes \ZZ_p = \oplus_{\gerp_{K^+}} \calO_{{K^+}, \gerp_{K^+}}$, $\calO_K \otimes \ZZ_p = \oplus_{\gerp_K} \calO_{K, \gerp_K}$, induced decompositions
\[ \DD = \oplus_{\gerp_{K^+}} \DD(\gerp_{K^+}), \qquad \DD(\gerp_{K^+}) = \oplus_{\gerp_K\vert \gerp_{K^+}}\DD(\gerp_K). \]
Here each $\DD(\gerp_{K^+})$ is a self-dual Dieudonn\'e module of dimension $2e(\gerp_{K^+}/p)f(\gerp_{K^+}/p)$, which is then decomposed in Dieudonn\'e modules $\DD(\gerp_K)$ of dimension $e(\gerp_K/p)f(\gerp_K/p)$. On $\DD(\gerp_{K^+})$ there is an action of $\calO_{{K^+}, \gerp_{K^+}} \otimes \fpbar \cong \oplus_{\alpha} \fpbar[t]/(t^e)$, where the summation is over embeddings $\alpha$ of the maximal unramified subring $\calO_{{K^+}, \gerp_{K^+}}^{\rm ur}$ of $\calO_{{K^+}, \gerp_{K^+}}$ into $W(\fpbar)$ and $e = e(\gerp_{K^+}/p)$. There is a similar and compatible decomposition of $\calO_{K, \gerp_K} \otimes \fpbar$. These decompositions induce decompositions of the Dieudonn\'e modules $\DD(\gerp_{K^+}), \DD(\gerp_K)$, such that $\DD(\gerp_{K^+}) = \oplus_\alpha \DD(\gerp_{K^+}, \alpha), \DD(\gerp_K) = \oplus_\alpha \DD(\gerp_K, \alpha)$. $\DD(\gerp_{K^+}, \alpha)$ is a vector space of dimension $2e(\gerp_{K^+}/p)$, which is a free rank $2$  module over $\fpbar[t]/(t^e)$ on which $\calO_{{K^+}, \gerp_{K^+}} = \calO_{{K^+},\gerp_{K^+}}^{\rm ur}[\pi]$ acts via the map $\bar \alpha: 
\calO_{{K^+},\gerp_{K^+}}^{\rm ur} \arr \fpbar$ and $\pi$, which is an Eisenstein element, acts via $t$. A similar and compatible description is obtained for $\DD(\gerp_K, \alpha)$. Frobenius induces maps $\DD(\gerp_{K^+}, \alpha) \arr \DD(\gerp_{K^+}, \sigma \circ \alpha)$.

Implicit in our considerations is the identification of $\Hom(K, N)$ with $\Hom(K, \overline{\QQ}_p)$, where $N$ is a normal closure of $K$. 
This identification is done as discussed in detail above. In particular, we note that the subspace $\DD(\gerp_{K}, \alpha)$ is associated with the prime ideal $\gerp_{K, \alpha^{-1}}$. Since $H^0(\bar{A}, \Omega^1_{\bar{A}, \fpbar}) \subset \DD$, any $\alpha \in \Phi$ contributes $1$ to the dimension of the kernel of Frobenius on $\DD(\gerp_{K, \alpha^{-1}})$. This often allows us to conclude that $\Fr^2 = 0$ on $\DD$, which implies $a = 2, f = 0$ and, so, superspecial reduction.

\

Another useful tool to quickly decide some properties of the reduction is the following relation. Let $K^\ast$ be the reflex field defined by the CM type of the abelian variety under consideration and let $\Phi^\ast$ be the reflex type. Let $\gerp_{K^\ast, 1} = \gerp_{N, 1} \cap K^\ast$. Then some power of $\Norm_{\Phi^\ast}(\gerp_{K^\ast, 1})$ is equal to a power of $\Fr$, viewed as endomorphisms of the reduction. One can be more precise (see \cite{LangCM}), but we note that this suffices to calculate the $f$-number of the reduction. 

\subsection{$K$ cyclic Galois} In this case $K = N= K^\ast$. The Galois group is cyclic of order $4$, generated by $g$, say, where $g^2$ is complex conjugation. The CM types are either $\{1, g\}, \{g, g^2\}, \{g^2, g^3\}$ or $\{g^3, 1\}$. Since the reduction type does not depend on the way $K$ is embedded in $A$, namely we can compose with an automorphism $K \arr K$, we may assume that the CM type is $\{1, g\}$. The reflex CM field $K^\ast$ is $K$ and $\Phi^\ast = \{1, g^{-1}\}$. We have the following possibilities. 

\begin{center}
{\tiny
\begin{table}
\caption{Reduction in the cyclic case.}
\begin{tabular}{||p{.5 cm}|p{1.2 cm}|p{1.2 cm}|p{2.8 cm}|p{2.8 cm}|p{0.7 cm}|p{0.7 cm}|p{1 cm}||}
\hline\hline
& $I$ & $D$ & decomposition of $p$ in $K=K^\ast$ & decomposition of $p$ in $K^+$ & $a$ & $f$ & super-special?\\\hline
i& $\{1\}$ & $\{1\}$ & $\gerp_{K, 1}\gerp_{K, g}\gerp_{K, g^2}\gerp_{K, g^3}$ & $\gerp_{K^+, 1} \gerp_{K^+, g}$ & $0$ & $2$ & $\times$ \\\hline
ii & $\{1\}$ & $\{1,g^2\}$ & $\gerp_{K,1}\gerp_{K, g}$ & $\gerp_{K^+, 1} \gerp_{K^+, g}$ & $2$ & $0$ & $\surd$\\\hline
iii & $\{1\}$ & $G$ & $\gerp_{K, 1}$ & $\gerp_{K^+, 1}$ & $1$ & $0$ & $\times$\\\hline

\hline

iv & $\{1, g^2\}$ & $\{1, g^2\}$ & $\gerp_{K, 1}^2\gerp_{K, g}^2$ & $\gerp_{K^+, 1}\gerp_{K^+, g}$ & $2$ & $0$ & $\surd$\\\hline
v & $\{1, g^2\}$ & $G$ & $\gerp_{K, 1}^2$ & $\gerp_{K^+, 1}$ & $2$ & $0$ & $\surd$\\\hline
vi&$G$ & $G$ & $\gerp_{K, 1}^4$ & $\gerp_{K^+, 1}^2$ & $2$ & $0$ & $\surd$\\\hline\hline
\end{tabular}
\end{table}}
\end{center}

\

\id The unramified case appears in \cite{GorenReduction}, but we shall do one case to illustrate our method.
Consider the case ii. We have a decomposition 
\[\DD = \DD(\gerp_{K^+, 1}) \oplus \DD(\gerp_{K^+, g}), \]
and $\DD(\gerp_{K^+, i})$, $i=1, g$, is a two dimensional $\fpbar$-vector space that does not decompose further relative to the $\calO_{K^+}$ action. However, $\DD(\gerp_{K^+, i}) = \DD(\gerp_{K, i})$, because $\gerp_{K^+, i}$ is inert in $K$, and
\[ \DD(\gerp_{K, i}) = \DD(\gerp_{K, i}, \alpha) \oplus \DD(\gerp_{K, i}, \sigma \circ \alpha).\]
Frobenius takes $\DD(\gerp_{K, i}, \alpha)$ to $\DD(\gerp_{K, i}, \sigma \circ \alpha)$, and vice-versa. The CM type is $\{1, g\}$ and we note that $g$ switches $\gerp_{K, 1}$ and $\gerp_{K, g}$. This means that the cotangent space, or rather $H^0(A, \Omega^1_{A/\fpbar}) \otimes_{\fpbar,\sigma} \fpbar = \DD(\Ker \; \Fr)$, which is an $\calO_K$-module, is not contained completely in any of $\DD(\gerp_{K, i})$. Thus, Frobenius has a kernel on each of $\DD(\gerp_{K, i})$. It follows that $\Fr^2$ is zero on each $\DD(\gerp_{K, i})$ and hence on $\DD$ and that implies that $a(\bar A) = 2$, by a well known and elementary argument and $f(\bar A) = 0$.

In case iv we again have \[\DD = \DD(\gerp_{K^+, 1}) \oplus \DD(\gerp_{K^+, g}), \]
and $\DD(\gerp_{K^+, i})$ is a two dimensional $\fpbar$-vector space that does not decompose further relative to the $\calO_{K^+}$ action. However, $\DD(\gerp_{K^+, i}) = \DD(\gerp_{K, i})$ and $\DD(\gerp_{K, i})$ becomes a rank $1$ module over $\fpbar[t]/(t^2)$ by using the $\calO_K$ action and Frobenius is a module homomorphism. Once more, since $g$ permutes $\gerp_{K^+, 1}$ and $\gerp_{K^+, g}$, it follows that Frobenius has a kernel on each of $ \DD(\gerp_{K^+, i})$ and since the dimension of the kernel of Frobenius is two, it follows that the kernel Frobenius must be 
$(t) \oplus (t)\subset D(\gerp_{K, 1}) \oplus D(\gerp_{K, g}) $ and $\Fr^2 = 0$. 

In case v, after a similar analysis we reach the conclusion that $\DD = \fpbar[t]/(t^2) \oplus  \fpbar[t]/(t^2)$ and that Frobenius, which commutes with the $ \fpbar[t]/(t^2)$ structure, permutes the components. Whether the kernel of Frobenius is one of the components, or the submodule $(t) \oplus (t)$, we have $\Fr^2 = 0$ (in fact, taking into consideration the CM type we must have the kernel is $(t) \oplus (t)$, but this is not important at present). 

In case vi we conclude that $\DD = \fpbar[t]/(t^4)$ and that Frobenius acts as a $\fpbar[t]/(t^4)$-module homomorphism. It follows that the kernel of Frobenius, being an $\fpbar[t]/(t^4)$-module is $(t^2)$ and so is the image. Hence $\Fr^2 = 0$ again. 

\subsection{$K$ biquadratic} In this case $K = N$ is the compositum $K_1K_2$ where $K_i$ are quadratic imaginary fields. Let $K^+$ be the totally real subfield of $K$. Write the Galois group is $\{1, \alpha_1, \alpha_2, \beta\}$ where $K_i$ is fixed by $\alpha_i$ and $\beta$ is complex conjugation. We have the following diagram:

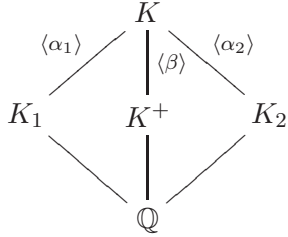
\[ \xymatrix{ & K\ar@{-}[dl]_{\langle \alpha_1 \rangle}\ar@{-}[d]^{\langle \beta \rangle}\ar@{-}[dr]^{\langle \alpha_2 \rangle} & \\ K_1\ar@{-}[dr] & K^+ \ar@{-}[d] & K_2\ar@{-}[dl] \\ & \QQ &}\]
The possible CM types are $\{1, \alpha_i\}, \{\beta, \alpha_i\}$ and twisting the action of $\calO_K$ by an automorphism we may assume the CM type is $\{1, \alpha_1\}$ or $\{1, \alpha_2\}$. The situation being symmetric we assume w.l.o.g that the CM type is $\{1, \alpha_1\}$. The reflex CM field is $K_1$ and the reflex CM type is $\{1\}$. In this case $A$ is isogenous to $E\otimes_\ZZ \calO_L$, or equivalently to $E \otimes_{K_1} K$, where $E$ is an elliptic curve with CM by $\calO_{K_1}$. Thus, $\bar A$ is ordinary if $p$ is split in $K_1$ and supersingular otherwise (and in that case one still needs to figure out its $a$ number). Now, $p$ is split in $K_1$ if and only if $\langle D, \alpha_1 \rangle \neq G$.

\begin{center}
{\tiny
\begin{table}
\caption{Reduction in the bi-quadratic case.}
\begin{tabular}{||p{.5 cm}|p{1.2 cm}|p{1.2 cm}|p{2.8 cm}|p{2.8 cm}|p{0.7 cm}|p{0.7 cm}|p{1 cm}||}
\hline\hline
& $I$ & $D$ & decomposition of $p$ in $K=K^\ast$ & decomposition of $p$ in $K^+$ & $a$ & $f$ & super-special?\\\hline
i& $\{1\}$ & $\{1\}$ & $\gerp_{K, 1}\gerp_{K, \alpha_1}\gerp_{K, \beta}\gerp_{K, \alpha_2}$ & $\gerp_{K^+, 1} \gerp_{K^+, \alpha_1}$ & $0$ & $2$ & $\times$ \\\hline

ii & $\{1\}$ & $\{1,\alpha_1\}$ & $\gerp_{K,1}\gerp_{K, \beta}$ & $\gerp_{K^+, 1} $ & $0$ & $2$ & $\times$\\\hline

iii & $\{1\}$ & $\{1, \beta\}$ & $\gerp_{K, 1}\gerp_{K, \alpha_1}$ & $\gerp_{K^+, 1}\gerp_{K^+, \alpha_1}$ & $2$ & $0$ & $\surd$\\\hline

iv & $\{1\}$ & $\{1, \alpha_2\}$ & $\gerp_{K, 1}\gerp_{K, \beta}$ & $\gerp_{K^+, 1}$ & $2$ & $0$ & $\surd$\\\hline

\hline

v & $\{1, \alpha_1\}$ & $\{1, \alpha_1\}$ & $\gerp_{K, 1}^2\gerp_{K, \beta}^2$ & $\gerp_{K^+, 1}^2$ & $0$ & $2$ & $\times$\\\hline

vi&$\{1, \alpha_1\}$ & $G$ & $\gerp_{K, 1}^2$ & $\gerp_{K^+, 1}^2$ & $2$ & $0$ & $\surd$\\\hline

vii & $\{1, \beta\}$ & $\{1, \beta\}$ & $\gerp_{K, 1}^2\gerp_{K, \alpha_1}^2$ & $\gerp_{K^+, 1}\gerp_{K^+, \alpha_1}$ & $2$ & $0$ & $\surd$\\\hline

viii&$\{1, \beta\}$ & $G$ & $\gerp_{K, 1}^2$ & $\gerp_{K^+, 1}$ & $2$ & $0$ & $\surd$\\\hline

ix & $\{1, \alpha_2\}$ & $\{1, \alpha_2\}$ & $\gerp_{K, 1}^2\gerp_{K, \beta}^2$ & $\gerp_{K^+, 1}^2$ & $2$ & $0$ & $\surd$\\\hline

x&$\{1, \alpha_2\}$ & $G$ & $\gerp_{K, 1}^2$ & $\gerp_{K^+, 1}^2$ & $2$ & $0$ & $\surd$\\\hline

xi&$G$ & $G$ & $\gerp_{K, 1}^4$ & $\gerp_{K^+, 1}^2$ & $2$ & $0$ & $\surd$\\\hline\hline
\end{tabular}
\end{table}}
\end{center}

\

\id Consider for example case vi. After the usual analysis we find that 
$\DD \cong \fpbar[t]/(t^2) \oplus \fpbar[t]/(t^2)$, where $\Fr$ is $\fpbar[t]/(t^2)$ $\sigma$-linear and switches the components. Its kernel is then either one of the components, or the submodule $(t) \oplus (t)$. In any case, $\Fr^2 = 0$ and so $a = 2$. Cases vii,  viii and x lead exactly to the same setting. 
 
In case ix, once again $\DD \cong \fpbar[t]/(t^2) \oplus \fpbar[t]/(t^2)$ but now $\Fr$ acts on each component separately. $\bar A$ is ordinary if the kernel of $\Fr$ is one of the components and is superspecial if the kernel is $(t) \oplus (t)$. Since ordinary is not possible, because $p$ is inert in $K_1$ (or, we can argue by using the CM type that Frobenius has a kernel on each component), we are in the superspecial case. 

In case xi we find that $\DD \cong \fpbar[t]/(t^4)$ and we must have that the kernel of Frobenius is the submodule $(t^2)$. It follows that $\Fr^2 = 0$.

\subsection{$K$ non-Galois} In this case the normal closure of $K$ is a Galois extension $N/\QQ$ of degree $8$ and Galois group $D_4$. As above, we view $N$ as embedded in $\CC$. $K$ is the fixed field of a non-central involution we call $x$. Let $y$ be an element of order $4$, then $y^2$ is complex conjugation and $xyx = y^{-1} = y^3$.
We identify $\Hom(K, \CC)$ with $\{1, y, y^2, y^3\}$ and the CM types are $\{1, y\}, \{y^2, y^3\}, \{1, y^3\}$ and $\{y^2, y^3\}$. We may twist the action of $K$ by complex conjugation and so assume that the CM type is $\{1, y\}$ or $\{1, y^3\}$. If it is $\{1, y^{-1}\}$ we can change the presentation of our group by using the generator $y^{-1}$ instead of $y$. We can therefore assume that $K$ is fixed by $x$, the Galois group is $\langle x, y| x^2, y^4, xyxy\rangle$ and the CM type is $\{1, y\}$. The reflex CM field $K^\ast$ is then fixed by $\{1, xy^3\}$ (follow the recipe in \cite[Ch. 1, Theorem 5.1]{LangCM}) and the reflex CM type is $\{1, y^{-1}\}$.

We have the following diagrams of fields and subgroups:

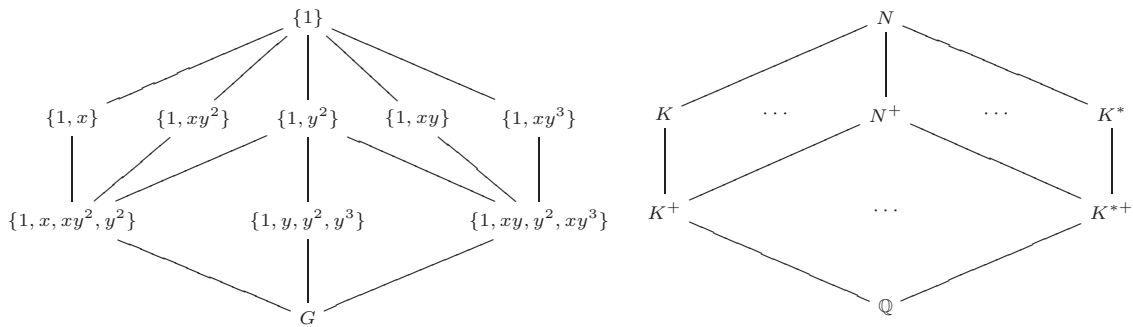
\[{\tiny
\xymatrix@C=0.5pt{ & & \{1\}& & \\ \{1, x\}\ar@{-}[urr] & \{1, xy^2\}\ar@{-}[ur] & \{1, y^2\}\ar@{-}[u] & \{1, xy \}\ar@{-}[ul] & \{1, xy^3\}\ar@{-}[ull] 
\\ \{1, x, xy^2, y^2\}\ar@{-}[urr]\ar@{-}[ur]\ar@{-}[u] & & \{1, y, y^2, y^3 \}\ar@{-}[u] && \{1, xy, y^2, xy^3\}\ar@{-}[ull]\ar@{-}[ul]\ar@{-}[u] \\
&& G\ar@{-}[urr]\ar@{-}[u]\ar@{-}[ull] &&
}
\quad
\xymatrix{& & N & & \\ K\ar@{-}[urr] & \cdots& N^+\ar@{-}[u] & \cdots & K^\ast \ar@{-}[ull]\\ K^+\ar@{-}[u]\ar@{-}[urr] & &\cdots & & K^{\ast +} \ar@{-}[u]\ar@{-}[ull]\\ & & \QQ\ar@{-}[urr]\ar@{-}[ull] & &}}
\]

\

\id
The analysis of the reduction of $A$ proceeds along the same lines as above. Namely, one considers the decomposition of the Dieudonn\'e module as a module over $\calO_K \otimes \fpbar$ and the induced action of Frobenius, which is $1\otimes \sigma$-linear, so to say. In most cases, this suffices to determine the $a$ and $f$ numbers, but in certain cases one needs to decide between two possibilities, and there the CM type matters. The interpretation of the CM type mod $p$ is done through the formalism of \S\ref{subsec: embeddings and primes}.

For example, referring to the table, in case viii we find that $\DD \cong \fpbar[t]/(t^2) \oplus \fpbar[t]/(t^2)$ and Frobenius acts $\sigma$-$\fpbar[t]/(t^2)$ linearly (meaning, it acts $\sigma$-linearly on $\fpbar$ and commutes with $t$) on each component. The kernel, a-priori could be one of the components or the submodule $(t) \oplus (t)$. Taking the CM type into consideration, we see that Frobenius has a kernel in each component and so its kernel is $(t) \oplus (t)$. It follows that $\Fr^2 = 0$. Case x is the same.

Case ix is easier as in this case $\DD \cong \fpbar[t]/(t^2) \oplus \fpbar[t]/(t^2)$, where $\Fr$ is acting $\sigma$-$\fpbar$-linearly, but permutes the components. The kernel is either one of the components or the submodule $(t) \oplus (t)$ and, regardless, $\Fr^2 = 0$. Case xi is the same. 

\begin{landscape}
{\tiny
\begin{center}
\begin{table}
\caption{Reduction in the non Galois case.}
\renewcommand{\arraystretch}{1.5}
\begin{tabular}{||p{.5 cm}|p{1 cm}|p{1 cm}|p{2.4 cm}|p{2.4 cm}|p{1.8 cm}|p{3 cm}|p{1.8 cm}|| p{1.4 cm}||p{0.3 cm}|p{0.3 cm}|p{0.5 cm}||}
\hline\hline
& $I$ & $D$ & decomposition of $p$ in $N$ &decomposition of $p$ in $K$ & decomposition of $p$ in $K^+$ & decomposition of $p$ in $K^\ast$&decomposition of $p$ in $K^{\ast +}$ & $N_{\Phi^\ast}(\gerp_{K^\ast, 1})$& $a$ & $f$ & ss?\\\hline

i& $\{1\}$ & $\{1\}$ & $\prod_{\alpha\in G}\gerp_{N, \alpha}$ & $\gerp_{K, 1}\gerp_{K, y}\gerp_{K, y^2}\gerp_{K, y^3}$ &$\gerp_{K^+, 1}\gerp_{K^+, y}$&$\gerp_{K^\ast, 1}\gerp_{K^\ast, y}\gerp_{K^\ast, y^2}\gerp_{K^\ast, y^3}$ &$\gerp_{K^{^\ast+}, 1}\gerp_{K^{^\ast+}, y}$ & $\gerp_{K, 1}\gerp_{K, y^3}$ & $0$ & $2$ & $\times$ \\\hline

ii & $\{1\}$ & $\langle x\rangle$ & $\gerp_{N,1}\gerp_{N,y}\gerp_{N,y^2}\gerp_{N,y^3}$ & $\gerp_{K, 1}\gerp_{K, y}\gerp_{K, y^2}$ & $\gerp_{K^+, 1}\gerp_{K^+, y}$ & $ \gerp_{K^\ast, 1} \gerp_{K^\ast, y^2}$ & $\gerp_{K^{\ast +},1}$ & $\gerp_{K, 1}^2\gerp_{K, y}$ & $1$ & $1$ & $\times$\\\hline

iii & $\{1\}$ & $\langle xy\rangle$ & $\gerp_{N, 1}\gerp_{N, y}\gerp_{N, y^2}\gerp_{N, y^3}$ & $\gerp_{K, 1}\gerp_{K, y^2}$ & $\gerp_{K^+, 1}$& $\gerp_{K^{\ast}, 1}\gerp_{K^{\ast}, y}\gerp_{K^{\ast}, y^3}$& $\gerp_{K^{\ast+}, 1}\gerp_{K^{\ast+}, y}$&  $p$ &$2$ & $0$ & $\surd$\\\hline

iv & $\{1\}$ & $\langle  xy^2\rangle$ & $\gerp_{N, 1}\gerp_{N, y}\gerp_{N, y^2}\gerp_{N, y^3}$& $\gerp_{K, 1}\gerp_{K, y}\gerp_{K, y^3}$& $\gerp_{K^+, 1}\gerp_{K^+, y}$ &  $\gerp_{K^\ast, 1}\gerp_{K^\ast, y}$& $\gerp_{K^{\ast+},1}$ &  $\gerp_{K, 1}\gerp_{K, y^3}^2$ & $1$ & $1$ & $\times$\\\hline

v & $\{1\}$ & $\langle xy^3\rangle$ & $\gerp_{N, 1}\gerp_{N, y}\gerp_{N, y^2}\gerp_{N, y^3}$ & $\gerp_{K, 1}\gerp_{K, y^2}$ &$ \gerp_{K^+, 1}$& $\gerp_{K^\ast, 1}\gerp_{K^\ast, y}\gerp_{K^\ast, y^2} $& $\gerp_{K^{\ast +}, 1}\gerp_{K^{\ast +}, y}$&  $\gerp_{K, 1}^2$&$0$ & $2$ & $\times$\\\hline

vi&$\{1\}$ & $\langle y^2\rangle$ & $\gerp_{N, 1}\gerp_{N, x}\gerp_{N, y}\gerp_{N, xy}$ & $\gerp_{K, 1}\gerp_{K, y}$ & $\gerp_{K^+, 1}\gerp_{K^+, y}$&  $\gerp_{K^\ast, 1}\gerp_{K^\ast, y}$ &  $\gerp_{K^{\ast+}, 1}\gerp_{K^{\ast+}, y}$   &  $p$ & $2$ & $0$ & $\surd$\\\hline

vii & $\{1\}$ & $\langle  y\rangle$ & $\gerp_{N, 1}\gerp_{N, x}$ & $\gerp_{K, 1}$ & $\gerp_{K^+, 1}$& $\gerp_{K^\ast, 1}$ & $\gerp_{K^{\ast+}, 1}$ & $p^2$ & $1$ & $0$ & $\times$\\\hline

\hline

viii &$\langle y^2 \rangle $ &$\langle y^2 \rangle $& $\gerp_{N, 1}^2\gerp_{N, x}^2\gerp_{N, y}^2\gerp_{N, xy}^2$ & $\gerp_{K, 1}^2\gerp_{K, y}^2$ & $\gerp_{K^+, 1}\gerp_{K^+, y}$ & $\gerp_{K^\ast, 1}^2\gerp_{K^\ast, y}^2$ &$\gerp_{K^{\ast+}, 1}\gerp_{K^{\ast+}, y} $ & $\gerp_{K, 1}\gerp_{K, y}$&$2$ & $0$ & $\surd$\\\hline

ix &$\langle y^2 \rangle $ &$\langle y \rangle $& $\gerp_{N, 1}^2 \gerp_{N, x}^2$ & $\gerp_{K, 1}^2$ & $\gerp_{K^+, 1}$ & $\gerp_{K^\ast, 1}^2$& $\gerp_{K^{\ast+}, 1}$& $p$ &$2$ & $0$ & $\surd$\\\hline

x &$\langle y^2 \rangle $  &$\langle x, y^2 \rangle $&$\gerp_{N, 1}^2\gerp_{N, y}^2$ & $\gerp_{K, 1}^2 \gerp_{K, y}^2$ & $\gerp_{K^+, 1} \gerp_{K^+, y}$ & $\gerp_{K^\ast, 1}^2$ & $\gerp_{K^{\ast+}, 1}$ & $p$ &$2$ & $0$ & $\surd$\\\hline

xi &$\langle y^2 \rangle $ &$\langle xy, y^2 \rangle $ &$\gerp_{N, 1}^2\gerp_{N, y}^2$ & $\gerp_{K, 1}^2$ & $\gerp_{K^+, 1}$ & $\gerp_{K^\ast, 1}^2\gerp_{K^\ast, y}^2$ & $\gerp_{K^{\ast+}, 1}\gerp_{K^{\ast+}, y}$ &$p$ & $2$ & $0$ & $\surd$\\\hline

xii &$\langle x \rangle$ & $\langle x \rangle$ &$\gerp_{N, 1}^2\gerp_{N, y}^2\gerp_{N, y^2}^2\gerp_{N, y^3}^2$ &$\gerp_{K, 1}\gerp_{K, y}^2\gerp_{K, y^2}$ & $\gerp_{K^+, 1}\gerp_{K^+, y}$ & $\gerp_{K^\ast, 1}^2\gerp_{K^\ast, y^2}^2$ & $\gerp_{K^{\ast+}, 1}^2$ & $\gerp_{K, 1}\gerp_{K, y}$ &$1$ & $1$ & $\times$\\\hline

xiii&$\langle x \rangle$ & $\langle x, y^2 \rangle$ & $\gerp_{N, 1}^2\gerp_{N, y}^2$ & $\gerp_{K, 1}\gerp_{K, y}^2$& $\gerp_{K^+, 1}\gerp_{K^+, y}$ & $\gerp_{K^\ast, 1}^2 $& $\gerp_{K^{\ast+}, 1}^2$& $p$ &$2$ & $0$ & $\surd$\\\hline

xiv& $\langle xy^2 \rangle$ & $\langle x y^2 \rangle$ & $\gerp_{N, 1}^2\gerp_{N, y}^2\gerp_{N, y^2}^2\gerp_{N, y^3}^2$ & $\gerp_{K, 1}^2\gerp_{K, y}\gerp_{K, y^3}$&$\gerp_{K^+, 1}\gerp_{K^+, y}$ & $\gerp_{K^\ast, 1}^2\gerp_{K^\ast, y}^2$ & $\gerp_{K^{\ast+}, 1}^2$ & $\gerp_{K, 1}\gerp_{K, y^3} \triangle $ & $1$ & $1$ &$\times$\\\hline

xv& $\langle xy^2 \rangle$ & $\langle x, y^2 \rangle$ & $\gerp_{N, 1}^2\gerp_{N, y}^2$ & $\gerp_{K, 1}^2\gerp_{K, y}$& $\gerp_{K^+, 1}\gerp_{K^+, y}$ & $\gerp_{K^\ast, 1}^2$& $\gerp_{K^{\ast+}, 1}^2$& $p$ & $2$ & $0$ & $\surd$\\\hline

xvi&$\langle xy \rangle$ &$\langle xy \rangle$& $\gerp_{N, 1}^2\gerp_{N, y}^2\gerp_{N, y^2}^2\gerp_{N, y^3}^2$ & $\gerp_{K, 1}^2 \gerp_{K, y^3}^2$ & $\gerp_{K^+, 1}^2$ & $\gerp_{K^\ast, 1}^2\gerp_{K^\ast, y}\gerp_{K^\ast, y^3}$ & $\gerp_{K^{\ast+}, 1}\gerp_{K^{\ast+}, y}$ & $\gerp_{K, 1}\gerp_{K, y^3}\triangle$  & $2$ & $0$ & $\surd$\\\hline

xvii&$\langle xy \rangle$& $\langle xy, y^2 \rangle$& $\gerp_{N, 1}^2\gerp_{N, y}^2$ & $\gerp_{K, 1}^2$ & $\gerp_{K^+, 1}^2$ & $\gerp_{K^\ast, 1}^2\gerp_{K\ast, y}$ & $\gerp_{K^{\ast+}, 1}\gerp_{K^{\ast+}, y}$ & $p$ & $2$ & $0$ & $\surd$\\\hline

xviii&$\langle xy^3 \rangle$ &$\langle xy^3 \rangle$& $\gerp_{N,1}^2\gerp_{N,y}^2\gerp_{N,y^2}^2\gerp_{N,y^3}^2$ & $\gerp_{K,1}^2\gerp_{K,y}^2$ & $\gerp_{K^+,1}^2$ & $\gerp_{K^\ast,1}\gerp_{K^\ast,y}^2\gerp_{K^\ast,y^2}$ & $\gerp_{K^{\ast+},1}\gerp_{K^{\ast+},y}$ & $\gerp_{K, 1}^2\triangle$ & $2$& $0$ & $\surd$\\\hline

xix&$\langle xy \rangle$& $\langle xy, y^2 \rangle$& $\gerp_{N, 1}^2\gerp_{N, y}^2$ & $\gerp_{K, 1}^2$ & $\gerp_{K^+, 1}^2$ & $\gerp_{K^\ast, 1}^2\gerp_{K^\ast, y}$ & $\gerp_{K^{\ast+}, 1}\gerp_{K^{\ast+}, y}$ & $p$ & $2$& $0$ & $\surd$\\\hline

xx&$\langle y \rangle$& $\langle y\rangle$& $\gerp_{N, 1}^4\gerp_{N, x}^4$ & $\gerp_{K, 1}^4$ & $\gerp_{K^+, 1}^2$ & $\gerp_{K^\ast, 1}^4$ & $\gerp_{K^{\ast+}, 1}^2$ & $\gerp_{K, 1}^2$ & $2$& $0$ & $\surd$\\\hline

xxi&$\langle y \rangle$& $G$& $\gerp_{N, 1}^4$ & $\gerp_{K, 1}^4$ & $\gerp_{K^+, 1}^2$ & $\gerp_{K^\ast, 1}^4$ & $\gerp_{K^{\ast+}, 1}^2$ & $p$ & $2$& $0$ & $\surd$\\\hline

xxii&$\langle x, y^2 \rangle$&$\langle x, y^2 \rangle$&$\gerp_{N, 1}^4\gerp_{N, y}^4$ & $\gerp_{K, 1}^2\gerp_{K, y}^2$ & $\gerp_{K^+, 1}\gerp_{K^+, y}$ & $\gerp_{K^\ast, 1}^4$ & $\gerp_{K^{\ast+}, 1}^2$ & $\gerp_{K, 1}\gerp_{K, y}$ & $2$& $0$ & $\surd$\\\hline

xxiii&$\langle x, y^2 \rangle$&$G$& $\gerp_{N, 1}^4$ & $\gerp_{K, 1}^2$ & $\gerp_{K^+, 1}$ & $\gerp_{K^\ast, 1}^4$ & $\gerp_{K^{\ast+}, 1}^2$ & $p$ & $2$& $0$ & $\surd$\\\hline

xxiv&$\langle xy, y^2 \rangle$&$\langle xy, y^2 \rangle$&$\gerp_{N, 1}^4\gerp_{N, y}^4$ & $\gerp_{K, 1}^4$ & $\gerp_{K^+, 1}^2$ & $\gerp_{K^\ast, 1}^2\gerp_{K^\ast, y}^2$ & $\gerp_{K^{\ast+}, 1}\gerp_{K^{\ast+}, y}$ & $\gerp_{K, 1}^2$ & $2$& $0$ & $\surd$\\\hline

xxv&$\langle xy, y^2 \rangle$&$G$& $\gerp_{N, 1}^4$ & $\gerp_{K, 1}^4$ & $\gerp_{K^+, 1}^2$ & $\gerp_{K^\ast, 1}^2$ & $\gerp_{K^{\ast+}, 1}$ & $p$ & $2$& $0$ & $\surd$\\\hline

xxvi& $G$ & $G$ & $\gerp_{N, 1}^4$ & $\gerp_{K, 1}^4$ & $\gerp_{K^+, 1}^2$ & $\gerp_{K^\ast, 1}^4$ & $\gerp_{K^{\ast+}, 1}^2$ & $p$ & $2$& $0$ & $\surd$\\\hline\hline
\end{tabular}
\end{table}
\end{center}
}
\end{landscape}

\subsection{Examples} Take a curve $C$ of genus $2$ over $\QQ$ (to simplify). Given a prime $p$ at which $C$ has good reduction $\bar C$, one has a simple method of writing down the Hasse Witt matrix $M$ of $\bar A = Jac(\bar C)$ and so deciding the $a$ number and $f$ number of $\bar A$: The $f$ number is the rank of $M^{(p)}M$ and the $a$-number is the co-rank of $M$. In general it is hard to decide the reduction type by examining $M$, but in certain cases we can do that and compare our results with the results above when $A = \Jac(C)$ has complex multiplication.

Let $C: y^2 = f(x)$, where $f(x) = x^5 + a_4x^4 + \dots + a_0$ be a hyperelliptic curve and write
\[ f(x)^{(p-1)/2} = \sum_{j \geq 0} c_j x^j.\]
Then the Hasse-Witt matrix $M$ is given by
\[ \begin{pmatrix} c_{p-1} & c_{p-2} \\ c_{2p-1} & c_{2p-2}
\end{pmatrix},\]
and $M^{(p)}$ is
\[ \begin{pmatrix} c_{p-1}^p & c_{p-2}^p \\ c_{2p-1}^p & c_{2p-2}^p
\end{pmatrix}  .\]
Exactly the same recipe works if $f(x)$ is a sextic. See \cite[p. 129]{IKO}
\subsubsection{} Let $C: y^2 = x^5 + 1$. The curve has good reduction outside $2\cdot 5$. The Jacobian has complex multiplication by $\QQ(\zeta_5)$ and the automorphism group of the curve in characteristic zero is $\mu_{10}$. The coefficient of $x^n$ in $f(x)^{(p-1)/2} $ is $0$ if $5\nmid n$ and, for $n$ not larger than $5(p-1)/2$  such that $5 \vert n$, is $\binom{(p-1)/2}{n/5}$. We divide the analysis to several cases:
\begin{itemize}
\item If $p \equiv 1 \pmod{5}$, $M =  \left(\begin{smallmatrix} \binom{(p-1)/2}{(p-1)/5} & 0 \\ 0 & \binom{(p-1)/2}{(2p-2)/5}
\end{smallmatrix} \right)$ has rank $2$ and we conclude that $\bar A$ is ordinary. Note that $p$ splits completely in this case. Namely we are in case i of the cyclic Galois case.
\item If $p \equiv 2 \pmod{5}, p> 2$, $M =  \left(\begin{smallmatrix} 0 & \binom{(p-1)/2}{(p-2)/5} \\ 0 & 0\end{smallmatrix} \right)$ has rank $1$ and $M^{(p)} M = 0$. Thus, $f = 0$ and $a = 1$. This is a supersingular, but not superspecial reduction, in accordance to case iii.
\item If $p \equiv 3 \pmod{5}$, $M = \left(\begin{smallmatrix} 0 & 0 \\ \binom{(p-1)/2}{(2p-1)/5} & 0\end{smallmatrix} \right)$ has rank $1$ and $M^{(p)} M = 0$. Thus, $f = 0$ and $a = 1$. This is a supersingular, but not superspecial reduction, in accordance to case iii again.
\item If $p \equiv -1 \pmod{5}$, $M =  \left(\begin{smallmatrix} 0 & 0 \\ 0 & 0\end{smallmatrix} \right)$ has rank $0$ and we have superspecial reduction, in accordance with case ii.
\item $p=5$.  It follows from Igusa's classification of genus 2 curves with many automorphisms \cite[\S 8]{IgusaArithmeticModuli} that the reduction of a stable model of $y^2 = x^5 + 1$ modulo 5 is isomorphic, possibly after base change, to the curve $y^2 = f(x)$, where $f(x) = x(x-1)(x+1)(x-2)(x+2)$. That is, since the characteristic is $5$,  $f(x) = x^5 - x$. Then $f(x)^2 = x^{10} -2x^6 + x^2$ and the Hasse-Witt matrix is the zero matrix, giving us superspecial reduction. This agrees with case v.
\item In characteristic 2, Igusa's classification gives us the model $y^2 - y = x^5$. According to our table, since we are in case iii, this curve should be supersingular, but not superspecial. The fact that the curve is supersingular, which in genus $2$ is equivalent to $f = 0$, follows from the theory of Artin-Schreier coverings, c.f. \cite[Lemma 2.6]{PZ}. According to \cite[Theorem 3.3]{IKO} there are no superspecial non-singular curves of genus 2 in characteristic 2. Therefore, we have supersingular and not superspecial reduction. 
\end{itemize}

\subsubsection{} Consider the curve $y^2 = -8x^6 - 64x^5 + 1120x^4 + 4760x^3 - 48400x^2 + 22627x - 91839$, which has complex multiplication by the ring of integers of $K = \QQ(\sqrt{-65 + 26\sqrt{5}})$ by \cite{vanW}. The field is a cyclic Galois extension with a totally real field $K^+ = \QQ(\sqrt{5})$. Its discriminant is $5^3\cdot 13^2$.
The prime $5$ decomposes as $\gerp_{K^+}^2 = \gerp_{K}^4$ and belongs to case vi, the prime $13$ decomposes as $\gerq_{K^+}=\gerq_{K}^2$ and belongs to case v. In any case, we have superspecial reduction. And,
indeed, in both cases one finds that the Hasse-Witt matrix is identically zero modulo the corresponding prime.
For example, for $p=5$ we have $f(x)^2 = 64 x^{12} + 1024 x^{11} - 13824 x^{10} - 219520 x^9 + 1419520 x^8 + 16495568 x^7 - 87185232 x^6 - 398328128 x^5 + 2352249680 x^4 - 3064600880 x^3 + 9401996329 x^2 - 4156082106 x + 8434401921$ and the Hasse-Witt matrix is 
$\left(\begin{smallmatrix}2352249680&-3064600880\\-219520&1419520\end{smallmatrix} \right)\equiv 0 \pmod{5}$.


\subsubsection{Cases (v) and (vi) in Table 3.3.1 for Galois cyclic fields}
Examples 1 and 2 below demonstrate cases (v) and (vi) in the table for Galois cyclic fields.
For both, we take the Galois cyclic field $K=\QQ[x]/(x^4 + 238x^2 + 833)$, with real quadratic subfield $\QQ(\sqrt{17})$.
It can be constructed by adjoining $\sqrt{-119+28\sqrt{17}}$ to $\QQ$.
The class number of $K$ is $2$ and the field discriminant is $7^2 17^3$.

The three Igusa Class polynomials are:



{\tiny{
\begin{multline*}
h_1(x) = x^2 + \frac{3^{16} \cdot 11 \cdot 163 \cdot 4801 \cdot 712465984819 \cdot 152160175753014902257305649143422239021984895543}{
2^{23}\cdot7^{6}\cdot 43^{12}\cdot 179^{12}}x  \\  -\frac{3^{30} \cdot 62273^5 \cdot 173166943^5}{2^{22}\cdot 7^{12}\cdot 43^{12}\cdot 179^{12}}
\end{multline*}}}

{\tiny{
\begin{multline*}h_2(x) = x^2 + \frac{3^{11} \cdot 5 \cdot 967 \cdot 199763665249568296384949088855973069605073}{2^{9} \cdot 7^{3} \cdot 43^{8} \cdot 179^{8}}x 
-\frac{3^{22} \cdot 5^2 \cdot 19^2 \cdot 191 \cdot 62273^3 \cdot 173166943^3}{ 2^{6}\cdot 7^{8}\cdot 43^{8}\cdot 179^{8}}\end{multline*}}}

{\tiny{
\begin{multline*}h_3(x) = x^2 + \frac{3^9 \cdot 1823 \cdot 8197340996395223625771218888046149724668749}{ 2^{11}\cdot 7^{3}\cdot 43^{8}\cdot 179^{8}} x \\
- \frac{3^{18} \cdot 359 \cdot 1667 \cdot 1811 \cdot 2281229974265082675220366841972155717537}{ 2^{10}\cdot 7^{8}\cdot 43^{8}\cdot 179^{8}}\end{multline*}}}

{\bf Example 1} (Case v) The prime $7$ decomposes in $K$ as the square of an inert prime with inertia degree $2$. Modulo $7$ the class polynomials reduce badly, since $7$ is in the denominator.
The two CM curves each reduce to a product of elliptic curves with product polarization modulo $7$, and the Galois action takes one curve to the other. Both have superspecial reduction.

{\bf Example 2} (Case vi) The prime $17$ is totally ramified in K.
Modulo $17$ the reduction of the Igusa class polynomials is: 
\[h_1(x) = (x + 13)^2 \pmod{17}, \quad h_2(x) = (x + 12)^2 \pmod{17}, \quad h_3(x) = (x + 2)^2 \pmod{17}.\]

Taking the absolute Igusa invariants $[i_1,i_2,i_3]=[-13, -12, -2]$ modulo $17$, we recover a $4$-tuple of Igusa-Clebsch invariants $[I_2,I_4,I_6,I_{10}] = [1,14,8,13]$ via the formulas:
$I_2 = 1$, $I_{10} = I_2^5/i_1$, $I_4 = i_2 \cdot I_{10}/I_2^3$, $I_6 = i_3 \cdot I_{10}/I_2^2$.
Using Magma's implementation of Mestre's algorithm, we obtain a genus $2$ curve $C: y^2 = x^6 + 16$ with these invariants over $\FF_{17}$.
Taking $f(x) = x^6 + 16 \pmod{17}$, we compute the $(p-1)/2 = 8^{th}$ power and compute the Hasse-Witt matrix.
The only non-zero coefficients of $f$ are for terms whose degree is $0 \pmod{6}$, so the Hasse-Witt Matrix is zero and the reduction is superspecial.


\subsubsection{Cases (xii), (xiv), (xvii) and (xix) in Table 3.5.1 for non-Galois fields}
In Examples 3 and 4 below we deal with cases (xii) and (xiv) (Example 4) and cases (xvii) and (xix) (Example 3) 
in the table for non-Galois fields.
We work with a non-Galois quartic CM field, given by $K=\QQ[x]/(x^4 + 134x^2 + 89)$
with reflex field given by $K^*=\QQ[x]/(x^4 + 268x^2 + 17600)$.
The class number of $K$ is $4$ and the discriminant is $2^4  11^2  89$.

For typographical reasons we list the class polynomials in modified form. To get the class polynomials $h_i(x)$ from the polynomials $h_i^\ast(x)$ listed below, divide by the leading coefficient in each case.

{\begin{landscape}
{\tiny{
\begin{multline*}
h_1^\ast = 467861685008274198315825008595700654800896648612454642253063065763346063674621433392530889250338077545166015625\cdot x^8 + \\ 
555449149845517528201830854630774702288460206836540032347806689557044680668121067380364116957025544252246618270874023437500000 \cdot x^7 + \\
184033686764733003916214393323122175930726657165358821209777937427864516170252466041678857284459287923047251104058697819709777832031250000000 \cdot x^6 - \\
18532528196713610966248735059989496921744294218655209931046129134295796866360219594850246272816241093321185745310256539534492913064849853515625000000000 \cdot x^5 - \\
149517615773862216077075501785526135664390163794144774072964515539112873485\\177946865799467841097175021951747554258225537368778725711443911431489057223000000000000 \cdot x^4 - \\
274500212787786320363174922451987418656288895564254168526715333725750375585\\56384959106460164764658831329000379776543259072657572814515177392240269322428312127339426217984 \cdot x^3 - \\
1297531069082446204942804872389223658522300816123923235450253734042421899655\\805930017719515089898192145168479582847645622244801024566788907131236811092595248135449429095219200000 \cdot x^2 + \\
75816198120430164253210000030177809833640567679756337667150872417366816696541\\96164395918854519565530006696018114342720043906982109112415240533721325054782428254517780807680000000000 \cdot x- \\
166561076259218874524380391618627812459200629952377728540602961024102700278352\\475504124640248501826031024603695578842862255022395446214265265991340473323825199368431179137024000000000000000
\end{multline*}

\begin{multline*}
h_2^\ast = 122620993224533990854266979572168589900407195091247558593750000 \cdot x^8 + \\
7485929269991071436519019319213472872675919432653818688502883911132812500000 \cdot x^7 +\\
127911590573429429764061252422626647909635036233546648623604176763112582318377685546875000 \cdot x^6 - \\
432801469302398970120563934143486307948625635434432325277226168869895543943151085803437889746093750 \cdot x^5 - \\
70989757220371345897539040783507004210240989969604889311893737913941059181926255773255664903749716042965625 \cdot x^4 + \\
141214583953749258746190038912978215937828708913783023311482635400978729488802928890913822935905126587220991510912 \cdot x^3 - \\
324730974425347314917488050857215655038539099494418111993188797560893578833446457406316467999877717129727990440513280000 \cdot x^2 + \\
2878258800484146973496313274835799307245769049641717521354166360884626643674126222273800205511215767305294130902374400000000 \cdot x - \\
8757766750510816031715743862941216509133894670889936799087401365869157766568945511079523707916023330470602373237659904000000000000
\end{multline*}}}
\end{landscape}}

{\begin{landscape}
{\tiny{
\begin{multline*}
h_3^\ast = 31390974265480701658692346770475159014504241943359375000000000000 \cdot x^8 + \\
493348323893392512322187882201836480657190909721221154566235351562500000000000 \cdot x^7 + \\
2168443965418989986038492688067403045710941961035989372240912887706777245531152343750000000 \cdot x^6 - \\
2302525585957788818152082352653829396337430793844914883168947610481921539535736987830563608984375000 \cdot x^5 - \\
152380762091374020434799837277117715974184875809865052975561585447684346918113356183254740900302324932628125 \cdot x^4 + \\
101261095338271190490530687171870069034863165796195122032131006101920226887769776012517741443429566675432329475648 \cdot x^3 - \\
82394230890068050809147635660557623685629965618893227966125666811860080407429030642077770538444660191227910486571712000 \cdot x^2 - \\
1926409131567661484196961498816000531174411060031335222220813078857194020390371535396618981891327471335621587834624000000 \cdot x - \\
1870374669751414608923737345184889994628232369194056109733545677638411383291159282002508930826987969131561815775577216000000000
\end{multline*}}}
\end{landscape}
}
{\bf Example 3} (cases xvii, xix) The prime decomposition of $11$ in $K$ is such that it is ramified in $K^+$ and the prime above it in $K^+$ is inert in $K$.  
Further, $11$ is split in ${K^*}^+$, and mixed in $K^*$ (one degree-one prime ideal with ramification index $2$, and one unramified prime ideal of degree $2$).
The prime $11$ appears in the denominator, so at least one of the curves with CM by $K$ is superspecial.

{\bf Example 4} (cases xii and xiv) The prime decomposition of $89$ in $K$ is mixed: one ramified prime of degree $1$ and two unramified primes of degree $1$. It is  split in $K^+$, ramified in $K^{\ast, +}$, and that prime in $K^{\ast, +}$ then splits in $K^\ast$.
Modulo $89$ the class polynomials factor as a product of the squares of two degree-$2$ polynomials:
$$h_1= (x^2 + 17x + 9)^2(x^2 + 18x + 25)^2 \pmod{89}$$
$$h_2=(x^2 + 37x + 67)^2(x^2 + 69x + 57)^2 \pmod{89}$$
$$h_3=(x^2 + 83x + 83)^2(x^2 + 85x + 45)^2 \pmod{89}.$$

Note that in this case, it is not obvious from the polynomials how to match up roots of the three 
polynomials to form triples of Igusa invariants.  A common approach has been to use the knowledge of the CM field to determine the possible group orders of the Jacobian of the curve, and then to run through all possible triples of roots of these polynomials until the correct triples and the corresponding curves are found.  In the case that the prime $p$ splits completely in the field $K$ (case (i) in Table 3.5.1), a method for determining the possible group orders was given in~\cite{Weng} and~\cite[Proposition 4]{EL}, and the resulting CM curves constructed there were indeed ordinary.  For other possible decompositions of the prime $p$ in $K$, alternative algorithms are needed to compute the possible group orders. In the case of $p$-rank $1$, a solution was given in~\cite{HMNS}. In some of the other examples, we show how to determine the group orders for other cases below.

The possible group orders in the case for Example 3 are $\#J(C)(\FF_{89^2}) = 62045284$ or $63439556$, for a genus 2 curve $C$ over $\FF_{89^2}$ with CM by $K$.  
This can be seen as follows: let $p = \gerp_1\gerp_2\gerp_3^2$.  In this case it can be verified using Magma or pari that
both of the ideals $\gerp_1\gerp_3$ and $\gerp_2\gerp_3$ are principal, generated by $\pi$ and $\overline{\pi}$, and 
$\pi\overline{\pi}=p$.  As in the algorithm explained in~\cite{HMNS}, we find the Weil $p^2$-numbers 
$\beta = \pm \pi\overline{\pi}^{-1}p$.  
Then the corresponding group orders for these Weil $p^2$-numbers are $N=\prod_{\sigma} (1-\beta^{\sigma})$, where $\sigma$ ranges over the complex embeddings of $K$.

Represent $\FF_{89^2} = \FF_{89}[\alpha]$, where $\alpha$ satisfies $\alpha^2 + 82\alpha + 3=0$.
The four curves are 

$$y^2 = f_1(x) =
\alpha^{5245}x^6 + \alpha^{2244}x^5 + \alpha^{7129}x^4 + \alpha^{1567}x^3 + \alpha^{2060}x^2 + \alpha^{5783}x + \alpha^{3905} $$

$$ y^2 = f_2(x) =
\alpha^{2667}x^6 + \alpha^{795}x^5 + \alpha^{1956}x^4 + \alpha^{5619}x^3 + \alpha^{5331}x^2 + \alpha^{7272}x + 52 $$

$$y^2 = f_3(x) =
\alpha^{6464}x^6 + \alpha^{795}x^5 + \alpha^{4574}x^4 + \alpha^{2946}x^3 + \alpha^{1544}x^2 + \alpha^{6684}x + \alpha^{803} $$ 

$$y^2 = f_4(x) =
\alpha^{132}x^6 + \alpha^{3403}x^5 + \alpha^{2326}x^4 + \alpha^{3493}x^3 + \alpha^{5184}x^2 + \alpha^{1943}x + \alpha^{4418}$$

Calculating the Hasse-Witt matrix for the first curve, one computes $f_1^{44}$ and finds
$c_{88}= \alpha^{7555}$,
$c_{87} = \alpha^{7787}$,
$c_{177} = \alpha^{950}$,
$c_{176} = \alpha^{1182}$, and that both $M$ and $M^{(p)}M$ have rank $1$, so both the $f$-number and the $a$-number equal $1$.
The same is true for the other three curves as well.

\subsubsection{Cases (ii) and (iv) in Table 3.5.1 for non-Galois fields} 

We still refer to the field $K=\QQ[x]/(x^4 + 134x^2 + 89)$ and the class polynomials given above.

{\bf Example 5}
To give an example for cases (ii) and (iv) in Table 3.5.1 for non-Galois fields, we let $p=313$.
The prime $p=313$ decomposes in $K$ as the product of two prime ideals of degree $1$ and one prime ideal
with residue degree $2$.
Modulo $313$, the class polynomials factor as a product of four degree-two polynomials:
\begin{equation*}
\begin{split}
h_1(x) & = (x^2 + 25 x + 273)(x^2 + 137 x + 39)(x^2 + 200x + 108)(x^2 + 312x + 249) \pmod{313},\\
h_2(x) & = (x^2 + 20 x + 121)(x^2 + 90x + 119)(x^2 + 138x + 297)(x^2 + 173x + 78) \pmod{313},\\
h_3(x) & = (x^2 + 105x+ 276)(x^2 + 133x + 230)(x^2 + 232x+183)(x^2 + 289x + 91) \pmod{313}.
\end{split}
\end{equation*}
The two possible group orders are $\#J(C)(\FF_{89^2}) = 9607909136$ or $9588315136$, for a genus 2 curve $C$
over $\FF_{313^2}$ with CM by $K$.  This can be seen because both of the prime ideals of $K$ of degree $1$ lying above $p$
are principal, and letting $\pi$ and $\overline{\pi}$ be the generators, we find the Weil $p^2$-numbers 
$\beta = \pm \pi\overline{\pi}^{-1}p$ (this is also explained in~\cite{HMNS}).  
Then the corresponding group orders for these Weil $p^2$-numbers are $N=\prod_{\sigma} (1-\beta^{\sigma})$, where
$\sigma$ ranges over the complex embeddings of $K$.
Represent $\FF_{313^2} = \FF_{313}[\alpha]$, where $\alpha$ satisfies $\alpha^2 + 310\alpha + 10=0$.
We find eight curves defined over $\FF_{313^2}$.
For example, the first one is the hyperelliptic curve defined over $\FF_{313^2}$ by 
$$y^2 = f(x) = \alpha^{20046}x^6 + \alpha^{18815}x^5 +
    \alpha^{77496}x^4 + \alpha^{26504}x^3 + \alpha^{19266}x^2 + \alpha^{53721}x + \alpha^{1332}.$$
Calculating $f(x)^{156}$, one finds that the coefficients of the Hasse-Witt matrix $M$ are:
$c_{p-1} = \alpha^{91834}$,
$c_{p-2} = \alpha^{18900}$,
$c_{2p-1} = \alpha^{62990}$,
$c_{2p-2} = \alpha^{88024}.$
The determinant of both $M$ and $M^{(p)}M$ is $0$ and the rank is $1$.  
The same is true for all 8 curves: they all have $a=1$ and $f=1$.
\

\subsubsection{Cases (iii) and (v) in Table 3.5.1}
This next set of cases is very interesting, because we can see here that the decomposition of the prime in $K$ only determines the reduction of the abelian surface in combination with the CM type.  
This is the first time we have an example of both superspecial and ordinary reduction modulo the same prime (of CM abelian surfaces with CM by the same field $K$, but different CM type). This phenomenon does not occur in genus $1$.

We again work with the primitive quartic CM field $K=\QQ[x]/(x^4 + 134x^2 + 89)$ and the class polynomials given above.
Let $p=47$. As in cases (iii) and (v) in Table 3.5.1, the prime $p=47$ decomposes in $K$ as a product of two prime ideals of degree $2$: $p$ is inert in $K^+$, the real quadratic subfield of $K$, and then splits in $K$.
The class polynomials factor modulo $47$ as 
\begin{equation*}
\begin{split}
h_1(x) & = (x^2 + 18)^2(x^2 + 22 x + 12)(x^2 + 33x + 19)(x^2 + 37x + 6) \pmod{47},\\
h_2(x) & = (x^2 + 23)^2(x^2 + 10x + 46)(x^2 + 6x + 17)(x^2 + 9x + 39) \pmod{47},\\
h_3(x) & = (x^2 + 2)^2(x^2 + 42x + 26)(x^2 + x + 19)(x^2 + 27x + 7) \pmod{47}.
\end{split}
\end{equation*}

{\bf Example 6 (case (v))}
Both degree-$2$ prime ideals lying over $p=47$ are principal in this case, and we denote the generators by 
$\pi$ and $\overline{\pi}$.  In this case, $\pi\overline{\pi} = 47u$, where $u$ is a unit.  Setting $\beta = \pm p^2/u$,
gives two possible Weil $p^2$-numbers.  The two possible group orders are $\#J(C)(\FF_{47^2}) = \prod_{\sigma} (1-\beta^{\sigma}) = 4901092$ or $4865732$, where $\sigma$ ranges over the complex embeddings of $K$.
There are 4 ordinary CM points corresponding to these possible group orders.  

Represent $\FF_{47^2} = \FF_{47}[\alpha]$, where $\alpha$ satisfies $\alpha^2 + 45\alpha + 5=0$.
Then the four curves with these group orders are:
$$ y^2 = \alpha^{829}x^6 + \alpha^{1842}x^5 + \alpha^{622}x^4+ \alpha^{1262}x^3 + \alpha^{956}x^2 + \alpha^{398}x + \alpha^{1255} $$
$$ y^2 = \alpha^{929}x^6 + \alpha^{1219}x^5 + \alpha^{1483}x^4 + \alpha^{1511}x^3 + \alpha^{251}x^2 + \alpha^{224}x + \alpha^{1437} $$ 
$$ y^2 = \alpha^{1852}x^6 + \alpha^{2038}x^5 + \alpha^{1790}x^4 + \alpha^{1078}x^3 + \alpha^{1166}x^2 + \alpha^{1634}x + \alpha^{1518} $$
$$ y^2 = \alpha^{1783}x^6 + \alpha^{892}x^5 + \alpha^{1454}x^4 + \alpha^{665}x^3 + \alpha^{1014}x^2 + \alpha^{871}x + \alpha^{1754}.$$
For all four curves, we checked that the Hasse-Witt matrix $M$ and $M^{(p)}M$ both have rank $2$, so these curves are indeed all ordinary.

{\bf Example 7 (case (iii))}
Each of the three class polynomials has one linear factor modulo $47$.  The curve over $\FF_{47}$
with those $\FF_{47}$-rational invariants is the hyperelliptic curve defined by 
$$y^2 = 40x^6 + 22x^5 + 43x^4 + x^3 +29x^2 +8x + 28.$$
Its Jacobian has $\#J(C)(\FF_{47}) = p^2+2p+1 = 2304$ points and $\#C(\FF_{47}) = p+1 = 48$.
The Hasse-Witt matrix $M$ is identically $0$ modulo $47$, so the curve is superspecial.
This curve occurs ``with multiplicity two" modulo $47$.

The other two CM abelian surfaces reduce to curves defined over $\FF_{47^2}$.  
They are the hyperelliptic curves defined by 
$$y^2 = \alpha^{487}x^6 + \alpha^{977}x^5 + \alpha^{1698}x^4 + \alpha^{1530}x^3 + \alpha^{1790}x^2 + \alpha^{1618}x + \alpha^{1063}$$
$$ y^2 = \alpha^{809}x^6 + \alpha^{1759}x^5 + \alpha^{318}x^4 + \alpha^{1254}x^3 + \alpha^{226}x^2 + \alpha^{974}x + \alpha^{1385}.$$
They both have $\#J(C)(\FF_{47^2}) = p^4 - 2p^2 + 1 = 4875264$ points and $\#C(\FF_{47^2}) = p^2+1 = 2210$.
They both have the property that the Hasse-Witt matrix $M$ is identically $0$ modulo $47$, so the curves are both superspecial.

\

\subsubsection{Case (vii) in Table 3.5.1: totally inert}
We again work with the non-galois quartic CM field $K=\QQ[x]/(x^4 + 134x^2 + 89)$ and the class polynomials given above.
The prime $p=13$ is totally inert in $K$.  Modulo $13$, the class polynomials are:
\begin{equation*}
\begin{split}
h_1(x) & = (x^2 + 2 x + 9)(x^2 + 6x + 1)(x^4 + 8x^3 + 10x^2 + 12) \pmod{13},\\
h_2(x) & = (x^2 + 5 x + 1)(x^2 + 8x + 1)(x^4 + 7x^3 + 6x^2 + 7x + 8) \pmod{13},\\
h_3(x) & = (x^2 + 2)(x^2 + 11)(x^4 + 6x^3 + 4x^2 + 5) \pmod{13}.
\end{split}
\end{equation*}
We look for curves over $\FF_{13^2}$ with $\#J(C)(\FF_{13^2}) = (p^4+2p^2+1) = 28900$.
Represent $\FF_{13^2} = \FF_{13}[\alpha]$, where $\alpha$ satisfies $\alpha^2 + 12\alpha + 2=0$.
We find 4 curves over $\FF_{13^2}$, for example the first one is:
$$y^2 = \alpha^{99}x^6 + \alpha^{47}x^5 + \alpha^{156}x^4 + \alpha^{75}x^3 + \alpha^{27}x^2 + x + \alpha^{148}.$$
Its Hasse-Witt matrix $M$ has rank $1$ and the rank of $M^{(p)}M$ is $0$, so $a=1$ and $f=0$ as predicted in the tables.  
The same is true of the other 3 curves as well.



\

\

\section{The moduli space of pairs of elliptic curves}
\label{section: The moduli space of pairs of elliptic curves}

Let $N$ be a positive integer. Consider the functor $\BB_N$ on schemes associating to a scheme $S$ the isomorphism class of triples 
\[(E_1, E_2, \gamma),\]
where $\pi_i: E_i \arr S, i=1, 2,$ are elliptic curves over $S$ and $\gamma$ is a full level structure on $E_1[N] \times E_2[N]$, namely, an isomorphism,
\[ \gamma: E_1[N] \times E_2[N] \arr (\ZZ/N\ZZ)^4, \]
which is symplectic relative to the Weil pairing on $E_1 \times E_2$ (obtained as the product of the Weil pairings on each elliptic curve, or, equivalently, associated to the product polarization on $E_1 \times E_2$) and the standard pairing on $(\ZZ/N\ZZ)^4$ given by the matrix $\left(\begin{smallmatrix}  0& 1 & &  \\  -1&0  & &  \\& & 0& 1   \\ & &-1& 0  \\
\end{smallmatrix} \right)$.

An isomorphism $\varphi\colon(E_1, E_2, \gamma) \arr (E_1', E_2', \gamma')$ of two such triples over $S$ is a pair of isomorphisms of $S$ schemes, $\varphi_i: E_i \arr E_i'$, such that $\gamma = \gamma' \circ (\varphi_1 \times \varphi_2)$.

The functor $\BB_N$ is naturally equivalent to the functor parameterizing isomorphism classes of quadruples $(A, \lambda, e, \gamma)$ over $S$, where $(A, \lambda)$ is a principally polarized abelian surface over $S$, $e$ is a non-trivial idempotent, fixed under the $\lambda$-Rosati involution, and $\gamma$ is a symplectic level $N$ structure. Indeed, given a triple $(E_1, E_2, \gamma)$ associate to it $(E_1\times E_2, \lambda_1\times \lambda_2, e, \gamma)$, where $\lambda_i$ are the canonical principal polarizations on $E_i$ and $e$ is the idempotent endomorphism $(x, y) \mapsto x$. The converse construction associates to $A$ the triple $(E_1, E_2, \gamma)$, where $E_1 = \Ker (1 - e), E_2 = \Ker(e)$. It is not hard to verify that these constructions give a natural equivalence between the functors. 

\begin{lem} For $N\geq 3$ the moduli problem is rigid. Namely, any automorphism 
$\varphi$ of a triple $(E_1, E_2, \gamma)$ is the identity.  
\end{lem}
\begin{proof} Such an automorphism induces an automorphism of $(A, \lambda, \gamma)$, where $A = E_1 \times E_2$. It is well known that such an automorphism must be the identity. 
\end{proof}

It follows then from standard techniques that for $N\geq 3$ the functor $\BB_N$ is representable by a quasi-projective scheme $\scrB_N$\label{BN} over $\ZZ[\zeta_N, N^{-1}]$.

\begin{prop} \label{Proposition: BN and A2N} Let $N\geq 2$. Let $J$\label{J} be the automorphism of $\scrB_N$ whose effect on points is
\[ (E_1, E_2, \gamma) \mapsto (E_2, E_1, \gamma \circ s), \]
where $s$ is the natural ``switch", $s: E_1[N]\times E_2[N] \arr E_2[N]\times E_1[N]$. We have a commutative diagram, 
\[\xymatrix{\scrB_N\ar[dr]^\beta \ar[d]& \\ \scrB_N/\langle J \rangle  \ar@{^{(}->}^{\beta_J}[r]& \scrA_{2, N}, }\]
where the diagonal arrow $\beta$\label{beta} is the natural morphism $(E_1, E_2, \gamma) \mapsto (E_1 \times E_2, \lambda_1 \times \lambda_2, \gamma)$, the vertical arrow is an \'etale Galois cover with Galois group $\ZZ/2\ZZ$ and the bottom arrow $\beta_J$ is a closed immersion, induced by $\beta$, whose image is the Humbert surface $\scrH_{1, N}$\label{H1N} in $\scrA_{2, N}$, the Zariski closure of $H_{1, N}\subset \scrA_{2, N}(\CC)$.
\end{prop}
\begin{proof}
We first show that the morphism $\scrB_N \arr \scrB_N/\langle J \rangle $ is unramified. Suppose that $J(E_1, E_2, \gamma) = (E_2, E_1, \gamma \circ s)$ is isomorphic to $(E_1, E_2, \gamma)$. There are then isomorphisms $\varphi_1: E_2 \arr E_1$, $\varphi_2: E_1 \arr E_2$ such that $\gamma \circ s = \gamma \circ (\varphi_1 \times \varphi_2)$ and so $s = \varphi_1 \times \varphi_2$ on $E_1[N]\times E_2[N]$. But, for $(a,b)\in E_1[N]\times E_2[N]$ we have $s(a, b) = (b, a)$, while $\varphi_1 \times \varphi_2 (a, b) = (\varphi_1(a), \varphi_2(b))$, which obviously cannot hold for every pair $(a, b)$ if $N\geq 2$. 

The morphism $\scrB_N \arr \scrB_N/\langle J \rangle $, being a quotient by a finite group, is a finite morphism. We conclude that it is a finite \'etale cover with Galois group $\ZZ/2\ZZ$. The natural morphism $\beta:\scrB_N \arr \scrA_N$ clearly factors through $\scrB_N/\langle J \rangle$ and we denote the induced morphism 
\[ \beta_J: \scrB_{N, J} \arr \scrA_{2, N}.\]
We claim that this is a geometrically injective morphism. Suppose that 
\[ (E_1 \times E_2, \lambda_1 \times \lambda_2, \gamma) \cong (E_1' \times E_2', \lambda_1' \times \lambda_2', \gamma').\] 
By a theorem of Weil, after possibly switching $E_1'$ with $E_2'$ , we may assume that $E_1 \cong E_1', E_2 \cong E_2'$ and so, under these identifications, that $\gamma = \gamma'$. Namely, up to applying $J$, every point in the image has a unique pre-image.

The morphism $\beta_J$ is also proper. This follows from the valuative crietrion of properness. As we shall see below the scheme $\scrB_N$ is a union of products of modular curves, in particular it is noetherian and so we can use discrete valuation rings in the criterion. To apply it, we must show that if $R$ is a discrete valuation ring with field of fractions $K$, $(A,\lambda, \gamma)/R$ is an abelian scheme whose generic fiber is isomorphic over $K$ to $(E_1 \times E_2, \lambda_1 \times \lambda_2, \gamma)$ then the elliptic curves $E_i$ extend to elliptic curves over $R$ and then so does the isomorphism. The fact that the elliptic curves extend follows from the theory of N\'eron models (since $E_1\times E_2 = A\otimes_R K$ obviously has good reduction). The extension of the isomorphism follows from the fact that $\scrA_{2, N}$ has a toroidal compactification which is proper over $\ZZ[\zeta_N, N^{-1}]$. Since both $ \scrB_N/\langle J \rangle$ and $\scrA_{2, N}$ are reduced and the morphism $\beta_J$ is proper and injective (hence quasi-finite), $\beta_J$ is a finite injective morphism. We will conclude it is an isomorphism onto its image, the Humbert surface $\scrH_{1,N}$ by showing that for a geometric point $x$ of $ \scrB_N/\langle J \rangle$ and its image $y$ in $\scrA_{2, N}$ the completed local rings are isomorphic. Note that the Humbert divisor $\scrH_{1, N}$ is the image of $\beta_J$, since they have the same generic fiber and both are the closure of their generic fiber.

Indeed, suppose that $y$ is the image of the $k$-geometric point $(y_1, y_2)$ of $\scrB_N$. The completed local ring on $\scrB_N$ is then just isomorphic to $W(k)[\![t_1, t_2]\!]$, as $\scrB_N$ is a product of smooth curves. Moreover, if $E_i$ is the elliptic curve corresponding to $y_i$, then $t_i$ is the parameter arising via the local deformation theory for elliptic curves (the level structure need not be a product level structure; regardless it extend uniquely by \'etaleness). On the other hand, the completed local ring on $\scrA_{2, N}$ of the point $y$ corresponding to $(A, \lambda, \gamma) = (E_1 \times E_2, \lambda_1 \times \lambda_2, \gamma)$ is isomorphic to the ring $W(k)[\![t_{11}, t_{1, 2}, t_{2, 1}, t_{2, 2}]\!]/(t_{1, 2} - t_{2, 1})$ and $\scrH_{1, N}$ contains locally the closed formal subscheme defined by the ideal $(t_{1, 2}, t_{2, 1})$, as is clear from the interpetation of the variables through local deformation theory. Since $\scrB_{N}/\langle J \rangle$ is locally irreducible and the morphism is geometrically injective also
$\scrH_{1, N}$ is locally irreducible. It follows that $\scrH_{1, N}$ is defined locally by the ideal $(t_{1, 2}, t_{2, 1})$ and that the morphism is an isomorphism on every completed local ring, which is sufficient to conclude the proof.

Another way to conclude the proof is to prove that the morphism $\beta_J$ is universally injective (or a monomorphism) and then use EGA IV, \S 8.11, Proposition (8.11.5). Since $\scrB_{N}/\langle J \rangle$ is the categorical quotient of $\scrB_N$, we know it as a functor of points and so injectivity boils down to the following statement: Given elliptic curves $E_1, \dots, E_4$ over a connected scheme $S$ such that $E_1 \times E_2 \cong E_3 \times E_4$ as principally polarized abelian schemes over $S$ then, either $E_1 \cong E_3$ and $E_2 \cong E_4$, or $E_1 \cong E_4$ and $E_2 \cong E_3$. Note that to identify $E_1$ in $E_3 \times E_4$ is equivalent to giving an endomorphism. Choose a geometric point $x$ of $S$ and use Weil's theorem as above together with Grothendieck's theorem $\End_S(E_3 \times E_4) \injects \End_{k(x)}((E_3 \times E_4)  \otimes k(x))$.
\end{proof}

We next discuss the complex uniformization of $\scrB_N$. Recall the classical construction of the modular curves:  Given $\tau \in \gerH$ one lets $E_\tau = \CC/\langle 1, \tau \rangle$ be the corresponding elliptic curve, and we get a symplectic isomorphism $E_\tau[N] \arr (\ZZ/N\ZZ)^2$ by sending $1/N$ to $(1, 0)$ and $\tau/N$ to $(0, 1)$. We call this level structure $\gamma_0$. Let $\sigma = M\tau$, where $M = \left(\begin{smallmatrix} a & b\\ c & d \end{smallmatrix} \right)$. Then the isomorphism $E_\sigma \arr E_\tau$ is given by multiplication by $j(M, \tau) = c\tau + d$. Since $\gamma_0(A + B\sigma)/N = {^t(A, B)}$ and $1/N$ is sent to $(d + c\tau)/N$, while $\sigma/N$ is sent $(b + a\tau)/N$, we find that $(E_\sigma, \gamma_0)$ is isomorphic to $(E_\tau, \left(\begin{smallmatrix}a & -b \\ -c & d \end{smallmatrix}\right)\circ \gamma_0)$.
We remark that $M = \left(\begin{smallmatrix}a & b \\ c & d \end{smallmatrix}\right)\mapsto M^\dagger :=\left(\begin{smallmatrix}a & -b \\ -c & d \end{smallmatrix}\right)$ is an outer automorphism of $\SL_2(\ZZ)$ given by conjugating by $\left(\begin{smallmatrix}1 & 0 \\0 & -1 \end{smallmatrix}\right)$ in $\GL_2(\ZZ)$.

Consider the space
\[ \gerH \times \gerH \times \Symp_4(\ZZ/N\ZZ).\]
(Here the symplectic group is relative to the pairing fixed above.)
To a point $(\tau_1, \tau_2, \gamma)$ of this space we associate the triple $(E_{\tau_1}, E_{\tau_2}, \gamma \circ (\gamma_0 \times \gamma_0))$. The group $\SL_2(\ZZ) \times \SL_2(\ZZ)$ acts on the space by
\[(M_1, M_2) \ast (\tau_1, \tau_2, \gamma) = (M_1\tau_1, M_2\tau_2, \diag(M_1^\dagger, M_2^\dagger) \circ \gamma).\]
The space of orbits is isomorphic to $\scrB_N(\CC)$. Furthermore, choose a complete set of representatives $\gamma_1, \dots, \gamma_t$ ($t = t(N)$) for $\SL_2(\ZZ/N\ZZ) \times \SL_2(\ZZ/N\ZZ) \backslash \Symp_4(\ZZ/N\ZZ)$. Then, 
\[\scrB_N (\CC) \cong \coprod_{i = 1}^t (\Gamma(N) \backslash \gerH)^2 = \coprod_{i = 1}^t Y(N) \times Y(N).\]
Via this identification, we associate to a pair $(\tau_1,\tau_2)$ in the $i$-th (or $\gamma_i$-th, if one prefers) component of $\scrB_N(\CC)$ the triple $(E_{\tau_1}, E_{\tau_2}, \gamma_i\circ (\gamma_0 \times \gamma_0))$. 

The involution $J$ takes the $\gamma_i$-component to $\gamma_j$-component where $\gamma_j$ is determined by $\gamma_i \circ (\gamma_0 \times \gamma_0) \circ s\in (\SL_2(\ZZ/N\ZZ) \times \SL_2(\ZZ/N\ZZ)) \gamma_j\circ (\gamma_0 \times \gamma_0) $. Typically, $\gamma_j \neq \gamma_i$. In fact, 
the components of $\scrB_N$ are parameterized by $\SL_2(\ZZ/N\ZZ) \times \SL_2(\ZZ/N\ZZ)\backslash \Symp_4(\ZZ/N\ZZ)$, while the components of $\scrB_{J}/\langle N\rangle$ are parameterized by 
$\SL_2(\ZZ/N\ZZ) \times \SL_2(\ZZ/N\ZZ)\backslash \Symp_4(\ZZ/N\ZZ)/H$, where $H = \{1, \left(\begin{smallmatrix}0 & I_2 \\ I_2 & 0 \end{smallmatrix} \right)\}$. 

\begin{rmk} Here is a typical example illustrating the difference between $\scrH_{1, N}$ and $\scrB_N$. Let $K$ be a field, $L$ a quadratic Galois extension of $K$ and $\sigma$ the non-trival automorphism of $L$ over $K$. Let $E_1$ be an elliptic curve defined over $L$ and not over $K$. Let $E_2$ be the curve obtained by $\sigma$ to the equation of $E_1$ (and so $j(E_2) = \sigma (j(E_1))$). The point $(E_1, E_2)$ of $\scrB_1$ is defined over $L$, but not over $K$. On the other hand, its image, $A = (E_1\times E_2, \lambda_1 \times \lambda_2)$ is defined over $K$. A quadratic extension is needed to define the elliptic curves $E_1, E_2$ such that $A \cong E_1 \times E_2$. To study the situation more precisely, we must include level $N$ structure.

From a scheme theoretic point of view we have the following cartesian diagram,
\[\xymatrix{\Spec(K) \times_{\scrA_{2, N}} \scrB_N\ar[d]\ar[r] & \scrB_N \ar[d]\ar[dr] & \\ \Spec(K) \ar[r] &\scrA_{2, N} & \scrB_{N, J}\ar@{_{(}->}[l].}\]
The morphism $\Spec(K) \times_{\scrA_{2, N}} \scrB_N \arr \Spec(K)$ is finite \'etale (being a base change of the morphism $\scrB_N \arr \scrA_{2, N}$) and so $\Spec(K) \times_{\scrA_{2, N}} \scrB_N = \Spec(L')$, where $L'/K$ is a separable quadratic $K$-algebra. 

\end{rmk}


\

\

\section{A lemma in arithmetic intersection theory}
\label{section: intersection theory}

Let $R$ be a Dedekind ring, finite over $\ZZ_p$, $\gerp \normal R$ a prime ideal. Let $\pi:S \arr \Spec(R)$ be a smooth scheme of finite type over $\Spec(R)$. Let $x\in S$ be a closed point of characteristic $p$ lying over $\gerp$. Then $\calO_S^{\wedge x}$, the completed local ring at $S$ is isomorphic to $\tilde{R}[\![x_1, \dots, x_n]\!]$ where $n$ is the relative dimension of $S$ over $R$ and  
$\tilde{R}=R\otimes_{R_0} W(R/\gerp)$, where $R_0$ the maximal unramified subring of $R$. See \cite{CohenI}.
In particular, $\calO_S^{\wedge x}$ is a noetherian unique factorization domain. As a consequence, every divisor on $\Spf(\calO_S^{\wedge x})$
is principal. (We remark that in fact this latter fact follows directly from the Auslander-Buchsbaum theorem without need for Cohen's theorem.)

\begin{lem}\label{Lemma: valuation and reduction} Let $S \arr \Spec(R)$ be a smooth integral scheme of finite type over a Dedekind ring~$R$ containing $\ZZ$. Let $B$ be a Dedekind ring containing $R$, $K$ its field of fractions and $\eta$ be the generic point. Let 
\[\iota: \Spec(B) \arr S, \]
be a morphism of schemes over $R$. Let $f$ be a rational function on $S$ such that the divisor of $f$ intersects the image of $\iota$ properly (in particular, $f(\eta) = \iota^\ast f$ is a well defined element of $K$). 
Let the divisor of $f$ equal $(f)_0 - (f)_\infty = \sum m_iD_i$, where the $m_i$ are non-zero integers and $D_i$ irreducible reduced effective divisors. Let $Z$ be the closed reduced subscheme which is the support of ${\rm div}(f)_0$.

Let $\gerp$ be a prime ideal of $B$ and $x$ its image under $\iota$. Suppose that $\val_\gerp(f(\eta))  = \alpha >0$. Then $d=\max\{m_i: x\in D_i\} > 0$. Let $a =\lceil \alpha/d\rceil$. Then $a>0$ and the morphism $\iota: \Spec(B/\gerp^a)$ factors through ${\rm div}(f)_0$:
\begin{equation}\label{equation: factoring through divisor}
\xymatrix{\Spec(B/\gerp^a) \ar[r]\ar@{-->}[drr] & \Spec(B) \ar[r]^\iota & S \\
& & Z\ar@{^{(}->}[u]  } 
\end{equation}
\end{lem}

\begin{rmk} We shall apply this Lemma later, in the following context: $S$ will be the modular scheme $\scrA_{2, N}$, $f$ will be a function such that $f = \Theta^k/g$, where $g$ is a modular form of weight $10k$ with rational Fourier coefficients, the morphism $\iota$ will be such that $\iota(\eta)$ is a CM point and our assumption will be that $\val_\gerp(f) = a > 0$. \end{rmk}
\begin{proof} We first argue that we may replace $S$ by the $\Spf(\calO_S^{\wedge x})$. Indeed, on the one hand, diagram~(\ref{equation: factoring through divisor}) gives by passing to completions at $x$ a diagram
\begin{equation}\label{equation: factoring through formal divisor}
\xymatrix{\Spec(B/\gerp^a) \ar[r]\ar@{-->}[dr] &  \Spf(\calO_S^{\wedge x}) \\
 & Z\cap \Spf(\calO_S^{\wedge x})\ar@{^{(}->}[u]  } 
\end{equation}
On the other hand, diagram~(\ref{equation: factoring through formal divisor}) is coming from unique continuous morphisms $\calO_S^{\wedge x} \arr B/\gerp^n$ etc., that arise uniquely from morphisms $\calO_S \arr B/\gerp^n$, etc. 

In $\Spf(\calO_S^{\wedge x})$ every divisor is principal and so we may write there $D_i' = (f_i)$ where $f_i\in \calO_S^{\wedge x}$, and $D_i'$ is the induced divisor on $\Spf(\calO_S^{\wedge x})$. $D_i'$ may be reducible, but it is reduced. 
If $x \not\in D_i$ then $f_i$ is a unit  in $\calO_S^{\wedge x}$. Via the morphism
$ \Spec(B_\gerp) \arr \Spec(B) \arr S$, that induces a morphism $\Spec(B_\gerp) \arr \Spf(\calO_S^{\wedge x})$, we may view $f(\eta)$ as an element of $K_\gerp$, which is equal, up to a unit, to $\prod_if_i(x)^{m_i}$ and so:
\begin{multline}\alpha = \val_\gerp(f(\eta)) = \sum_{\{i: x\in D_i\}} m_i \cdot \val_\gerp(f_i(\eta))\\ = \sum_{\{i: x\in D_i, m_i>0\}} m_i \cdot \val_\gerp(f_i(\eta))
+\sum_{\{i: x\in D_i, m_i<0\}} m_i \cdot \val_\gerp(f_i(\eta))
.\end{multline}
We note that if $x\in D_i$ then $\val_\gerp(f_i) \geq 1$ (it may be strictly bigger, of course). In particular, $d>0$. Consider $\alpha' = \sum_{\{i: x\in D_i, m_i>0\}} \val_\gerp(f_i(\eta))$; clearly $\alpha'\cdot d \geq \alpha$ and so $\alpha' \geq \lceil \alpha/d \rceil$ and so it will be enough to prove that diagram~(\ref{equation: factoring through formal divisor}) holds with $\alpha'$. Consider the function $f_Z = \prod_{\{i: x\in D_i, m_i >0\}} f_i$ which defines $Z\cap \Spf(\calO_S^{\wedge x})$. To show diagram~(\ref{equation: factoring through formal divisor}) holds is equivalent to show that $f_Z$, when pulled back to $\Spec B_\gerp$ has valuation at least $\alpha'$. But the valuation is precisely $\sum_{\{i: x\in D_i, m_i>0\}} \val_\gerp(f_i(\eta))$ and we are done.  
\end{proof}

\subsubsection{Examples} The whole theory is developed precisely to deal with situations where one cannot just ``write down everything explicitly", and so our examples are a bit artificial.
\begin{itemize}
\item Consider the scheme $S = \Spec(\ZZ[x])$ and the function $f(x) = x^2 - 1$. The divisor of $f$ is 
\[D_1 + D_2, \qquad D_1 = {\rm div}(x-1), D_2 = {\rm div}(x+1).\]
Let $\tau = 3$ corresponding the the point determined by the homomorphism $\ZZ[x] \arr \ZZ, x \mapsto 3$. We have $\val_2(f(\tau)) = \val_2(8) = 3$. We examine the situation on the completed local ring of the point $(2, x-3) = (2, x-1) = (2, x+1)$ (the reduction of $\tau$ modulo $2$). Also at this completed local ring the divisor of $f$ is given by $D_1 = {\rm div}(x-1), D_2 = {\rm div}(x+1)$ (with a slight abuse of notation). It follows from our lemma that the morphism $\Spec(\ZZ) \arr \Spec(\ZZ[x])$ corresponding to $\tau$ induces a morphism
\[ \Spec(\ZZ/2^3\ZZ) \arr D_1 \cup D_2,\] where by $D_1 \cup D_2$ we mean the closed reduced subscheme whose support is $D_1 \cup D_2$, namely $\Spec(\ZZ[x]/(x^2 - 1)$. Indeed, this is nothing but saying that there is indeed a well defined homomorphism $\ZZ[x]/(x^2 - 1) \arr \ZZ/2^3\ZZ$ taking $x$ to $3$.

An interesting feature of this example is that the morphism $\Spec(\ZZ) \arr \Spec(\ZZ[x])$ only induces a well defined morphism $\Spec(\ZZ/2^i\ZZ) \arr D_i$ (where $D_i$ is the reduced closed scheme supported on $D_i$, namely $\Spec(\ZZ[x]/(x-1))$ for $i=1$ and  $\Spec(\ZZ[x]/(x+1))$ for $i=2$). Moreover, the divisors $D_1$ and $D_2$ intersect transversely, the intersection being $(x-1, x+1)$. The subtlety is in the scheme structure on $D_1 \cup D_2$ and in particular in the fact that $\ZZ[x]/(x^2 - 1) \subsetneqq \ZZ[x]/(x-1) \oplus \ZZ[x]/(x+1)$.
\item Once more $S = \Spec(\ZZ[x])$ but now $f(x) = x^2 +1$, which is irreducible. The point $\tau = 2$ corresponds to the homomorphism $\ZZ[x] \arr \ZZ, x\mapsto 2$. We have $\val_5(f(\tau)) = \val_5(5) = 1$. We have an induced morphism $\Spec(\ZZ/5\ZZ) \arr \Spec(\ZZ[x]/(x^2 +1))$, which amount to the fact that there is a homomorphism $\ZZ[x]/(x^2 +1) \arr \ZZ/5\ZZ$ taking $x$ to $2$.

In the completed local ring of the point $(5, x-2)$ the function $f$ decomposes as $f(x) = (x-i)(x+i)$ where $i$ is an element of $\ZZ_5$ whose square is $-1$ and whose reduction is $2$ modulo $5$. Thus, the function $x-i$ vanishes to first order at this point, while the function $x+i$ is a unit. The divisor of $f$ is locally $D_1 = {\rm div}(x+i)$ and the lemma states that we have an induced morphism $\Spec(\ZZ/5\ZZ) \arr \Spf(\ZZ_5[\![(x-2)]\!]/(x-i))$, which amounts to the fact that there is a well defined continuous homomorphism $\ZZ_5[\![(x-2)]\!]/(x-i) \arr \ZZ/5\ZZ$ taking $x$ to $2$.

\item Consider $\Spec(\ZZ[1/6][x, y]/(y^2 - (x^3 - 1)))$ and the function $f(x) = x-1$ whose divisor is $2[(1, 0)]$, and we note that the divisor $[(1, 0)]$ is not principal . We have $x^3 - 1 = (x-1)(x^2 + x + 1)$ and we let $S'$ be the open subscheme whose complement is given by $x^2 + x + 1$. The divisor of $f$ on $S'$ is still $2[(1, 0)]$ but now $[(1, 0)]$ is locally principal; it is the divisor $D = [(1, 0)]$ of $(x-1)/y$. The divisor of $f$ is $2D$. Finally, let $S$ be the base change of $S'$ to $\ZZ[1/6,\sqrt{215}]$.

We consider the point $\tau = (6, \sqrt{215})$ of $S$, corresponding to the homomorphism \[\ZZ[1/6,\sqrt{215}, x,y,1/(x^2 + x+1)]/(y^2 - x^3 + 1) \arr \ZZ[1/6, \sqrt{215}],\]  given by \[ (x, y) \mapsto (6, \sqrt{215}).\] Let $\gerp$ be the prime ideal above $5$ in $\ZZ[1/6,\sqrt{215}]$. We have $f(\tau) = 5$ and $\val_\gerp(f(\tau))=2$. We deduce from our lemma that we have an induced morphism $\Spec(\ZZ/5\ZZ) \arr D$, corresponding to the fact that there is a well defined homomorphism \[ \ZZ[1/6,\sqrt{215}, x,y,1/(x^2 + x+1)]/(y^2 - x^3 + 1, (x-1)/y) \arr \ZZ[1/6,\sqrt{215}]/(\gerp) \cong \ZZ/5\ZZ,\] where $(x, y) \mapsto (6, \sqrt{215})$.
\end{itemize}


\

\

\section{A problem in deformation theory} \label{section: deformation theory}
\subsection{Deforming endomorphisms} Let~$A$ be an abelian
variety of dimension~$g$ over a perfect field~$k$ of characteristic
$p$ and let~$r$ be the rank over~$\ZZ$ of~$\End_k(A)$ (it is finite
and at most~$4g^2$). Let~$(R, \germ_R)$ be a local artinian ring
with residue field~$k = R/\germ_R$ of characteristic~$p$. Let~$n_R$
be the minimal positive integer such that~$\germ_R^{n_R} = 0$. Let
$t_R$ be the least positive integer such that~$p^{t_R} \in
\germ_R^{p-1}$.

Let~$\AA/R$ be a deformation of~$A$.\label{AA} By that we mean that~$\AA\arr
\Spec(R)$ is an abelian scheme and that there are given closed
immersions:
\[ \xymatrix@M=10pt{\AA \ar[d] & A \ar[d]\ar@{_{(}->}[l] \\ \Spec(R) & \Spec(k)\ar@{_{(}->}[l]}\]
By a fundamental result of Grothendieck, we have an inclusion of
rings
\[ \End_R(\AA) \injects \End_k(A).\]
Let us define the magnitudes  (a-priory possibly infinite)
\[ i(\AA/R)  = [\End_k(A): \End_R(\AA)],\]
and
 \begin{align} \label{iI}
 \geri(R) & = \inf\{i(\AA/R): \AA/R \text{\rm \; a deformation of
\;}A \}, \\ \qquad \gerI(R) & = \sup\{i(\AA/R): \AA/R \text{\rm \; a
deformation of \;}A \}.
\end{align}

\id These depend on~$A$ but we suppress that from
the notation. We are interested in studying $i(\AA/R), \geri(R)$
and~$\gerI(R)$. Although we provide below some general results, our
focus later is on the case of elliptic curves. The general case certainly
deserves further study, but it will not be carried out here. 

\begin{prop}\label{prop:basic estimate on index} The quantity~$i(\AA/R)$ is finite and is a power
of~$p$. So are~$\geri(R)$ and~$\gerI(R)$. 
The following inequalities hold.
\[1 \leq \geri(R) \leq \gerI(R)  \leq  p^{(r-1) t_R \lceil (n_R-1)/(p-1)\rceil}.\]
\end{prop}
\begin{cor} Let~$K$ be a CM field and~$\calO$ an order of~$K$. Let~$A \arr \Spec(R)$ be an
abelian scheme over a dvr $(R, \germ_R)$ whose residue field
is a perfect field~$k$ of characteristic~$p$,  and suppose that we are given
an optimal embedding~$\calO \injects \End_R(A)$. Let
$\calO'\supseteq \calO$ be the optimally embedded order of~$K$ in
$\End_k(A \otimes k)$. Then~$[\calO':\calO]$ is a power of~$p$.
\end{cor}
\begin{exa} \label{exa: supersingular}Suppose that~$E$ is an elliptic curve over a number
field~$M$ with complex multiplication by an optimally embedded order
$\calO$ of a quadratic imaginary field~$K$. Let~$\gerp$ be a prime
ideal of~$M$ of residue characteristic~$p$, and assume that~$E$ has
good reduction modulo~$\gerp$, denoted~$E'$, and that the conductor
of~$\calO$ is prime to~$p$. Then~$\calO$ is optimally embedded
in~$\End(E')$.

On the other hand, the conductor always becomes smaller when it is
divisible by~$p$. Suppose that $E$ has
supersingular reduction,~$\calO_K = \ZZ[\delta]$ and~$\calO =
\ZZ[pr\delta]$, where~$r\in \ZZ$. One verifies that~$pr\delta$ has
degree divisible by~$p^2$. Since~$E'$ is supersingular any isogeny
of degree~$p^2$ vanishes on~$E'[p]$ and it follows that~$r\delta$
is also an isogeny of~$E'$. It is an interesting situation. Because~$\calO$ is optimally embedded in $\End(E)$, the
kernel of the multiplication-by-$p$ map on the finite flat group
scheme~$\Ker[pr\delta]$ has order~$p$ generically, but order~$p^2$
modulo~$p$.
The same happens in the ordinary case;
see \S~\ref{subsubsec: ordinary elliptic curves}. This example is well-known but is usually proven by other techniques. See, for example \cite[Theorem 5, \S~13.2]{LangEllipticFunctions}.
\end{exa}

\begin{prop} \label{prop: algebraic properties of deformations}
Let $A$ be an abelian variety over an algebraically closed field $k$
of characteristic $p$.
\begin{enumerate}
  \item Let~$\calO \subset \End(A)$ be a set. Let~$R^u$ be the universal formal deformation space of~$A$.
  There is closed subscheme~$Z_\calO$ which is universal for the
  property of extending~$\calO$ to a deformation.
  \item Let~$n$ be an integer. There is a closed subscheme that is
  universal for deformations~$\AA$ of~$A$ such that
 ~$[\End(A):\End(\AA)] \vert p^n$. (The same holds true if we wish to
  work with elementary divisors for the quotient abelian group
 ~$\End(A)/\End(\AA)$.)
\end{enumerate}
\end{prop}
\id Proposition \ref{prop: algebraic properties of deformations} is folklore. The first assertion is proven in
\cite[Lemma 4.3.5]{Dok}. The proof consists of verifying
Schlessinger's criteria for pro-representability. 
The second assertion follows immediately from the first given that
there are only finitely many subrings of a given index (let alone
of given elementary divisors) and they are all
finitely generated as $\ZZ$-modules.

The proof of Proposition \ref{prop:basic
estimate on index} is given below, after we review Grothendieck's crystalline
deformation theory.

\subsection{Crystalline deformation theory}\label{subsec: crystalline deformation theory}
Our main reference here is Grothendieck's monograph \cite{Grothendieck}. First recall the notion of \emph{divided
power structure (d.p.)} on a pair $(R, I)$ consisting of a ring
$R$ and an ideal $I\normal R$ (loc. cit. Chapitre IV, \S 1.1).
These are functions $\gamma_n: I \arr I, n = 1, 2, 3, \dots$ that
``behave like" $x^n/n!, n = 1, 2, 3, \dots$, that is, the
following properties hold true: 
\begin{enumerate}
\item $\gamma_1(x) = x$; \item
$\gamma_n(x+y) = \gamma_n(x) + \sum_{i = 1}^{n-1}
\gamma_{n-i}(x)\gamma_i(y) + \gamma_n(y)$; \item $\gamma_n(xy) =
x^n\gamma_n(y)$ for $x\in R, y\in I$; \item $\gamma_m(\gamma_n(x))
= \frac{(mn)!}{(n!)^m m!}\gamma_{mn}(x)$; \item
$\gamma_m(x)\gamma_n(x) = \frac{(m+n)!}{m! n!}\gamma_{m+n}(x)$.
\end{enumerate}

\id The axioms imply the identities \[x^n = n! \gamma_n(x), \qquad x\in I, \; n = 1, 2, 3 \dots.\]
Hence, if $R$ is an integral domain, whose quotient field is of
characteristic $0$, there is at most one d.p. structure on $I$. It is
given by $\gamma_n(x) = x^n/n!$. This d.p. structure is well defined if $x^n/n!\in
I$ for all $x\in I, n=1, 2, 3, \dots$. A divided powers structure
is called \emph{nilpotent} if there is an $N$, such that for any
positive integers $a_1, \dots, a_r$ with $\sum_{i=1}^r a_i \geq N$
and elements $x_1, \dots, x_r $ of $I$, we have $\gamma_{a_1}(x_1)
\gamma_{a_2}(x_2)\cdots \gamma_{a_r}(x_r) = 0$.

\begin{exa}\label{exa: pd structure} Let $p$ be a prime.
Suppose that $I^p = 0$ and that $1, 2, 3, \dots p-1$ are invertible
in $R$, then we may define $\gamma_n (x) = x^n/n!, n = 1, 2, \dots,
p-1$ and $\gamma_n(x) = 0, n\geq p$. This is a nilpotent d.p. structure with
$N=p$. 
\end{exa}
\begin{exa} Let $(R, I)$ be a discrete valuation ring of mixed
characteristic $(0, p)$ and uniformizer $\pi$. We assume that
$\val(p) = 1$ and $\val(\pi) = 1/e$. We have $\pi^n/n! \in I$ if
and only if $n/e \geq (n-s_n)/(p-1)$, where $s_n$ is the sum of
the digits in the $p$-adic development of $n$. See loc. cit. IV \S
1.3.) That is, $\pi^n/n! \in I$ for all $n\geq 1$ iff $e \leq p-1$. 

If $R$ has a d.p. structure, i.e. $e \leq p-1$, then we have an
induced d.p. structure on $(R/I^N, I/I^N)$, which is nilpotent of level $N$ if
$e < p-1$. We say then that the d.p. structure on $(R, I)$ is
topologically nilpotent. The condition $e < p-1$ is necessary for that.

\end{exa}

The theorem that we need is in loc. cit. V \S 4. Following the notation
there, we use $\DD^\ast(A)_S$ to denote the relative de Rham
cohomology $\HH^1_{\rm dR}(A/S)$. It will take us too long to
define the notions of the crystalline site and crystals in
general. For that see loc. cit.. We just note a particular example
of the theorem: Let $S \injects S'$ be a closed immersion of affine schemes, $\Spec(R) \arr \Spec(R')$,
where $R' \arr R$ is a surjective ring homomorphism with kernel $I$, such that $I$ is
equipped with nilpotent d.p.. This is an example of a nilpotent thickening of $S$ by $S'$. For instance, in Schlessinger's theory one
considers the case where the rings $R, R'$, are local rings with maximal
ideals $\germ, \germ'$, respectively, $R' \arr R$ is a local
homomorphism whose kernel is principle, say equal to $(t)$ and
$\germ' t = 0$. Note that this implies that $t^2 = 0$. One then
has a canonical nilpotent d.p. structure on $(t)$ given by
$\gamma_1(x) = x$ and $\gamma_n(x) = 0, n = 2, 3, \dots$.
\begin{thm}
Let $S$ be a scheme and $S'$ nilpotent thickening of $S$ with d.p.
which is locally nilpotent. Consider the natural functor from
abelian schemes over $S'$ to the category of couples $(A, {\rm
Fil}^1)$ of an abelian scheme $A$ over $S$ and a submodule,
locally a direct summand, ${\rm Fil}^1$ of $\DD^\ast(A)_{S'}$,
which is a prolongation of ${\rm Fil}^1\DD^\ast(A)_S = \uomega
_A$. This functor is an equivalence of categories.
\end{thm}
\begin{exa}
Let $K$ be a quadratic imaginary field and $\calO_{K, m}$ be the
order of conductor~$m$ in~$K$ and say $p^a \| m$, $m= p^an$. Let $E$ be a
superspecial elliptic curve over $\bar{\FF}_p$ with an action of
$\calO_{K,n}$. One may wish to calculate the deformations of $E$ to which
the action of the subring $\calO_{K, m}$ of $\calO_{K, n}$ extends. (Note that this is the
general situation by Example~\ref{exa: supersingular}.)
Unfortunately, such a calculation is not accessible via crystalline
deformation theory. For example, consider such deformations to
characteristic zero that are defined over a d.v.r. $R$ with d.p..
Every such deformation $\EE$ defines then a submodule of
$H^1_{\text crys}(E/R) = H^1_{\text crys}(E/W(\fpbar))\otimes R$,
which is a direct summand of rank $1$ extending the Hodge-de Rham
filtration on $H^1_{\rm dR}(E/\fpbar)$. We assume such a
deformation exists, which means that there are two embeddings
$\iota_1, \iota_2:\calO_{K,n} \arr R$, the first induced from the
action of $\calO_{K, n}$ on the tangent space and the second is its
Galois twist. We have $H^1_{\text crys}(E/R) = \calO_{K,n} \otimes_\ZZ
R \injects R\oplus R$ by $(\iota_1\otimes 1, \iota_2 \otimes 1)$.
If $p\neq 2$ is unramified this is an isomorphism of rings and
under this isomorphism the order of conductor $m$ is sent to the
subring $\calO_a:= \{(x, y) \in R\oplus R: x \equiv y
\pmod{p^a}\}$, generated as an $R$-module by $(1, 1), (p^a,
-p^a)$. A direct summand $R$-module of rank $1$ of $R^2$ is given
by $(x, y)$ with either $x$ or $y$ a unit. To be preserved under
$\calO_a$ we must have $x=0$ or $y=0$. Thus, we see that there is
a unique deformation for which the action of $\calO_a$ extends, and then also
$\calO_0$ acts. The conclusion is that elliptic curves over a
finite extension of $\QQ_p$ on which $\calO_{K, m}$ acts optimally
are not defined over a base affording d.p.. That is, the
ramification index is at least $p$. Of course the theory of
complex multiplication and class field theory give more precise
results. It remains an interesting problem to actually calculate the closed 
subscheme of the deformation space of $E$ to which the action of $\calO_{K, m}$ extends.
\end{exa}

\subsection{Proof of Proposition \ref{prop:basic estimate on
index}} We remark that there are many cases where $\geri(R)=1$.  An obvious example is when~$R =
k[\epsilon]$ and we take the constant deformation~$\AA = A \otimes_k
k[\epsilon]$. A more interesting example can be given in the case of
ordinary abelian varieties, see \S\ref{subsec: ordinary AV}.

Let~$(R, \germ_R)$ be a local artinian ring with residue field~$k =
R/\germ_R$. Let~$n_R$ be the minimal positive integer such that
$\germ_R^{n_R} = 0$, as before. We define successively rings
\[R_0 = R/\germ_R, \quad
R_1 = R/\germ_R^{1+(p-1)}, \quad R_2 = R/\germ_R^{1+2(p-1)}, \cdots, \quad
R_\ell = R/\germ_R^{1+\ell(p-1)},\] where~$\ell = \lceil (n_R -
1)/(p-1)\rceil.$ There are canonical surjections
\[ R_\ell \surjects R_{\ell-1} \surjects \dots \surjects R_{1} \surjects
R_{0},\] and we let~$I_j = \germ_R^{1+ (j-1)(p-1)}/\germ_R^{1+
j(p-1)}$,~$j = 1,2, \dots, \ell$, be the kernel of the surjection
$R_j \arr R_{j-1}$. We note that~$I_j^p = 0$ in~$R_j$ and hence
the morphism
\[ \Spec(R_{j-1}) \injects \Spec(R_j), \]
is a nil-immersion with canonical divided powers structure
as in Example~\ref{exa: pd
structure}. Let~$t_R$ be the minimal power of~$p$ such
that~$p^{t_R}\in \germ_R^{p-1}$. Then~$p^{t_R} I_j = 0$ in~$R_j$.

Now, by arguing inductively on~$j$, we reduce to the following
situation. Let~$A \arr \Spec(R_{j-1})$ be an abelian scheme of
relative dimension~$g$ and let~$\AA \arr \Spec(R_j)$ a deformation
of it. We need to show that~$[\End(\AA):\End(A)]$ is finite and is
equal to a power of~$p$. By crystalline deformation theory, the
closed immersion of abelian schemes~$A \injects \AA$ corresponds
functorially to a diagram
\[\xymatrix@R = 2pt{R_j^{2g} \ar@{>>}[rr]&& R_{j-1}^{2g}  = \HH^1_{\rm dR}(A/R_{j-1}) \\ \cup & & \cup \\
\omega_{j} \ar@{>>}[rr] & & \omega_{j-1}  = H^0(A, \Omega^1_{A/R_{j-1}})   }
\]
where~$\omega_{j-1}, \omega_{j}$ are free~$R$-modules that are
rank~$g$ direct summands of~$R_{j-1}^{2g}$ and~$R_{j}^{2g}$,
respectively. In particular, an endomorphism~$f\in \End(A)$ acts
canonically and compatibly on~$R_{j-1}^{2g}$ and~$R_{j}^{2g}$ and
preserves~$\omega_{j-1}$. It extends to an endomorphism of~$\AA$
if and only if it preserves~$\omega_j$. Consider then~$p^{t_R}f$.
Let~$x \in \omega_j$ and choose a~$y\in \omega_j$ such that~$f(x)
= y \pmod{I_j}$, i.e. equality holds between the images of~$f(x)$
and $y$ in~$\omega_{j-1}$. Then~$f(x) - y$ is in the kernel of the
homomorphism~$\omega_{j} \arr \omega_{j-1}$, which is certainly
contained in~$I_jR_j^{2g}$. Since~$p^{t_R} I_j = 0$, we conclude
that~$p^{t_R}f(x) - p^{t_R}y = 0$ and so~$p^{t_R}f(x) \in
\omega_j$.

We note that the same reasoning gives that if $s\cdot f$ extends to an
endomorphism of~$\AA$ and~$(p, s) = 1$ then~$f$ also extends,
because~$s$ is invertible in~$R$. This can also be concluded from
the Serre-Tate theory that gives~$\End(\AA) = \{ f\in
\End(\AA[p^\infty]): f\vert_{A[p^\infty]} = g\vert_{A[p^\infty]} \text{ for some } g\in \End(A)\}$,
namely, the endomorphisms of~$\AA$ are the endomorphisms of its
$p$-divisible group whose restriction to the~$p$-divisible group of
$A$ is induced from a bona fide endomorphism of~$A$.

We have~$r-1$ appearing in the power of~$p$ in the statement of the
proposition, namely there is ``a saving of~$1$", because~$\ZZ\subseteq
\End(\AA)$ and is a direct summand in it (as an abelian group).

\subsection{Ordinary abelian varieties}\label{subsec: ordinary AV} Strictly speaking, the following is not needed for the main results of our paper, as we shall need to consider supersingular abelian varieties. It is included here for the sake of completeness.
Our main reference is Katz's paper \cite{Katz}. 

Let~$k$ be an
algebraically closed field of characteristic~$p$ and let~$A$ be an
ordinary abelian variety over~$k$. We let $T_p(A) = \ilim A[p^n](\fpbar)$ and $V_p(A)  = T_p(A) \otimes_\ZZ \QQ$. Then the deformations of~$A$
are pro-represented by a formal torus over the Witt vectors
of~$k$, $\widehat{\GG}_m^{g^2}\arr {\rm Spf}(W(k))$. One fixes
isomorhisms~$T_p(A) \cong \ZZ_p^g$ and~$T_p(A^t) \cong \ZZ_p^g$.
The deformations~$\AA\arr \Spec (R)$ of~$A$ to a local artinian
ring~$R$ with residue field $k$ are in functorial bijection with
\[ \Hom(T_p(A) \otimes_{\ZZ_p} T_p(A^t) , \widehat{\GG_m}(R)) =
\Hom(\ZZ_p^g \otimes_{\ZZ_p} \ZZ_p^g, 1 + \germ_R), \] and so can
viewed as bilinear forms on~$\ZZ_p^g \times \ZZ_p^g$ with values in
the multiplicative group~$1+ \germ_R$. We denote the bilinear form
corresponding to a deformation~$\AA \arr \Spec(R)$ of~$A$ by
$\langle\;, \; \rangle_\AA$. In particular, an endomorphism~$f\colon A
\arr A$ extends to~$\AA$ if and only if
\[ \langle fx, y\rangle_\AA = \langle x, f^ty\rangle_\AA , \]
where we use~$f$ to denote also the endomorphism of~$\ZZ_p^g$
induced from~$f$ via the chosen identification~$\ZZ_p^g \cong
T_p(A)$, and similarly for $f^t\colon A^t \arr A^t$.

The canonical lift of~$A$ to~$R$ is the deformation~$\AA$ such
that $\langle\;, \; \rangle_\AA$ is the trivial pairing
(identically~$1$) and we see the well-known fact that for this
deformation~$\End(A) = \End(\AA)$ and so $\geri(R) = 1$. We also
see that if~$p^a$ is the exponent of the multiplicative group~$1 +
\germ_R$ then if~$f\in \End(A)$ then~$p^af$ extends to any
deformation of~$A$ to~$R$.

Let us assume that~$A$ is a simple ordinary abelian variety with
complex multiplication. This is the case for example if~$A$ is
simple and defined over~$\fpbar$. In this case,~$\End^0(A)$ is a CM
field~$K$; let~$\calO\subset K$ be the order optimally
embedded in~$\End(A)$. Since the action of~$\calO$ lifts to the
canonical lift of~$A$ and so to characteristic zero, it follows
from the theory of complex abelian varieties that~$f^t$ is just
given by $\bar{f}$ (complex conjugation applied to~$f$) if one chooses any polarization to identify $V_p(A)$ with $V_p(A^t)$. We find
that~$f$ extends to a deformation~$\langle\;, \; \rangle_\AA$ if
and only if
\[ \langle fx, y \rangle_\AA = \langle x, \bar{f}y \rangle_\AA .\]
For a fixed~$f$ this is a linear equation in the matrix coefficients
of~$\langle\;, \; \rangle_\AA$. In general, this can be used to
explicitly determined the closed subscheme~$Z_\calO$ and one sees
that they come out formal sub-tori. To illustrate we consider the
one-dimensional case.

\subsubsection{Ordinary elliptic curves} \label{subsubsec: ordinary elliptic curves}
In this case~$\langle\;, \; \rangle_\AA$ is just an
element~$\langle 1, 1 \rangle_\AA=q_\AA \in 1 + \germ_R$. We claim
that since~$p$ is split in~$K$, the identification of~$f$ with an
element of~$\ZZ_p$ (viewed as a homomorphism of~$\ZZ_p\cong
T_p(A)$) is just viewing $f$ as an element of~$\ZZ_p$ via the
embedding~$K \arr \QQ_p$ determined by one of the prime ideals
of~$K$ above~$p$ (it would not matter which). An endomorphism~$f$
extends to this deformation if and only if~$q_\AA^{f - \bar{f}} =
1$. If the order of~$q_\AA$ is $p^a$ then this says that~$f -
\bar{f}$ is divisible by~$p^a$. Thus, $\End(\AA)$ is the
intersection of~$\End(A)$ with the order of conductor~$p^a$
in~$K$.

In particular, we see that if~$\End(A) = \calO_K$ and the exponent
of~$1+\germ_R$ is~$p^b$ then for every~$0 \leq a \leq b$ there is a
deformation~$\AA$ with~$\End(A)$ the order~$\calO_{K,p^a}$ of
conductor~$p^a$. Conversely, the set of~$q\in 1 + \germ_R$ to which
$\calO_{K,p^a}$ extend is defined by the equation~$q^{p^a} = 1$.
These are the closed sets appearing in Proposition~\ref{prop:
algebraic properties of deformations}. This is of interest: if~$R$
is a dvr of mixed characteristic and we are trying to find a
deformation~$\AA$ of~$A$ to~$R$ such that $\End(\AA) = \calO_{K,
p}$, say, then we must introduce ramification. We need a~$p$-th root
of unity. In fact, a closer look reveals that for an elliptic curve
$E/\overline{\FF}_p$ with CM by $\calO_K$, there are precisely $p^a
- p^{a-1}$ deformations $\EE$ to characteristic zero such that
$\calO_{K, p^a}$ is optimally embedded in $\End(\EE)$. They are
provided by the primitive $p^a$-roots of unity. The fact that there
is a unique deformation for which we get an action of $\calO_K$ implies
that no two singular moduli for $\calO_K$ are congruent modulo $p$.
The relation between the class numbers $h, h_{p^a}$, of $\calO_K$ and
$\calO_{K, p^a}$ respectively, is (assume for simplify that $K\neq \QQ(i),
\QQ(\omega)$):  $h_{p^a} = h \times (p^a - p^{a-1})$ (cf.
\cite[Theorem 7, \S 8.1]{LangEllipticFunctions}) and so we conclude that for every
elliptic curve $\EE$ over a $p$-adic ring with action of $\calO_{K,
p^a}$ the reduction has action by $\calO_K$, as in the supersingular
case. This is classical (see \cite[Theorem 5, \S~13.2]{LangEllipticFunctions}). Note
also that we get very precise information about the $p$-adic
completion of the field generated by the singular moduli for
$\calO_{K, p^a}$ and so about the ramification at $p$ of this field.

\subsection{Supersingular elliptic curves} Let~$k = \overline{\FF}_p$. Let~$V$ be
a complete dvr containing the completion of the maximal unramified extension~$W(k)$ of
$\ZZ_p$ and of ramification index~$e_V < p-1$. Then~$V \arr k$ has
topologically nilpotent divided powers coming from~$\gamma_n(x) = x^n/n!$ in~$V$.
In fact, using results of Zink (see remarks on page 6 of \cite{Zink}), it is enough to assume
that~$e_V \leq p-1$ and so that~$V$ has divided powers structure
(not necessarily nilpotent). The advantage is that~$p=2$ is allowed
too, as long as it is unramified.

Let~$E/k$ be a supersingular elliptic curve. Recall that~$\End(E)$
is a maximal order in the rational quaternion algebra~$B_{p,
\infty}$ ramified only at~$p$ and~$\infty$. We apply Grothendieck's
crystalline deformation theory to study for a deformation~$\EE/R$ of
$E$ the index~$[\End(E):\End(\EE)]$.

\begin{lem}
The following holds:
\begin{enumerate}
  \item~$[\End(E):\End(\EE)] = [\End(E)\otimes \ZZ_p:\End(\EE)\otimes
  \ZZ_p]$.
  \item~$\End(E)\otimes \ZZ_p \cong \left\{\left(\begin{smallmatrix}
  a & b \\ pb^\sigma & a^\sigma
  \end{smallmatrix} \right): a, b \in W(\FF_{p^2})\right\} = : D$, where
 ~$\sigma$ is the Frobenius automorphism.
  \item There is a basis~$\{e_1, e_2\}$ of~$H^1_{\rm crys}(E/W(k))$ with respect to which the
  action of~$\End(E)$ is given as matrices as in \emph{(2)}.
\end{enumerate}
\end{lem}
\begin{proof} The first claim holds, because by Proposition~\ref{prop:basic estimate on
index} the index is a power of~$p$. To prove the rest, we note
that $E$ can be defined over $\FF_{p^2}$ and so $H^1_{\rm
crys}(E/W(k))$ has a basis $e_1, e_2$ defined over $\FF_{p^2}$
such that the $\sigma$-linear Frobenius map is given by the matrix
$\left(\begin{smallmatrix} 0 & 1 \\ p & 0
\end{smallmatrix} \right)$ with respect to this basis. Now, we
have $\End(E) \otimes \ZZ_p \cong \End(E[p^\infty])$ (this uses
Tate's theorem at $p$ plus the fact that the Galois action, being
in the commutant of the quaternion algebra $\End^0(E)$ is
central), which is in turn isomorphic to the endomorphisms of
$H^1_{\rm crys}(E/W(k))$ commuting with $\left(\begin{smallmatrix}
0 & 1 \\ p & 0
\end{smallmatrix} \right)$. The condition then comes out
\[ \begin{pmatrix}
a & b \\ c & d
\end{pmatrix}
\begin{pmatrix}
0 & 1 \\ p & 0
\end{pmatrix}
 =
\begin{pmatrix}
0 & 1 \\ p & 0
\end{pmatrix}
 \begin{pmatrix}
a^\sigma & b^\sigma \\ c^\sigma & d^\sigma
\end{pmatrix},\]
i.e.,
\[ \begin{pmatrix}
pb & a \\ pd & c
\end{pmatrix} = \begin{pmatrix}
c^\sigma & d^\sigma \\ pa^\sigma & pb^\sigma
\end{pmatrix},\]
from which now both (2) and (3) follow.
\end{proof}

\begin{prop}\label{prop: condition to extend endo in ss case}
In the basis~$\{ e_1, e_2\}$ the Hodge filtration on~$H^1_{\rm
dR}(E/k)$ is given by the image of the span of~$e_1$ in~$H^1_{\rm
crys}(E/W(k))$.

Let~$n$ be a positive integer. A deformation~$\EE$ of~$E$ to~$R := V/\germ_V^n$, equipped with its canonical divided powers
structure is given by the span of a
vector in~$R^2 = H^1_{\rm dR}(E/R)$of the form~$(1, y)$ with~$y\in
\germ_R$ and so we denote it $\EE_y$. In particular, an element~$\left(\begin{smallmatrix}
  a & b \\ pb^\sigma & a^\sigma
  \end{smallmatrix} \right)$ in~$\End(E)\otimes \ZZ_p$ extends to
  the deformation~$\EE_y$ if and only if
  \[ \left(\begin{smallmatrix}
  a & b \\ pb^\sigma & a^\sigma
  \end{smallmatrix} \right)\left(\begin{smallmatrix}
  1 \\ y
  \end{smallmatrix}\right) \in \Span_R\langle \left(\begin{smallmatrix}
  1 \\ y
  \end{smallmatrix}\right) \rangle.\]
\end{prop}
\id (The proof is straightforward.)

\begin{thm}\label{thm: bounds on indices of endo rings under deformation}
Let~$V$ be as in Proposition~\ref{prop: condition to extend endo in
ss case}. Then
\[ p^{2(\lceil n/e_V\rceil -
1)} \leq \geri(V/\germ_V^n) \leq \gerI(V/\germ_V^n) \leq
p^{3(n-1)}.\]
Furthermore, these bounds are optimal.
\end{thm}
\begin{proof}
Let~$R = V/\germ_V^n$ and~$D_y = \left\{\left(\begin{smallmatrix}
  a & b \\ pb^\sigma & a^\sigma
  \end{smallmatrix} \right): a, b \in W(\FF_{p^2}), by^2 + (a - a^\sigma)y - pb^\sigma \equiv 0
  \pmod{\germ_V^n}\right\}$. Note that
  \[ \left(\begin{smallmatrix}
  a & b \\ pb^\sigma & d
  \end{smallmatrix} \right)\left(\begin{smallmatrix}
  1 \\ y
  \end{smallmatrix}\right) \in \Span_R\langle \left(\begin{smallmatrix}
  1 \\ y
  \end{smallmatrix}\right) \rangle \Leftrightarrow by^2 + (a - a^\sigma)y - pb^\sigma \equiv 0
  \pmod{\germ_V^n}, \]
  and so~$D_y$ is a ring, identified with~$\End(\EE_y) \otimes
  \ZZ_p$. We note that the map
  \[ \varphi: D \arr R, \qquad \left(\begin{smallmatrix}
  a & b \\ pb^\sigma & a^\sigma
  \end{smallmatrix} \right) \mapsto by^2 + (a - a^\sigma)y -
  pb^\sigma, \]
is a~$\ZZ_p$-linear map whose kernel is
$D_y$. We shall give a lower bound on $[D: D_y]$ by bounding $\sharp D/D_y  = \sharp \varphi(D)$  from below.

Suppose that $p\neq 2$. Let~$\{1, \alpha\}$ be a~$\ZZ_p$ basis to~$W(\FF_{p^2})$ such that
$\alpha^\sigma = -\alpha$ and~$\alpha$ is a
unit. We normalize the $p$-adic valuation so that $\val(p) = 1$. If
$A, B \in \ZZ_p$ then~$\val(A + B\alpha) = \val(A - B\alpha) =
\min\{ \val(A), \val(B)\}$. We note that
\[\varphi(D) = \Span_{\ZZ_p}\{ y^2 - p, \alpha(y^2 + p), \alpha
y\}. \] Consider the linear combination
\[ C(A, B) = A(y^2 - p) + B\alpha(y^2+p), \qquad A, B \in \ZZ_p.\]
We note that
\[\val(C(A, B)) < \gamma:= \frac{n}{e_v} \; \Longrightarrow \; C(A, B) \neq 0  \quad \text{(in $R= V/\germ_V^n$)}.\] Let $y$ denote also some lift of $y\in R$ to $V$.
We distinguish cases:
\begin{enumerate}
\item $\val(y) > 1/2$.$\quad $ We write
\[ C(A, B) = y^2(A + B\alpha) - p(A - B\alpha).\]
Since $\val(A+B\alpha) = \val(A - B\alpha)$ and $\val(y^2) >1$, we find that
\[\val(C(A, B)) = 1 + \min\{\val(A), \val(B).\]
It follows that as long as either $\val(A)$ and $\val(B)$ are both less than $\gamma-1$, or, equivalently, less or equal to $\lceil \gamma \rceil -2$, we have $C(A, B) \neq 0$. Equivalently, the group homomorphism
\[ \ZZ/p^{\lceil \gamma \rceil -1} \times \ZZ/p^{\lceil \gamma \rceil -1} \Arr R, \qquad (A, B) \mapsto C(A, B), \]
is injective. We conclude that $\sharp \varphi(D) \geq p^{2(\lceil \gamma \rceil -1)}$.
\item $\val(y) < 1/2$. $\quad$ In this case $\val(y^2) < 1$ and so we find that $\val(C(A, B)) = \val(y^2) + \min \{ \val(A), \val(B)\} < 1 + \min \{ \val(A), \val(B)\}$ and we get the same estimate (we do not bother with improving it).
\item $\val(y) = 1/2$. $\quad$ In this case we note that either $\val(y^2 - p) = 1$ or $\val(y^2+p) = 1$. So, either $\val(y^2 - p) = 1$ or $\val(\alpha(y^2+p)) = 1$. We assume that $\val(y^2 - p) = 1$, as the other case is entirely similar. In this case we consider the linear combination
\[ D(A, B) = A(Y^2-p) + B\alpha y, \qquad A, B\in \ZZ_p.\]
Since $\val(A(y^2 - p)) = \val(A) + 1$ and $\val(B\alpha y) = \val(B) + 1/2$ and, in particular, are never equal, we find that
\[ \val(D(A, B)) = \min \{ 1 + \val(A), 1/2 + \val(B)\},\]
and, as long as $\val(A) < \gamma - 1$ or $\val(B) < \gamma - 1/2$, $D(A,B) \neq 0 \in R$. Weakening the conclusion to $\val(A) < \gamma - 1$ and $\val(B)< \gamma - 1$, we find the previous estimate.
\end{enumerate}

\medskip
\id Next consider the case $p=2$. Represent $W(\FF_{p^2})$ as $W(\FF_p)[t]/(t^2 + t -1)$. A key point turn out to be that $\alpha = t - t^\sigma$ is a unit and for $A, B \in \ZZ_p$ we have
\[ \val(A+ Bt) = \val(A+Bt^\sigma) = \min\{\val(A), \val(B)\}.\]
One checks that
\[ \varphi(D) = \Span_{\ZZ_p}\{ y^2 - p, t y^2 - p t^\sigma, \alpha y\}. \]
We let now
\[C(A, B) = A(y^2 - p) + B(ty^2 - pt^\sigma), \qquad D(A, B) = A(y^2 - p) + B\alpha y, \qquad A, B \in \ZZ_p.\]
As before the analysis is divided into three cases: (i) $\val(y)>1/2$, (ii) $\val(y)>1/2$ and $\val(y) = 1/2$, which are treated in an entirely similar manner. In cases (i) and (ii) it is helpful to write $C(A, B) = y^2(A+Bt) - p(A + Bt^\sigma)$ and in case (iii) first one argues that we can not have both $\val(y^2-p)>1$ and $\val(ty^2 - pt^\sigma)$; assuming without loss of generality that $\val(y^2 - p) = 1$, one uses $D(A, B)$ 
for the estimate, as before.

\medskip

The
upper bound on~$\gerI(V/\germ_V^n)$ follows using the same technique as in the proof of
Proposition~\ref{prop:basic estimate on index}, which itself gives a slightly weaker exponent ($3\cdot \lceil \frac{p-1}{e_V}\rceil \cdot \lceil \frac{n-1}{p-1}\rceil$).

\bigskip

\id We now show that the bounds
in Theorem~\ref{thm: bounds on indices of endo rings under deformation}
are optimal.

In \cite{Gross} Gross studies the deformations of a supersingular
elliptic curve for which an action of a ring of integer of some
fixed quadratic imaginary field extends~$K$. He obtains that the
endomorphism ring of such a deformation over~$W(\overline{\FF}_p)$ is
precisely~$\calO_K + p^{n-1}\End(E)$ and, in particular, of index
$p^{2(n-1)}$ in~$\End(E)$. This conforms nicely with our theorem that states this is the best
possible. 

A concrete case of a deformation where this bound is achieved is the case when~$y=p>2$ and~$n=2$. Note that in
that case the target of the map~$\varphi$ is 
$W(\FF_{p^2})/(p^2)$, which has cardinality~$p^4$. It is also easily
verified that~$\varphi(D)$ is generated in this case over $\ZZ_p/(p^2)$ by~$p$ and
$\alpha p$ and so has cardinality~$p^2$. We conclude that~$D_y$ has
index~$p^2$. In fact,~$D_y$ are the matrices in~$D$ defined by the
condition~$a - a^\sigma = b^\sigma \pmod{p}$. Thus~$D_y$ contains
$pD$ and modulo~$pD$ it is given by the basis~$(a, b) = (1, 0)$ and
$(a, b) = (\alpha, -2\alpha^\sigma)$. If we take any quadratic
imaginary field~$K = \QQ(\sqrt{-d})$ ($d$ square-free integer) in
which~$p$ is inert and let~$\alpha = \sqrt{-d}$ then we find one of
the deformations considered by Gross for~$K$.

Now consider again the case of a general~$(V, \germ_V)$ but which is unramified over $\QQ_p$, where $p>2$. Suppose
that there are~$A, B, C \in \ZZ_p$ such that
\begin{equation}\label{eqn1}A(y^2 - p) + B\alpha(y^2+p) + C \alpha y =
(B\alpha+ A)\cdot y^2 + C\alpha \cdot y + p(B\alpha - A) \equiv 0
\pmod{\germ_V^n}
\end{equation}
Choose~$y$ to have
valuation~$1$, equal to the valuation of~$p$. If~$\val(B\alpha
+ A)< n-1$, Equation~(\ref{eqn1}) implies that~$\val(C) =
\val(B\alpha - A) = \val(B\alpha + A)$. We get an equation \[y^2 +
\frac{C\alpha}{B\alpha + A} y + p\frac{B\alpha - A}{B\alpha+ A}
\equiv y(y + \frac{C\alpha}{B\alpha + A}) \equiv 0 \pmod{\germ_V},\]
which is an equation with integral coefficients that holds in
$V/\germ_V$. By Hensel's lemma it follows that the polynomial~$Y^2 +
\frac{C\alpha}{B\alpha + A} Y + p\frac{B\alpha - A}{B\alpha+ A}$ in the variable $Y$ has a solution, say $y_0$, in~$W(\FF_{p^2})$ lifting~$0$. Moreover, if $y_0'$ is the other solution (so that $f(Y) = Y^2 +
\frac{C\alpha}{B\alpha + A} Y + p\frac{B\alpha - A}{B\alpha+ A} = (Y-y_0)(Y-y_0')$) then $\val(y_0 - y_0') = 0$, as $y_0'$ reduces to a unit modulo the maximal ideal. Now, let us choose $y$ so that in addition $y\not\in \WW(\FF_{p^2})$ and for every Galois conjugate $y'$ of $y$ the difference $y-y'$ is a unit. For example, $y$ could be $p\zeta$ where $\zeta$ is an $\ell$-t root of unity where $\ell \neq p$ is a large enough prime. Note that $f(y) = (y-y_0)(y-y_0') \equiv 0 \pmod{\germ_V^2}$. Since $\val(y) = 1$, $y - y_0'$ is a unit and so $\val(y - y_0) \geq 0$. It follows that $y_0$ is closed to $y$ than any of $y$'s conjugates and so, by Krasner's lemma, $y \in \QQ_p(y_0) = W(\FF_{p^2}) \otimes_{\ZZ_p} \QQ_p$ and that is a contracdiction.

Thus, $\val(B\alpha + A)\geq n-1$. We get then that~$\min\{\val(A),
\val(B)\} = n-1$ and then Equation~(\ref{eqn1}) give that~$\val(C)
\geq n-1$ as well. This shows that for such a choice of~$y$ we get that
$\sharp\; \varphi(D) \pmod{p^n}$ is~$3(n-1)$ and so the upper bound
in the theorem can be achieved. This shows that the bounds 
are optimal. \end{proof}

\subsection{Bound in the case of high ramification} 
\label{subsec: high ramification}
As above, let $V$ be a discrete valuation ring, which is a finite extension of $\WW(\fpbar)$ with absolute ramification index $e_V$. As before, let $E$ be a supersingular elliptic curve over $\fpbar$. The purpose of this section is to provide a lower bound on $\geri(V/\germ_V^n)$ (defined relative to deformations $\EE$ of $E$ to $V/\germ_V^n$) which is valid regardless of whether the ramification index $e_V$ is smaller than $p$ or not. The proof uses different techniques than the ones used above.

Consider a deformation $\EE$ over $R$ where $R = V/\germ_V^n$. The Hodge filtration
\[ 0 \arr H^0(\EE, \Omega^1_{\EE/R}) \arr \HH^1_{\rm dR}(\EE/R), \]
is stable under $\End(\EE)$ and so there is a resulting ring homomorphism
\[ \varphi: \End(\EE) \arr T(R), \]
where $T(R)$ are the upper triangular matrices with entries in $R$,
\[T(R) = \left\{ \begin{pmatrix} a & b \\ 0 & d\end{pmatrix} : a, b, d \in R\right\}.\] 
Let $\calO' = \End(\EE) \otimes_\ZZ\ZZ_p$. As we have proved above, $[\End(E): \End(\EE)] = [\calO: \calO']$, where $\calO$ is the maximal order of $B_{p, \infty} \otimes_\QQ \QQ_p$ obtained as the $p$-completion of $\End(E)$.

There is an induced ring homomorphism
\[ \varphi: \calO' \arr T(R).\]
Let $K = \Ker(\varphi)$, let $I(R)= \left\{ \begin{pmatrix} 0 & b \\ 0 & 0\end{pmatrix} : a, b, d \in R\right\}$, and let $P = \varphi^{-1}(I(R))$. Note that $I(R)$ is the kernel of the ring homomorphism
\[ T(R) \arr R \oplus R, \qquad  \left\{ \begin{pmatrix} a & b \\ 0 & d\end{pmatrix} : a, b, d \in R\right\} \mapsto (a, d).\]
It follows that $I(R)$ is a two-sided ideal, such that $T(R)/I(R)$ is a commutative ring. Moreover, $I(R)^2$ = 0. As consequence $P$ is a two sided ideal of $\calO'$ such that $P^2 \subset K$, where $K = \ker (\varphi)$ and $\calO'/P$ is commutative. 

The following lemmas will be proven in the next subsection.

\begin{lem} \label{Lemma: power of p is zero} Let 
\[ \calO_N = \ZZ_p + p^N \calO.\]
Then $\calO_N$ is an order of $\calO$. In the situation above, suppose that $\calO' = \calO_N$, then, in the ring $R$,
\[ p^{4N+2} = 0.\]
\end{lem}

\begin{lem} \label{Lemma: App and Ind}  For an order $\calO'\subseteq \calO$, let
\[ {\rm Ind}(\calO') = \log_p([\calO: \calO']), \qquad {\rm App}(\calO') = \min \{ N: \calO ' \supseteq \calO_N\}.\]
(${\rm Ind}$ is for index and ${\rm App}$ is for approximation.) Then,
\[ \App(\calO') \leq \Ind(\calO'). \]
\end{lem}

\

\id Assume the Lemmas. Given an order $\calO'$, we have $\calO' \supseteq \calO_N$ where $N = \App(\calO')$ and so the homomorphism $\varphi$ induces a homomorphism $\varphi: \calO_N \arr T(R)$, which implies by Lemma~\ref{Lemma: power of p is zero} that $p^{4N+2} = 0$. Since the minimal power of $p$ which is zero in $R$ is $\lceil n/e_V \rceil$ we conclude that $(4\cdot \App(\calO')+2) \geq  \lceil n/e_V \rceil$. Combining it with Lemma~\ref{Lemma: App and Ind}, we find that $\Ind(\calO') \geq \frac{1}{4}(\lceil n/e_V \rceil - 2)$. To summarize, we have proven the following theorem.

\begin{thm} \label{thm: bounds on indices of endo rings under deformation and ramification} With the above notation,
\[ p^{\frac{1}{4}(\lceil n/e_V \rceil - 2)} \leq \geri(V/\germ_V^n).\]
\end{thm}

\subsubsection{Proof of the Lemmas} We use the presentation for the maximal order $\calO$ given above,
\[\calO = \left\{ \begin{pmatrix} a & b \\ pb^\sigma & a^\sigma\end{pmatrix}: a, b \in W(\FF_p^2) \right\}. \]

Consider the situation where $\calO' = \calO_N = \ZZ_p + p^N \calO$. We have a homomorphism $\varphi: \calO_N \arr T(R)$ with kernel $K$ and the ideal $P = \varphi^{-1}(I(R))$. As we have noted $P^2 \subseteq K$. 
Let $[x, y]:= xy - yx$. 
Since $\calO_N/P$ is commutative, we must have $[x,y] \in P$ for all $x, y \in \calO_n$, and so $[x,y]^2 \in K$ for all $x, y \in \calO_N$. Consider the elements $x = p^N \begin{pmatrix}  & 1 \\ p &\end{pmatrix}, y = p^N\begin{pmatrix} & t \\ pt^\sigma \end{pmatrix}$, where for $p\neq 2$ we choose $t$ to be a unit in $W(\FF_{p^2})$ such that $t^\sigma = -t$ and for $t=2$ we choose $t\in W(\FF_{p^2})$ such that $t^2 + t-1 = 0$. In both cases $t - t^\sigma$ is a unit whose square is a unit in $\ZZ_p$, hence in $\calO_N$. Now,
\[ [x, y] = p^{2N+1} \begin{pmatrix} t^\sigma - t & \\ & t - t^\sigma \end{pmatrix}.\]
We conclude that $p^{4N+2} \begin{pmatrix} (t^\sigma - t)^2 & \\ & (t - t^\sigma)^2 \end{pmatrix}\in K$ and so that $p^{4N+2} = 0$ in $R$. Lemma~\ref{Lemma: power of p is zero} follows.

\

\id Lemma~\ref{Lemma: App and Ind} is in fact trivial. The abelian group $\calO/\calO'$ has order $p^{\Ind(\calO')}$ and thus, if $a\in \calO$ then $p^{\Ind(\calO')}\cdot a  = 0$ in $\calO/\calO'$, namely, $p^{\Ind(\calO')}\cdot \calO \subseteq \calO'$.

\subsubsection{Scholium}
One may ask if the bound in Lemma~\ref{Lemma: App and Ind} can be improved. The answer to that is \emph{no}. 
The reader is referred to the paper by Brzezinski \cite{Brzezkinski}. In particular, in Proposition (5.6) of that paper we find the classification of all Gorenstein orders in $\calO$. Examination of the classification shows that our Lemma cannot be improved; More precisely, in cases (a), (b) and (c$_{1}$) one actually finds that $\App(\calO') \leq \lceil \Ind(\calO')/2 \rceil$ (and the passage to non-Gorenstein order is not a problem using Proposition (1.4) of that paper), but this does not persist in case (c$_{2}$).


\

\

\section{The main theorem}\label{section: main theorem}

\id Let $K$ be a primitive CM field of degree four over $\QQ$. 
Let $K^+ = \QQ(\sqrt{d})$ where $d$ is a positive square-free integer. Write 
\[K = \QQ(\sqrt{d})(\sqrt{r}), \qquad r = \alpha + \beta\sqrt{d} \ll 0, \qquad \alpha, \beta \in \ZZ.\]
(That is, $r$ is negative under both embeddings of $K^+$ into $\RR$.)

Let $\tau \in \Symp(4, \ZZ) \backslash \gerH_2$ be a point such that the associated principally polarized abelian variety $A_\tau$ has CM by $\calO_K$. Let $L = N H_{K^\ast}$, where $N$ is the normal closure of $K$ over $\QQ$ and let $\gerp_L$ be a prime of $L$ above the rational prime $p$. We fix the notation as in \S~\ref{section: reduction of abelian surfaces with complex multiplication}. In particular the CM type is $\Phi$ as given there and we have prime ideals $\gerp_{N, 1} = \gerp_L \cap N, \gerp_{K, 1} = \gerp_L \cap K, \gerp_{K^\ast, 1} = \gerp_L \cap K^\ast$ and~$p$, corresponding to the fields in the diagram:
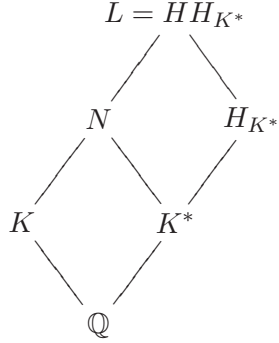
\[ \xymatrix@!C=5pt{& & L = HH_{K^\ast}\ar@{-}[dr]\ar@{-}[dl] & \\ & N\ar@{-}[dr]\ar@{-}[dl] & & H_{K^\ast}\ar@{-}[dl] \\ K\ar@{-}[dr] & & K^\ast\ar@{-}[dl] \\ &  \QQ & & 
}  \]
Let $e = e(\gerp_{N,1}/p)$ be the ramification index of $\gerp_{N, 1}$ over $p$.

\begin{thm} \label{theorem: main theorem} Let $\tau$ be a CM point, as above. Let $f = g/\Theta^k$ be a modular function of level one on $\gerH_2$ where:
\begin{enumerate}
\item $\Theta$ is Igusa's $4\cdot \chi_{10}$, the product of the squares of the ten Riemann theta functions with even integral chracteristics, normalized to have Fourier coefficients that are relatively prime integers.
\item $g$ is a level one modular form of weight $10 k$ with relatively prime integral Fourier coefficients.
\end{enumerate}
Then $f(\tau) \in L= NH_{K^\ast}$ and 
\begin{equation} \label{equation: bound on denominators}
\val_{\gerp_L}( f(\tau)) \geq \begin{cases}  
-4ke\left(\log_p\left(\frac{d \Tr(r)^2}{2}\right)+1\right) & e \leq p-1 \\
-4ke\left(8\log_p\left(\frac{d \Tr(r)^2}{2}\right)+2\right) & \text{else.}
\end{cases}
\end{equation}
Furthermore, unless we are in the situation of superspecial reduction, namely, unless we have a check mark in the last column of the tables in \S~\ref{section: reduction of abelian surfaces with complex multiplication}, $\val_{\gerp_L}( f(\tau)) \geq  0$. The valuation is normalized so that a uniformizer at $\gerp_L$ has valuation $1$.
\end{thm}
\begin{proof} Let $v = \val_{\gerp_L}( f(\tau))$. We may assume that $v < 0$. To conceptualize the proof, we divide it into steps.

\id {\bf Step 1: Adding level structure.} Let $N\geq 3$ be an integer prime to $p$. We abuse notation and identify $\scrA_{2, N}(\CC)$ with $\Gamma(N) \backslash \gerH_2$, where $\Gamma(N) \subseteq \Symp(4, \ZZ)$ is the principal congruence subgroup of matrices congruent to $1$ modulo $N$. Let $\tau_N\in \scrA$ such that
\[ \pi_N(\tau_N) = \tau,\]
where $\pi_N \colon \scrA_{2, N} \arr \scrA_{2, 1}$ is the natural projection. The point $\tau_N$ is defined over the field $\gerp_{L_N}$, where $L_N = L(A_\tau[N])$ is the field obtained from $L$ by adjoining the $N$-torsion points of $A_\tau$. The extension $L_N/L$ is unramified at $p$ (c.f. proof of Lemma~\ref{Lemma: existence of models with good reduction}). We let $\gerp_{L_N}$ be a prime of $L_N$ such that $\gerp_{L_N} \cap L = \gerp_{L}$.

\begin{lem} Let $f_N = f \circ \pi_N$. Then,
\[ \val_{\gerp_{L_N}}(f_N(\tau_N)) = \val_{\gerp_L}(f(\tau)).\]
\end{lem}
\begin{proof} This is clear: $f_N(\tau_N) = f(\tau)$ and the extension $L_N/L$ is unramified at $\gerp_L$. \end{proof}
It is therefore enough to prove the same bound given in (\ref{equation: bound on denominators}) but for $\val_{\gerp_{L_N}}(f_N(\tau_N))$.

\bigskip

\id {\bf Step 2: Reducing to a geometric problem.}

\begin{lem}
Let $(f_N)$ be the divisor of $f_N$ on $\scrA_{2, N}$. Let $(f)_\infty$ be its polar part. Let $\scrH_{1, N}$ be the Humbert divisor of invariant $1$ on $\scrA_{2, N}$. Then,
\[(f)_\infty = 4k\cdot \scrH_{1, N}.\]
\end{lem}

\begin{proof} It is well-known that $\Theta $ vanishes to order $2$ on $\scrH_{1, 1}$. The Lemma then follows immediately from Lemma~\ref{Lemma: ramification along Humbert surfaces}.
\end{proof}

By Lemma~\ref{Lemma: existence of models with good reduction}, the abelian variety $A_{\tau_N}$ has good reduction at $\gerp_{L_N}$. Let $\Lambda$ be the ring of integers of $\widetilde{L}_N$ (the completion of $L_N$ at $\gerp_{L_N}$) and $\gerP$ its maximal ideal. Then there is a morphism 
\[ \iota: \Lambda \arr \scrA_{2, N}, \]
corresponding to $A_{\tau_N}$. 

\begin{prop}
Let $A:=A_{\tau_N}$. There is an unramified field extension $M$ of $\widetilde{L}_N$ of degree at most $2$, with ring of integers $V$, such that 
\[ A \otimes (V/\germ_V^w) \cong \EE \times \EE', \]
as polarized abelian varieties, where $\EE, \EE'$ are elliptic curves over $V/\germ_V^w$, and where \[w = \lceil -v/4k \rceil.\]

\end{prop}
\begin{proof} By Lemma~\ref{Lemma: valuation and reduction}, applied to $1/f$, the morphism $\iota$ induces a morphism
\[ \iota: \Lambda/\gerP^w \arr \scrH_{1, N}. \]
In the notation of Proposition~\ref{Proposition: BN and A2N}, we have $\scrH_{1, N} = \beta(\scrB_N)$. 

Consider the cartesian diagram 
\[ \xymatrix@!C=20pt@!R=10pt{S  \ar[rr]\ar[dd] && \scrB_N\ar[dd]^\beta \\ & \square & \\ \Spec(\Lambda/\gerP^w) \ar[rr] & & \scrH_{1, N} }\]
Since $\beta\colon \scrB_N \arr \scrH_{1, N}$ is \'etale of degree $2$, the morphism $S \arr  \Spec(\Lambda/\gerP^w)$ is \'etale and affine, and so $S$ is an affine scheme, possibly disconnected. We can then choose a field $M$, as in the statement of the proposition, such that $\Spec(V/\germ_V^w)$ is equal to $S$ (or one of its connected components). We therefore get a point
\[ \Spec(V/\germ_V^w) \arr \scrB_N, \]
lifting $\iota$, and that means precisely that $A \otimes V/\germ_V^w$ is isomorphic, as a polarized abelian variety with level $N$ structure, to a product of elliptic curves over $V/\germ_V^w$, with the natural product polarization and some level $N$ structure.  
\end{proof}

\id Note that $\val_{\germ_V}(f_N(\tau_N)) = \val_{\gerp_{L_N}}(f_N(\tau_N))$ 
and so it is enough to show that (\ref{equation: bound on denominators}) holds for $\val_{\germ_V}(f_N(\tau_N)) $. Let us reset our notation and recall that at this point we have a principally polarized abelian surface $A = A_{\tau_N}\otimes V$ with CM by $\calO_K$, having good reduction at $\germ_V$ and such that 
\[ A \otimes V/\germ_V^w \cong (\EE \times \EE', \lambda_1 \times \lambda_2), \]
where $\EE, \EE'$, are elliptic curves over $V/\germ_V^w$. Recall also that $V$ is an unramified extension of the completion of $L = NH_{K^\ast}$ at the prime $\gerp_L$.

\bigskip

\id {\bf Step 3: Reduction to a statement about $\End(\EE)$}.
Our notation for the field $K = \QQ(\sqrt{d})(\sqrt{r})$ is precisely as in \cite{GL}. As in loc. cit., one argues that $\EE$ and $\EE'$ have supersingular reduction, denoted $E, E'$, respectively. One writes
\[ \sqrt{d} \mapsto \begin{pmatrix} a & b \\ b^\vee & -a\end{pmatrix}, \qquad
\sqrt{r} \mapsto \begin{pmatrix}x & y \\ -y^\vee & w \end{pmatrix}, \]
as elements of 
\[ \Hom(\EE\times \EE') = \begin{pmatrix} \End(\EE) & \Hom(\EE', \EE) \\
\Hom(\EE, \EE') & \End(\EE') \end{pmatrix}.\]
(We are using $\vee$ to denote the dual isogeny.)
Note that $b\in \Hom(\EE', \EE)$ is an isogeny of degree $bb^\vee \leq d$. Using $b$, we may view $\End(\EE \times \EE')$ as a subring of $M_2(\End^0(\EE))$ by 
\[  \begin{pmatrix} 1 &  \\  & b^{\vee, -1} \end{pmatrix}
 \begin{pmatrix}  \varphi_{ij}\end{pmatrix}
  \begin{pmatrix} 1 &  \\  & b^\vee \end{pmatrix}.\]
Appying this to the matrices defining $\sqrt{d}, \sqrt{r}$, we find the matrices 
\[ \begin{pmatrix} a & bb^\vee \\ 1 & -a \end{pmatrix}, \qquad \begin{pmatrix} x &  yb^\vee\\ -\frac{1}{bb^\vee} by^\vee & \frac{1}{bb^\vee}bwb^\vee \end{pmatrix}.  \]
As in \cite{GL}, the integral(!) elements $1, x, yb^\vee, xyb^\vee$ must be linearly independent over $\ZZ$ (one shows that otherwise they generate a quadratic imaginary subfield $K_1$ of $B_{p, \infty}$ such that we have $K \injects M_2(K_1)$, leading to a contraction). 
As in \cite[p. 464]{GL}, one finds that the norms of these elements are bounded, respectively, by
\[1, \delta_2, d\delta_1, d\delta_1\delta_2, \]
where
\[  \delta_1 = \vert \alpha \vert - \vert \beta\vert \cdot \vert a \vert, \qquad  \delta_2 = \vert \alpha \vert + \vert \beta\vert \cdot \vert a \vert .\]
It follows that 
\[ [\End(E): \End(\EE)] \leq [\End(E): \ZZ[1, x, yb^\vee, xyb^\vee]] \leq 4d^2(\delta_1\delta_2)^2.
\]
(Cf. \cite[p. 460]{GL} for the last inequality.)

\bigskip

\id {\bf Step 4: Input from deformation theory.}
We now utilize the results of \S~\ref{section: deformation theory} to bound the index $[\End(E): \End(\EE)]$ from below. Recall that $\EE$ is an elliptic curve over $V/\germ_V^w$ and $V$ is an unramified extension of the completion of $L =N H_{K^\ast}$, hence of $N$ completed at the prime $\gerp_{N, 1}$. Thus, $e_V$ -- the absolute ramification index of $V$ -- is equal to $e = e(\gerp_{N, 1}/p)$.

\begin{enumerate}

\item {\bf Small ramification.} Suppose that $e\leq p-1$. By Theorem \ref{thm: bounds on indices of endo rings under deformation},
\[ [\End(E): \End(\EE)] \geq p^{2(\lceil w/e \rceil - 1)}, \]
and so $2(\lceil w/e \rceil - 1) \leq \log_p(4d^2(\delta_1\delta_2)^2)$. Since $\delta_1\delta_2  = \alpha^2 - \beta^2a^2 \leq \alpha^2 = \frac{1}{4}(\Tr(r))^2$, we find that $w/e \leq \lceil w/e \rceil \leq \frac{1}{2}  \log_p(4d^2(\delta_1\delta_2)^2) + 1 \leq \log_p(\frac{d\cdot \Tr(r)^2}{2}) +1$. Since $w = \lceil -v/4k\rceil$, it follows that $-v \leq 4k w \leq 4k e \left[\log_p(\frac{d\cdot \Tr(r)^2}{2}) +1 \right]$.
\item {\bf High ramification.} Suppose that $e>  p-1$.  By Theorem \ref{thm: bounds on indices of endo rings under deformation and ramification} ,
\[ [\End(E): \End(\EE)] \geq p^{\frac{1}{4}(\lceil w/e \rceil - 2)}. \]
Similar computations yield $-v \leq 4k e \left[8\log_p(\frac{d\cdot \Tr(r)^2}{2}) +2 \right]$.
\end{enumerate}
\end{proof}

\subsection{Factorization of class invariants and denominators of Igusa class polynomials} 
We derive several consequences of Theorem \ref{theorem: main theorem}.

\begin{cor} \label{cor: denominators}
Let $K$ be a quartic primitive CM field, as in the beginning of \S~\ref{section: main theorem} and let $\gerh_i(x)$, $i=1, 2, 3$, be the class polynomial defined using the function $f_i/\Theta^{k(i)}$ as in \S \ref{section: Igusa class polynomials}, equation (\ref{equation for gerh polynomials}), where $k(i) = 6, 4, 4$ for $i=1, 2, 3$, respectively. In the notation of Theorem \ref{theorem: main theorem}, the coefficient of $x^{\deg(h_i) - a}$ in $h_i(x)$, which is a rational number, has valuation $\val_p$ greater or equal to 
\[ 
\begin{cases}  
-4a\cdot k(i)\left(\log_p\left(\frac{d \Tr(r)^2}{2}\right)+1\right) & e \leq p-1 \\
-4a\cdot  k(i)\left(8\log_p\left(\frac{d \Tr(r)^2}{2}\right)+2\right) & \text{else.}
\end{cases}
\]
 
\end{cor}
\begin{proof} Straightforward from Theorem \ref{theorem: main theorem}.
\end{proof}

\begin{rmk} We remark that this corollary is crucial in bounding the complexity of construction of CM curves of genus $2$, by the methods currently used. The Corollary is proven for the invariants that we find convenient; with little effort one can deduce easily such bound for the Igusa class polynomials appearing in equation \ref{class polynomials 2}, which are often used in the literature. Further, we could have equally proven the Corollary for class polynomials formed out of the Igusa coordinates $\gamma_i$ (see \S  \ref{section: absolute invariants}). In principle, this is ``the right thing to do", on the other hand, given Proposition~\ref{prop: failure of igusa invariants}, in practice it suffices to deal only with (some set of) the absolute Igusa invariants.
\end{rmk}
\begin{cor} \label{cor: class invariants}
Let $u(\Phi; \gera, \gerb)$ be the class invariant defined in \cite{DeShalitGoren}, associated to fractional ideals $\gera, \gerb$ of $K$. Let $\gerp_L$ be a prime of $L$, as in Theorem \ref{theorem: main theorem} and $\gerp_{H_{K^\ast}} = \gerp_L \cap H_{K^\ast}$. We note that $u(\Phi; \gera, \gerb) \in H_{K^\ast} \subseteq L$. We have
\[ \vert \val_{\gerp_{H_{K^\ast}}}(u(\Phi; \gera, \gerb)) \vert \leq \begin{cases}  
8 e^\ast\left(\log_p\left(\frac{d \Tr(r)^2}{2}\right)+1\right) & e \leq p-1 \\
8 e^\ast\left(8\log_p\left(\frac{d \Tr(r)^2}{2}\right)+2\right) & \text{else,}
\end{cases}
\]
where $e^\ast$ is the ramification index of $\gerp_{K^\ast} = \gerp_{K^\ast, 1}$ over $p$.
\end{cor}
\begin{proof} 
We refer to  \cite{DeShalitGoren} for the detailed definitions. We have
\[ u(\Phi, \gera) = \frac{\Theta(\Phi(\gera^{-1}))}{\Theta(\Phi(\calO_K))},\]
which may also be written as 
\[ u(\Phi, \gera) = \left(\frac{\Theta\vert_\gamma}{\Theta}\right)(\tau), \]
where $\tau$ is a period matrix for $\Phi(\calO_K)$ and, for $\gera^{-1}$ an integral ideal, $\gamma\in \Symp(4, \QQ)$ is a matrix with integral entries and 
determinant $\Norm(\gera)$ (cf. \cite{DeShalitGoren} p. 786 and \S 3.2). We remark that we may also write 
\[ u(\Phi, \gera) =  \left(\frac{\Theta}{\Theta\vert_{\gamma^{-1}}} \right)(\tau'), \]
where $\tau'$ is a period matrix corresponding to $\gera^{-1}$. 

Now fix a prime ideal $\gerP$ of $\overline\QQ$ above the rational prime $p$. Assume $\gera^{-1}$ is an integral ideal of norm $N\geq 3$, which is relatively prime to $\gerP$. We note that both $\frac{\Theta\vert_\gamma}{\Theta}$ and $\frac{\Theta}{\Theta\vert_{\gamma^{-1}}}$ are modular functions of level $N$, defined over $\QQ(\zeta_N)$, and $ u(\Phi, \gera)$ is obtained by evaluating them at a point with CM by $\calO_K$. We can therefore apply Theorem \ref{theorem: main theorem}, or, more precisely, the result we have obtained in its proof by passing to level $N$. We consider both $\left(\frac{\Theta\vert_\gamma}{\Theta} \right)(\tau)$ and $\left(\frac{\Theta}{\Theta\vert_{\gamma^{-1}}} \right)(\tau')$ to get from one a bound on the denominator of $u(\Phi, \gera)$ at $\gerP$ and, from the other, a bound on the numerator. The points $\tau, \tau'$ correspond to abelian varieties with CM by $\calO_K$ defined over the compositum $L'$ of $L$ and $\QQ(\zeta_N)$, which does not increase the ramification index $e$ of $p$ at $\gerp_{L} = \gerP \cap L$. We may then consider the valuation at $\gerp_{L'} = \gerP \cap L'$. We conclude that 
\[ \vert \val_{\gerp_{L'}}(u(\Phi, \gera))\vert \leq 
\begin{cases}  
4e\left(\log_p\left(\frac{d \Tr(r)^2}{2}\right)+1\right) & e \leq p-1 \\
4e\left(8\log_p\left(\frac{d \Tr(r)^2}{2}\right)+2\right) & \text{else.}
\end{cases}
\]
However, the algebraic number $u(\Phi, \gera)$ actually lies in $H_{K^\ast}$ and so we get
\[ \vert \val_{\gerp_{H_{K^\ast}}}(u(\Phi, \gera)) \vert \leq 
\begin{cases}  
4e^\ast\left(\log_p\left(\frac{d \Tr(r)^2}{2}\right)+1\right) & e \leq p-1 \\
4e^\ast\left(8\log_p\left(\frac{d \Tr(r)^2}{2}\right)+2\right) & \text{else,}
\end{cases}
\]
where $e^\ast = e(\gerp_{K^\ast}/p)$.

Let us now consider $u(\Phi; \gera, \gerb)$. The class invariant $u(\Phi; \gera, \gerb)$ depends only on the ideal class of $\gera$ and $\gerb$ in the class group of $K$. Having fixed $\gerP$, we may assume therefore that $\gera^{-1}, \gerb^{-1}$ are integral and of norm prime to $p$. We note the expressions:
\[ u(\Phi; \gera, \gerb) = \frac{u(\Phi, \gera\gerb)}{u(\Phi, \gera)u(\Phi, \gerb)} = \frac{\Theta(\Phi(\gera^{-1}\gerb^{-1})) \Theta(\Phi(\calO_K))}{\Theta(\Phi(\gera^{-1})) \Theta(\Phi(\gerb^{-1}))}.\]
Instead of using directly our bound above, we note that
for $\gera^{-1}, \gerb^{-1}$ integral ideals, we may write
\[  u(\Phi; \gera, \gerb) =  \left(\frac{\Theta\vert_\gamma}{\Theta}\right)(\tau') /  \left(\frac{\Theta\vert_\beta}{\Theta} \right)(\tau),
\]
where $\tau$ is a period matrix for $\Phi(\calO_K)$, $\tau'$ is a period matrix for $\Phi(\gera^{-1})$, $\beta, \gamma\in \Symp(4, \QQ)$ are matrices with integral entries and 
determinants prime to $p$. Thus, repeating the consideration above, we conclude the bound in the corollary.
\end{proof}

\

\


\end{document}